\documentclass[12pt]{amsart}
\pdfoutput=1
\usepackage{
    amsmath,
    amsfonts,
    amssymb,
    amsthm,
    amscd,
    comment,
    enumitem,
    etoolbox,
    gensymb,    % For \degree
    mathtools,
    mathdots,
    stmaryrd,
}
\usepackage{pdfpages}
\usepackage{appendix}
\usepackage{booktabs,tabularx}
\usepackage{makecell}

\usepackage{xcolor}
\usepackage[all]{xy}

\usepackage[T1]{fontenc}
\usepackage{bbm}                     % For \mathbbm{1} (unit in monoidal category)
\usepackage[colorlinks=true, linkcolor=blue, citecolor=blue, urlcolor=blue, breaklinks=true]{hyperref}

\DeclareFontFamily{OT1}{pzc}{}
\DeclareFontShape{OT1}{pzc}{m}{it}{<-> s * [1.10] pzcmi7t}{}
\DeclareMathAlphabet{\mathpzc}{OT1}{pzc}{m}{it}

\usepackage[capitalize]{cleveref}   % Can add 'nameinlink' option to have the name included in hyperlink

\crefname{defin}{Definition}{Definitions}
\crefname{eg}{Example}{Examples}
\crefname{egs}{Example}{Examples}
\crefname{lem}{Lemma}{Lemmas}
\crefname{theo}{Theorem}{Theorems}
\crefname{equation}{}{}
\crefname{enumi}{}{}
\renewcommand{\Im}{\mathrm{Im}}

\newcommand{\quarter}{\frac{1}{4}}
\newcommand{\half}{\frac{1}{2}}

\newcommand\C{\mathbb{C}}
\newcommand\N{\mathbb{N}}
\newcommand\OO{\mathbb{O}}

\newcommand\R{\mathbb{R}}

\newcommand\kk{\Bbbk}
\newcommand\one{\mathbbm{1}}

\newcommand\B{\mathbf{B}}

\newcommand\fg{\mathfrak{g}}
\newcommand{\fsl}{\mathfrak{sl}_2(\kk)}

\newcommand\md{\textup{-mod}}

\newcommand\cC{\mathcal{C}}

\newcommand\Fcat{\mathpzc{EL}}           % Diagrammatic F_4 category
\newcommand\Tcat{\mathpzc{T}}           % Trivalent category
\newcommand\TL{\mathpzc{TL}}            % Temperley-Lieb category

\newcommand\cI{\mathcal{I}}             % Tensor ideal

\newcommand\go{{\mathsf{I}}}            % Self-dual generating object

\DeclareMathOperator{\Add}{Add}
\DeclareMathOperator{\End}{End}

\DeclareMathOperator{\Hom}{Hom}
\DeclareMathOperator{\id}{id}
\DeclareMathOperator{\Kar}{Kar}

\DeclareMathOperator{\Ob}{Ob}

\DeclareMathOperator{\Rot}{Rot}
\DeclareMathOperator{\Switch}{Switch}

\DeclareMathOperator{\tr}{tr}
\DeclareMathOperator{\ad}{ad}

\usepackage{tikz}
\usetikzlibrary{arrows.meta}
\usetikzlibrary{decorations.markings}
\usetikzlibrary{calc}

\tikzset{anchorbase/.style={>=To,baseline={([yshift=-0.5ex]current bounding box.center)}}}
\tikzset{ % Syntax: \begin{tikzpicture}[centerzero={0,0.2}]
    centerzero/.style={>=To,baseline={([yshift=-0.5ex](#1))}},
    centerzero/.default={0,0}
}
\tikzset{wipe/.style={white,line width=3pt}}

\newcommand\braiddown{to[out=down,in=up]}

\newcommand\bub[1]{% \bub{position}
    \draw (#1)++(0,0.2) arc(90:-270:0.2)
}

\newcommand\bubble{% \cbubble{token}{dot}
    \begin{tikzpicture}[centerzero]
        \bub{0,0};
    \end{tikzpicture}
}
\newcommand\idstrand[1][a]{
    \begin{tikzpicture}[centerzero]
        \draw (0,-0.2) -- (0,0.2);
    \end{tikzpicture}
}
\newcommand\crossmor{
    \begin{tikzpicture}[centerzero]
        \draw (0.2,-0.2) -- (-0.2,0.2);
        \draw (-0.2,-0.2) -- (0.2,0.2);
    \end{tikzpicture}
}
\newcommand{\cupmor}{
    \begin{tikzpicture}[anchorbase]
        \draw (-0.15,0.15) -- (-0.15,0) arc(180:360:0.15) -- (0.15,0.15);
    \end{tikzpicture}
}
\newcommand{\capmor}{
    \begin{tikzpicture}[anchorbase]
        \draw (-0.15,-0.15) -- (-0.15,0) arc(180:0:0.15) -- (0.15,-0.15);
    \end{tikzpicture}
}
\newcommand\mergemor{
    \begin{tikzpicture}[centerzero]
      \draw (-0.2,-0.2) to (0,0);
      \draw (0.2,-0.2) to (0,0);
      \draw (0,0) to (0,0.2);
    \end{tikzpicture}
}
\newcommand\splitmor{
    \begin{tikzpicture}[centerzero]
      \draw (-0.2,0.2) to (0,0);
      \draw (0.2,0.2) to (0,0);
      \draw (0,0) to (0,-0.2);
    \end{tikzpicture}
}

\newcommand\lollydrop{
    \begin{tikzpicture}[centerzero]
        \draw (0,0.2) -- (0,0) to[out=-40,in=up] (0.15,-0.15) arc(0:-180:0.15) to[out=up,in=-130] (0,0);
    \end{tikzpicture}
}
\newcommand\dotcross{
\begin{tikzpicture}[centerzero]
	\draw (-.2,-.2) -- (.2,.2);
	\draw (.2,-.2) -- (-.2,.2);
	\draw[fill] (0,0) -- (-.08,.08) -- (-.08,-.08);
\end{tikzpicture}
}
\newcommand\triform{
    \begin{tikzpicture}[centerzero]
        \draw (-0.3,-0.2) -- (0,0.2);
        \draw (0,-0.2) -- (0,0.2);
        \draw (0.3,-0.2) -- (0,0.2);
    \end{tikzpicture}
}
\newcommand\explode{
    \begin{tikzpicture}[centerzero]
        \draw (-0.3,0.2) -- (0,-0.2);
        \draw (0,0.2) -- (0,-0.2);
        \draw (0.3,0.2) -- (0,-0.2);
    \end{tikzpicture}
}
\newcommand\trimor{
    \begin{tikzpicture}[anchorbase]
        \draw (0,0.15) -- (0.13,-0.075) -- (-0.13,-0.075) -- cycle;
        \draw (0,0.15) -- (0,0.3);
        \draw (0.13,-0.075) -- (0.24,-0.185);
        \draw (-0.13,-0.075) -- (-0.24,-0.185);
    \end{tikzpicture}
}

\newcommand\sqmor{
    \begin{tikzpicture}[centerzero]
        \draw (-0.15,-0.15) rectangle (0.15,0.15);
        \draw (-0.3,-0.3) -- (-0.15,-0.15);
        \draw (0.3,-0.3) -- (0.15,-0.15);
        \draw (-0.3,0.3) -- (-0.15,0.15);
        \draw (0.3,0.3) -- (0.15,0.15);
    \end{tikzpicture}
}
\newcommand\pentmor{
    \begin{tikzpicture}[centerzero]
        \draw (0,0.25) -- (-0.24,0.08) -- (-0.15,-0.20) -- (0.15,-0.20) -- (0.24,0.08) -- cycle;
        \draw (0,0.25) -- (0,0.4);
        \draw (-0.24,0.08) -- (-0.35,0.4);
        \draw (0.24,0.08) -- (0.35,0.4);
        \draw (-0.15,-0.2) -- (-0.35,-0.4);
        \draw (0.15,-0.2) -- (0.35,-0.4);
    \end{tikzpicture}
}
\newcommand\Hmor{
    \begin{tikzpicture}[centerzero]
        \draw (-0.3,-0.2) -- (-0.1,0) -- (-0.3,0.2);
        \draw (-0.1,0) -- (0.1,0);
        \draw (0.3,-0.2) -- (0.1,0) -- (0.3,0.2);
    \end{tikzpicture}
}
\newcommand\Imor{
    \begin{tikzpicture}[centerzero]
        \draw (-0.2,-0.3) -- (0,-0.1) -- (0.2,-0.3);
        \draw (-0.2,0.3) -- (0,0.1) -- (0.2,0.3);
        \draw (0,-0.1) -- (0,0.1);
    \end{tikzpicture}
}
\newcommand\hourglass{
    \begin{tikzpicture}[centerzero]
        \draw (-0.15,-0.3) -- (-0.15,-0.25) arc(180:0:0.15) -- (0.15,-0.3);
        \draw (-0.15,0.3) -- (-0.15,0.25) arc(180:360:0.15) -- (0.15,0.3);
    \end{tikzpicture}
}
\newcommand\jail{
    \begin{tikzpicture}[centerzero]
        \draw (-0.15,-0.3) -- (-0.15,0.3);
        \draw (0.15,-0.3) -- (0.15,0.3);
    \end{tikzpicture}
}

\newcommand{\symbox}[2]{
    \filldraw[fill=white,draw=black] (#1) rectangle (#2)
}

\newtheorem{theo}{Theorem}[section]

\newtheorem{prop}[theo]{Proposition}
\newtheorem{lem}[theo]{Lemma}
\newtheorem{cor}[theo]{Corollary}

\theoremstyle{definition}
\newtheorem{defin}[theo]{Definition}
\newtheorem{rem}[theo]{Remark}
\newtheorem{eg}[theo]{Example}

\numberwithin{equation}{section}
\allowdisplaybreaks

\setenumerate[1]{label=(\alph*)}          % Only need to reference enumerate items with \ref{}

\newtoggle{comments}
\newtoggle{details}
\newtoggle{detailsnote}

\iftoggle{comments}{%
  \newcommand{\ycomments}[1]{
    \ \\
    {\color{red}
      \textbf{YM:} #1
    }
    \ \\
    }
}{%
  \newcommand{\ycomments}[1]{}
 }

\iftoggle{details}{%
  \newcommand{\details}[1]{
      \ \\
      {\color{OliveGreen}
        \textbf{Details:} #1
      }
      \\
  }
}{%
  \newcommand{\details}[1]{}
}

\begin{document}
%===============

\title[Trivalent Categories for Adjoint Representations]{Trivalent Categories for Adjoint Representations of  Exceptional Lie Algebras}

\author{Youssef Mousaaid}
\address[Y.M.]{
  Department of Mathematics and Statistics \\
  University of Ottawa \\
  Ottawa, ON K1N 6N5, Canada
}
\email{ymous016@uottawa.ca}

\begin{abstract}
We consider the universal pivotal, symmetric, monoidal, $\kk$-linear category, generated by a Schurian object with a skew-symmetric multiplication, and  study some of its quotients. We show that these quotients give rise to either vector product algebras or representation categories of exceptional Lie algebras.
\end{abstract}



\ifboolexpr{togl{comments} or togl{details}}{%
  {\color{magenta}DETAILS OR COMMENTS ON}
}{%
}

\maketitle
\thispagestyle{empty}

\tableofcontents 

%=====================
\section{Introduction}
%=====================

In this paper we consider (\cref{Tdef}) the universal  pivotal, symmetric, $\kk$-linear, category, $\Tcat_{\delta}$, generated by a Schurian object $\go$ that is equipped with a symmetric bilinear   form   and a skew-symmetric multiplication.
%Our category is identical to the  one in \cite{GSZ}, except that the multiplication in their definition is symmetric instead of being skew-symmetric. 
We aim at classifying quotients of $\Tcat_{\delta}$ with small hom-spaces. Imposing restrictions on the dimensions of hom-spaces forces some relations to hold. This allows us to indirectly prove relations in categories of representations, and to construct functors from quotients of $\Tcat_{\delta}$ into these categories. This technique is used in \cite{AR99,GSZ}, and also in \cite{MPS17}, where similar classification results have been obtained in the context of braided monoidal categories. 

To unwrap the definition let $\kk$ be a field of characteristic not equal to $2$. Given $\delta\in\kk$, we define $\Tcat_\delta$ to be the strict $\kk$-linear monoidal category generated by a single object $\go$ and four morphisms 
\[\mergemor \colon \go \otimes \go \to \go,\quad
	\crossmor \colon \go \otimes \go \to \go \otimes \go,\quad
	\cupmor \colon \one \to \go \otimes \go,\quad
	\capmor \colon \go \otimes \go \to \one,
\]
where $\one$ is the unit object. These  morphisms satisfy certain relations, so that we have the following universal  property. Let $\fg$ be a simple complex finite-dimensional Lie algebra  over a field $\kk$. Let $\kappa$ be the Killing form of $\fg$, and let $\B_\fg$ and $ \{b^\vee\colon b \in \B_\fg\} $ be dual bases of $\fg$. Then for $\delta=\dim_{\kk}\fg$, there is a unique monoidal functor
	\[
	\Phi \colon \Tcat_{\delta} \to \fg\md
	\]
	given on objects by $\go \mapsto \fg$ and on morphisms by
	\begin{align*}
		\Phi(\mergemor) &\colon \fg \otimes \fg \to \fg,&
		a \otimes b &\mapsto [a,b],
		\\
		\Phi(\crossmor) &\colon\fg \otimes \fg \to \fg \otimes \fg,&
		a \otimes b &\mapsto b \otimes a,
		\\
		\Phi(\cupmor) &\colon \kk \to \fg \otimes \fg, &
		1& \mapsto \sum_{b \in \B_\fg} b \otimes b^\vee,
		\\
		\Phi(\capmor) &\colon \fg \otimes \fg \to \kk,&
		a \otimes b &\mapsto   \kappa(a,b).
	\end{align*}
	Furthermore
\[
\Phi \left( \splitmor \right) \colon \fg \to \fg \otimes \fg,\qquad \qquad \qquad \qquad \quad 
		a \mapsto	\sum_{b \in \B_\fg} [a,b] \otimes b^\vee.
\]

 The category $\Tcat_{-2}$ extends the Temperley--Lieb category (\cref{delta}), and so we assume for the remainder of this introduction that  $\delta\neq -2$. We prove that there is no quotient category for which the dimension of $\End(\go^{\otimes  2})$ is either one or two (\cref{tamud}). We then show (\cref{dilum}) that if the endomorphism algebra of $\go^{\otimes  2}$, in some quotient of $\Tcat_{\delta,\alpha}$, is of dimension three, four or five, then $\go$ is either a  vector product algebra object or a Lie algebra object.  In the case when this dimension  is three, we construct (\cref{Kauff}) a functor from a certain quotient of the category $\Tcat_{\delta,\alpha}$ into  the Temperley-Lieb category. The vector product algebra case  arises when the dimension of $\End(\go^{\otimes  2})$ is four (\cref{babel}). The resulting quotient has been considered  in \cite{West}, and is used to give a diagrammatic proof of Hurwitz' theorem. A braided version of this category is found in \cite{Street19}. In the current paper, the case $\dim\End(\go^{\otimes  2})=4$ is related to the category of representations of the  exceptional Lie group $G_2$ (\cref{G2}).

 The situation becomes more involved when the dimension  of $\End(\go^{\otimes  2})$ is five, and we restrict ourselves to the case $\dim\Hom(\one,\go^{\otimes  5})=16$.  The reason we chose these specific numbers is because  of the equations
 \[
 \dim\left(\Tcat_{\delta}/\cI\right)\left(\go^{\otimes 2},\go^{\otimes 2}\right)=5 \quad \text{and}\quad 	\dim\left(\Tcat_{\delta}/\cI\right)\left(\go^{\otimes 2},\go^{\otimes 3}\right)=16,
 \]
  which hold for all five exceptional Lie algebras.  This case is related to Deligne's conjecture \cite{Del,AR99}  about the existence of a universal category for the exceptional Lie algebras. The precise relation is as follows. We define (\cref{Fdef}) a quotient $\Fcat_\delta$ of $\Tcat_{\delta,1}$. For $\fg$ an exceptional Lie algebra of type $G_2,F_4$ or $E_8$, and $\delta=\dim_{\kk}\fg$, we show that there is a full,  essentially surjective monoidal functor (\cref{full}) from $\Fcat_\delta$ onto the category of representations of $\fg$. A full description of the kernel of this functor is still unknown in the case of $F_4$ and $E_8$. The case $G_2$ has been solved by the work of Kuperberg \cite{Kup96}. As a consequence of the fullness of this functor \cref{full}, there are surjective homomorphisms
  \[
 \Fcat_\delta(\go^{\otimes n}) \twoheadrightarrow \End_\fg(V^{\otimes n}),
 \]
  for all $n\in \N$, with $V$ the adjoint representation of $\fg$. Furthermore, the results of \cref{sec4} allow us to give explicit bases of $\Hom_\fg(V^{\otimes 2}, V^{\otimes 3})$ and $\Hom_\fg(V^{\otimes 2},V^{\otimes 2})$. The pivotal structure in $\Tcat_{\delta,\alpha}$ induces a pivotal structure in any quotient $\cC$, and this gives isomorphisms
   \[\cC(\go^{\otimes m}, \go^{\otimes n} )\cong \cC(\one,\go^{\otimes(m+n)}).\]
  We summarize our findings in the following table.  
 
\begin{center}
	  \begin{tabular}{@{}lllr@{}} \toprule
  	\multicolumn{2}{c}{Dimensions} \\ \cmidrule(r){1-2}
  	$ \cC(\one , \go^{\otimes 4}) $ & $ \cC(\one , \go^{\otimes 5}) $ & Example & Reference  \\ \midrule
  	1 or 2  & any & nonexistent & Cor.~\ref{tamud} \\
  	 3 & any  & $\TL(-2)$ &  Prop.~\ref{Kauff} \\
  	4 & any  &  \makecell[l]{Vector product algebras\\ $G_2\md$} & Prop.~\ref{babel},  \ref{G2} \\
  	5 & 16 & $\fg_2,\mathfrak{f}_4,\mathfrak{e}_8\md$& Prop.~\ref{splat}, \ref{full}, \ref{baja} \\ \bottomrule
  \end{tabular}
\end{center}

We plan to investigate  the case where  $\dim \End_{\cC}(\go^{\otimes 2})=5$ and $\dim \Hom_{\cC}(\one,\go^{\otimes 5})=n$, for $n=1,\ldots ,15$, and extend our results to the quantum setting, in a subsequent work. Note that some computations in this paper are performed using SageMath \cite{sagemath}. The SageMath notebook is in the \hyperref[appendix]{appendix} and it can be downloaded from arXiv.
  
%-----------------------------
\subsection*{Acknowledgements}
%-----------------------------
The author would like to thank his supervisor A. Savage for providing many comments for the improvement of this paper. This research was partially supported by his NSERC Discovery Grant RGPIN-2017-03854. 
%==================================
\section{A universal trivalent category}
%==================================

Throughout this paper, we let $\kk$ denote a field of characteristic not equal to $2$. All categories are $\kk$-linear and all algebras and tensor products are over $\kk$ unless otherwise specified.  We let $\one$ denote the unit object of a monoidal category.  For objects $X$ and $Y$ in a category $\cC$, we denote by $\cC(X,Y)$ the vector space of morphisms from $X$ to $Y$.

\begin{defin} \label{Tdef}
    Fix $\delta\in \kk$ and $\alpha\in \kk^{\times}$.  Let $\Tcat_{\delta,\alpha}$ be the strict monoidal category generated by the object $\go$ and generating morphisms
\begin{equation} \label{lego}
        \mergemor \colon \go \otimes \go \to \go,\quad
        \crossmor \colon \go \otimes \go \to \go \otimes \go,\quad
        \cupmor \colon \one \to \go \otimes \go,\quad
        \capmor \colon \go \otimes \go \to \one,
\end{equation}
    subject to the following relations:
\begin{gather} \label{vortex}
        \begin{tikzpicture}[centerzero]
            \draw (-0.3,-0.4) -- (-0.3,0) arc(180:0:0.15) arc(180:360:0.15) -- (0.3,0.4);
        \end{tikzpicture}
        =
        \begin{tikzpicture}[centerzero]
            \draw (0,-0.4) -- (0,0.4);
        \end{tikzpicture}
        =
        \begin{tikzpicture}[centerzero]
            \draw (-0.3,0.4) -- (-0.3,0) arc(180:360:0.15) arc(180:0:0.15) -- (0.3,-0.4);
        \end{tikzpicture}
        \ ,\quad
        \splitmor
        :=
        \begin{tikzpicture}[anchorbase]
            \draw (-0.4,0.2) to[out=down,in=180] (-0.2,-0.2) to[out=0,in=225] (0,0);
            \draw (0,0) -- (0,0.2);
            \draw (0.3,-0.3) -- (0,0);
        \end{tikzpicture}
        =
        \begin{tikzpicture}[anchorbase]
            \draw (0.4,0.2) to[out=down,in=0] (0.2,-0.2) to[out=180,in=-45] (0,0);
            \draw (0,0) -- (0,0.2);
            \draw (-0.3,-0.3) -- (0,0);
        \end{tikzpicture}
        \ ,\quad
        \begin{tikzpicture}[centerzero]
            \draw (-0.2,-0.3) -- (-0.2,-0.1) arc(180:0:0.2) -- (0.2,-0.3);
            \draw (-0.3,0.3) \braiddown (0,-0.3);
        \end{tikzpicture}
        =
        \begin{tikzpicture}[centerzero]
            \draw (-0.2,-0.3) -- (-0.2,-0.1) arc(180:0:0.2) -- (0.2,-0.3);
            \draw (0.3,0.3) \braiddown (0,-0.3);
        \end{tikzpicture}
        \ ,
        \\ \label{venom}
        \begin{tikzpicture}[centerzero]
            \draw (0.2,-0.4) to[out=135,in=down] (-0.15,0) to[out=up,in=-135] (0.2,0.4);
            \draw (-0.2,-0.4) to[out=45,in=down] (0.15,0) to[out=up,in=-45] (-0.2,0.4);
        \end{tikzpicture}
        =
        \begin{tikzpicture}[centerzero]
            \draw (-0.15,-0.4) -- (-0.15,0.4);
            \draw (0.15,-0.4) -- (0.15,0.4);
        \end{tikzpicture}
        \ ,\quad
        \begin{tikzpicture}[centerzero]
            \draw (0.3,-0.4) -- (-0.3,0.4);
            \draw (0,-0.4) to[out=135,in=down] (-0.25,0) to[out=up,in=-135] (0,0.4);
            \draw (-0.3,-0.4) -- (0.3,0.4);
        \end{tikzpicture}
        =
        \begin{tikzpicture}[centerzero]
            \draw (0.3,-0.4) -- (-0.3,0.4);
            \draw (0,-0.4) to[out=45,in=down] (0.25,0) to[out=up,in=-45] (0,0.4);
            \draw (-0.3,-0.4) -- (0.3,0.4);
        \end{tikzpicture}
        \ ,\quad
        \begin{tikzpicture}[anchorbase,scale=0.8]
            \draw (-0.4,-0.5) -- (0.2,0.3) -- (0.4,0.1) -- (0,-0.5);
            \draw (0.2,0.3) -- (0.2,0.6);
            \draw (-0.4,0.6) -- (-0.4,0.1) -- (0.4,-0.5);
        \end{tikzpicture}
        =
        \begin{tikzpicture}[anchorbase]
            \draw (-0.4,-0.4) -- (-0.2,-0.2) -- (-0.2,0) -- (0.2,0.4);
            \draw (0,-0.4) -- (-0.2,-0.2);
            \draw (0.2,-0.4) -- (0.2,0) -- (-0.2,0.4);
        \end{tikzpicture}
        \ ,\quad
        \begin{tikzpicture}[anchorbase,scale=0.8]
            \draw (-0.4,0.5) -- (0.2,-0.3) -- (0.4,-0.1) -- (0,0.5);
            \draw (0.2,-0.3) -- (0.2,-0.6);
            \draw (-0.4,-0.6) -- (-0.4,-0.1) -- (0.4,0.5);
        \end{tikzpicture}
        =
        \begin{tikzpicture}[anchorbase]
            \draw (-0.4,0.4) -- (-0.2,0.2) -- (-0.2,0) -- (0.2,-0.4);
            \draw (0,0.4) -- (-0.2,0.2);
            \draw (0.2,0.4) -- (0.2,0) -- (-0.2,-0.4);
        \end{tikzpicture}
        \ ,
        \\ \label{chess}
        \begin{tikzpicture}[anchorbase]
            \draw (-0.15,-0.4) to[out=45,in=down] (0.15,0) arc(0:180:0.15) to[out=down,in=135] (0.15,-0.4);
        \end{tikzpicture}
        = \capmor
        \ ,\quad
        \begin{tikzpicture}[anchorbase]
            \draw (-0.2,-0.5) to[out=45,in=down] (0.15,-0.2) to[out=up,in=-45] (0,0) -- (0,0.2);
            \draw (0.2,-0.5) to [out=135,in=down] (-0.15,-0.2) to[out=up,in=-135] (0,0);
        \end{tikzpicture}
        = - \mergemor
        \ ,\quad
        \begin{tikzpicture}[centerzero]
            \draw  (0,-0.4) -- (0,-0.2) to[out=45,in=down] (0.15,0) to[out=up,in=-45] (0,0.2) -- (0,0.4);
            \draw (0,-0.2) to[out=135,in=down] (-0.15,0) to[out=up,in=-135] (0,0.2);
        \end{tikzpicture}
        =\alpha \
        \begin{tikzpicture}[centerzero]
            \draw(0,-0.4) -- (0,0.4);
        \end{tikzpicture}
        \ ,\quad
        \bubble = \delta 1_\one.
\end{gather}
\end{defin}
The category $\Tcat_{\delta,\alpha}$ is very similar to the one in \cite[Def.~2.1]{GSZ}. In fact, the defining relations are the same except the second one in \cref{chess}, where their trivalent vertex morphism is symmetric while ours is skew-symmetric. Define
\begin{equation}
    \dotcross
    :=
    \begin{tikzpicture}[centerzero]
        \draw (-0.2,-0.4) -- (-0.2,0.4);
        \draw (-0.2,-0.2) -- (0.4,0.4);
        \draw (0.4,-0.4) -- (-0.2,0.2);
    \end{tikzpicture}
    \ .
\end{equation}

\begin{prop} \label{windy}
    The following relations hold in $\Tcat_{\delta,\alpha}$:
\begin{gather} \label{topsy}
        \begin{tikzpicture}[anchorbase]
            \draw (-0.4,-0.2) to[out=up,in=180] (-0.2,0.2) to[out=0,in=135] (0,0);
            \draw (0,0) -- (0,-0.2);
            \draw (0.3,0.3) -- (0,0);
        \end{tikzpicture}
        =
        \mergemor
        =
        \begin{tikzpicture}[anchorbase]
            \draw (0.4,-0.2) to[out=up,in=0] (0.2,0.2) to[out=180,in=45] (0,0);
            \draw (0,0) -- (0,-0.2);
            \draw (-0.3,0.3) -- (0,0);
        \end{tikzpicture}
        \ ,\quad
        \triform
        :=
        \begin{tikzpicture}[centerzero]
          \draw (-0.2,-0.2) to (0,0);
          \draw (0.2,-0.2) to (0,0);
          \draw (0,0) arc(0:180:0.2) -- (-0.4,-0.2);
        \end{tikzpicture}
        =
        \begin{tikzpicture}[centerzero]
          \draw (-0.2,-0.2) to (0,0);
          \draw (0.2,-0.2) to (0,0);
          \draw (0,0) arc(180:0:0.2) -- (0.4,-0.2);
        \end{tikzpicture}
        \ ,\quad
        \explode
        :=
        \begin{tikzpicture}[centerzero]
          \draw (-0.2,0.2) to (0,0);
          \draw (0.2,0.2) to (0,0);
          \draw (0,0) arc(360:180:0.2) to (-0.4,0.2);
        \end{tikzpicture}
        =
        \begin{tikzpicture}[centerzero]
          \draw (-0.2,0.2) to (0,0);
          \draw (0.2,0.2) to (0,0);
          \draw (0,0) arc(180:360:0.2) to (0.4,0.2);
        \end{tikzpicture}
        \ ,
        \\ \label{turvy}
        \begin{tikzpicture}[centerzero]
            \draw (-0.2,0.3) -- (-0.2,0.1) arc(180:360:0.2) -- (0.2,0.3);
            \draw (-0.3,-0.3) to[out=up,in=down] (0,0.3);
        \end{tikzpicture}
        =
        \begin{tikzpicture}[centerzero]
            \draw (-0.2,0.3) -- (-0.2,0.1) arc(180:360:0.2) -- (0.2,0.3);
            \draw (0.3,-0.3) to[out=up,in=down] (0,0.3);
        \end{tikzpicture}
        \ ,\quad
        \begin{tikzpicture}[anchorbase]
            \draw (-0.2,0.2) -- (0.2,-0.2);
            \draw (-0.4,0.2) to[out=down,in=225,looseness=2] (0,0) to[out=45,in=up,looseness=2] (0.4,-0.2);
        \end{tikzpicture}
        =
        \crossmor
        =
        \begin{tikzpicture}[anchorbase]
            \draw (0.2,0.2) -- (-0.2,-0.2);
            \draw (0.4,0.2) to[out=down,in=-45,looseness=2] (0,0) to[out=135,in=up,looseness=2] (-0.4,-0.2);
        \end{tikzpicture}
        \ ,\quad
        \begin{tikzpicture}[anchorbase]
            \draw (-0.2,0.2) -- (0.2,-0.2);
            \draw (-0.4,0.2) to[out=down,in=225,looseness=2] (0,0) to[out=45,in=up,looseness=2] (0.4,-0.2);
            \draw[fill] (0,0) -- (-0.08,0.08) -- (-0.08,-0.07);
        \end{tikzpicture}
        =
        -\dotcross
        =
        \begin{tikzpicture}[anchorbase]
            \draw (0.2,0.2) -- (-0.2,-0.2);
            \draw (0.4,0.2) to[out=down,in=-45,looseness=2] (0,0) to[out=135,in=up,looseness=2] (-0.4,-0.2);
           \draw[fill] (0,0) -- (-0.08,0.07) -- (-0.08,-0.08);
        \end{tikzpicture}
        \ .
\end{gather}
\end{prop}

\begin{proof}
      These relations are proved in \cite[Prop.~2.4]{GSZ}. Note, however, that the fourth relation in \cref{turvy} does not have a minus in it in \cite{GSZ}. The reason why we have a minus in this relation and they (in \cite[Prop.~2.4]{GSZ}) do not is because we also have a minus in the second relation of \cref{chess} but they do not (see \cite[Def.~2.1]{GSZ}). 
\end{proof}

Relations \cref{vortex,topsy,turvy} imply that the cups and caps equip $\Tcat_{\delta,\delta}$ with the structure of a \emph{strict pivotal} category.  This means that morphisms are invariant under isotopies fixing the top and bottom endpoints.  Thus,  it makes sense to allow horizontal strands in diagrams:
\begin{equation}\label{Hmor}
    \Hmor
    :=
    \begin{tikzpicture}[anchorbase]
        \draw (-0.4,-0.4) -- (-0.4,0) -- (-0.2,0.2) -- (0.2,-0.2) -- (0.4,0) -- (0.4,0.4);
        \draw (-0.2,0.2) -- (-0.2,0.4);
        \draw (0.2,-0.2) -- (0.2,-0.4);
    \end{tikzpicture}
    =
    \begin{tikzpicture}[anchorbase]
        \draw (0.4,-0.4) -- (0.4,0) -- (0.2,0.2) -- (-0.2,-0.2) -- (-0.4,0) -- (-0.4,0.4);
        \draw (0.2,0.2) -- (0.2,0.4);
        \draw (-0.2,-0.2) -- (-0.2,-0.4);
    \end{tikzpicture}
    \ .
\end{equation}
Moreover, it follows from the fact that $\go$ is self-dual that we have isomorphisms
\[
\Tcat_{\delta,\alpha}(\go^{\otimes m},\go^{\otimes n})\cong \Tcat_{\delta,\alpha}(\one,\go^{\otimes (m+n)}),\qquad m,n\in\N.
\]

It follows from  the first and the second relations in \cref{chess}, and the fact that $\Tcat_{\delta,\alpha}$ is pivotal, that
\begin{equation}\label{teardrop}
	\lollydrop = 
	\begin{tikzpicture}[centerzero]
		\draw (0,0.2) -- (0,.4);
		\draw[out=90,in=-40] (.15,0) to (0,.2);
		\draw[out=90,in=-130] (-.15,0) to (0,.2);
		\draw[out=-90,in=90] (-.15,0) to (.15,-.25);
		\draw[out=-90,in=90] (.15,0) to (-.15,-.25);
		\draw[out=-90,in=-90] (-.15,-.25) to (.15,-.25);
	\end{tikzpicture} =
     -\lollydrop 
     = 0
     \; .
\end{equation} 
\begin{rem}
	We can scale $\capmor$ and $\cupmor$, respectively, by $\alpha$ and $ \alpha^{-1}$ to make $ \alpha=1 $ in \cref{chess}. In fact, the categories $ \Tcat_{\delta,\alpha} $ and $\Tcat_{\delta,1}$ are isomorphic, and so we can remove $\alpha$ from the definition all together without losing any generality. However it is sometimes convenient to have the parameter $\alpha$, and so we include it in the definition for that reason.
\end{rem} 

The main result of this section is the construction of a functor from the category $\Tcat_{\delta,1}$, which will simply be denoted by $\Tcat_{\delta}$, to the category of representations of a simple complex Lie algebra of dimension $\delta$. Let $\fg$ be a finite-dimensional simple complex Lie algebra. Let  $[\cdot,\cdot]\colon \fg \times \fg \to \fg $ denote the Lie bracket of $\fg$, and let $\kappa\colon \fg \times \fg \to \kk$ denote its Killing form.  Choose a basis $\B_\fg$ of $\fg$, and let $\B_\fg^{\vee}=\{b^{\vee}\colon b\in \B_\fg\}$ be its dual basis defined by \[\kappa\left(a^\vee,b\right)=\delta_{a,b}, \qquad a,b\in \B_\fg.\]

\begin{lem}
	For $a,b,c\in \fg$, we have
\begin{equation}\label{lava}
		\kappa\left(a,[b,c]\right)=\kappa\left(b,[c,a]\right).
\end{equation} 
\end{lem}

\begin{proof}
	We have
	\[	\tr\left(\ad_{a}\ad_{[b,c]}\right) = \tr\left(\ad_a \left(\ad_b\ad_c-\ad_c\ad_b\right)\right)=\tr\left(\ad_b\left(\ad_c\ad_a-\ad_a\ad_c\right)\right)=\tr\left(\ad_b\ad_{[c,a]}\right). \]
\end{proof}
\begin{lem} For $ a\in \fg$, we have
\begin{equation}\label{magma}
		\sum_{b \in \B_\fg}  [a , b]  \otimes b^\vee 
		= \sum_{b \in \B_\fg} b \otimes  [b^\vee , a] 
\end{equation}
\end{lem}
\begin{proof} 
	We have
\begin{multline*}
		\sum_{b \in \B_\fg}  [a , b]  \otimes b^\vee 
		 =\sum_{b,c\in\B_\fg}  \kappa\left(c^\vee, [a,b]\right)c\otimes b^\vee= \sum_{b,c\in \B_\fg} c\otimes \kappa\left(c^\vee,[a,b]\right)b^\vee\\
		 \overset{\cref{lava}}{=}\sum_{b,c\in \B_\fg}c\otimes \kappa\left(b,[c^\vee,a]\right)b^\vee= \sum_{c\in \B_\fg}c\otimes [c^\vee,a].\qedhere
\end{multline*} 
\end{proof}

Let $\fg\md$ denote the category of finite-dimensional representations of $\fg$. We denote the adjoint representation of $\fg$ by the same symbol $\fg$. The following proposition is the main result of the section. 

\begin{prop}\label{magneto} 
	For $\delta=\dim_{\kk}\fg$, there is a unique monoidal functor
	\[
	\Phi \colon \Tcat_{\delta} \to \fg\md
	\]
	given on objects by $\go \mapsto \fg$ and on morphisms by
\begin{align*}
		\Phi(\mergemor) &\colon \fg \otimes \fg \to \fg,&
		a \otimes b &\mapsto [a,b],
		\\
		\Phi(\crossmor) &\colon\fg \otimes \fg \to \fg \otimes \fg,&
		a \otimes b &\mapsto b \otimes a,
		\\
		\Phi(\cupmor) &\colon \kk \to \fg \otimes \fg, &
		1& \mapsto \sum_{b \in \B_\fg} b \otimes b^\vee,
		\\
		\Phi(\capmor) &\colon \fg \otimes \fg \to \kk,&
		a \otimes b &\mapsto   \kappa(a,b).
\end{align*}
	Furthermore
\begin{equation} \label{Gsplit}
		\Phi \left( \splitmor \right) \colon \fg \to \fg \otimes \fg,\qquad \qquad \qquad \qquad \quad 
		a \mapsto	\sum_{b \in \B_\fg} [a,b] \otimes b^\vee.
\end{equation}
\end{prop}
\begin{proof}
	We need to check that $\Phi$ preserves the relations in \cref{Tdef}. The verification for the first, second and fourth relations in \cref{chess} is straightforward. For the remaining relation in \cref{chess}, we compute \[
	\Phi
	\left(
	\begin{tikzpicture}[centerzero]
		\draw  (0,-0.4) -- (0,-0.2) to[out=45,in=down] (0.15,0) to[out=up,in=-45] (0,0.2) -- (0,0.4);
		\draw (0,-0.2) to[out=135,in=down] (-0.15,0) to[out=up,in=-135] (0,0.2);
	\end{tikzpicture}
	\right)
	(a)
	=  \sum_{b\in \B_\fg} \left[\left[a,b\right],b^{\vee}\right]
	=\sum_{b \in \B_\fg}\left[b^{\vee},\left[b,a\right]\right]. \]
	Thus 	 \[
	\Phi
	\left(
	\begin{tikzpicture}[centerzero]
		\draw  (0,-0.4) -- (0,-0.2) to[out=45,in=down] (0.15,0) to[out=up,in=-45] (0,0.2) -- (0,0.4);
		\draw (0,-0.2) to[out=135,in=down] (-0.15,0) to[out=up,in=-135] (0,0.2);
	\end{tikzpicture}
	\right)= \sum_{b \in \B_\fg}\mathrm{ad}_{b^{\vee}}\circ \mathrm{ad}_{b}.\]
	Since $\fg$ is simple, the adjoint representation is irreducible and so by Schur's Lemma we have \[ \sum_{b \in \B_\fg}\mathrm{ad}_{b^{\vee}}\circ \mathrm{ad}_{b}=\alpha \id_{\fg},\]
	for some $\alpha\in \kk$. Applying the trace to both sides of the equality we get \[\delta = \sum_{b \in \B_\fg}\kappa(b^\vee, b)=  \sum_{b \in \B_\fg}\tr\left(\mathrm{ad}_{b^{\vee}}\circ \mathrm{ad}_{b}\right)= \alpha\tr(\id_{\fg})=\alpha \delta.\]
	Hence $\alpha=1$. This shows that $\Phi$ preserves the third relation in \cref{chess}.
	
	The verification of the equalities in \cref{venom} and the first two relations in \cref{vortex} is another easy computation. For second-to-last equality in \cref{vortex}, we have
	\[
	\Phi
	\left(
	\begin{tikzpicture}[anchorbase]
		\draw (-0.4,0.2) to[out=down,in=180] (-0.2,-0.2) to[out=0,in=225] (0,0);
		\draw (0,0) -- (0,0.2);
		\draw (0.3,-0.3) -- (0,0);
	\end{tikzpicture}
	\right)
	(a)
	=
	\sum_{b \in \B_\fg} b \otimes  [b^\vee , a] 
	\overset{\cref{magma}}{=}
	\sum_{b \in \B_\fg}  [a , b]  \otimes b^\vee
	=
	\Phi
	\left(
	\begin{tikzpicture}[anchorbase]
		\draw (0.4,0.2) to[out=down,in=0] (0.2,-0.2) to[out=180,in=-45] (0,0);
		\draw (0,0) -- (0,0.2);
		\draw (-0.3,-0.3) -- (0,0);
	\end{tikzpicture}
	\right)
	(a).  
	\]
	This shows that $\Phi$ preserves the second-to-last equality in \cref{vortex}. 
	
	Finally, for the fourth equality in \cref{vortex}, we have
	\[
	\Phi
	\left(
	\begin{tikzpicture}[centerzero]
		\draw (-0.2,-0.3) -- (-0.2,-0.1) arc(180:0:0.2) -- (0.2,-0.3);
		\draw (-0.3,0.3) \braiddown (0,-0.3);
	\end{tikzpicture}
	\right)
	(a \otimes b \otimes c)
	= \kappa(a, c) b
	= \Phi
	\left(
	\begin{tikzpicture}[centerzero]
		\draw (-0.2,-0.3) -- (-0.2,-0.1) arc(180:0:0.2) -- (0.2,-0.3);
		\draw (0.3,0.3) \braiddown (0,-0.3);
	\end{tikzpicture}
	\right)
	(a \otimes b \otimes c).\qedhere
	\] 
\end{proof}

%===============================
\section{Quotients of the trivalent category\label{QTC}}
%===============================
For the rest of this paper we assume $\delta\neq -2$. (See \cref{delta} for a motivation of this restriction.) In this section we study some quotients of $\Tcat_{\delta,\alpha}$. We get these quotients we impose some  dimension restrictions on the hom-spaces, which forces some relations to hold, and thus leads to these quotient categories. First recall that a \emph{tensor ideal} in a $\kk$-linear monoidal category $\cC$ is a collection $\cI$ of vector subspaces $\cI(X,Y)$ of $\cC(X,Y)$ for all $X,Y \in \Ob(\cC)$ such that
\[
    \cC(Y,Z) \circ \cI(X,Y) \subseteq \cI(X,Z)
    \quad \text{and} \quad
    \cI(X,Y) \circ \cC(Z,X) \subseteq \cI(Z,Y),
\]
and  
\[
    1_Z \otimes \cI(X,Y) \subseteq \cI(Z \otimes X, Z \otimes Y)
    \quad \text{and} \quad
    \cI(X,Y) \otimes 1_Z \subseteq \cI(X \otimes Z, Y \otimes Z),
\]
for all $X,Y,Z \in \Ob(\cC)$, with $1_Z$ being the identity morphism of $Z$.  If $\cI$ is a tensor ideal, it follows that $f \otimes g$ and $g \otimes f$ belong to $\cI$ for arbitrary morphisms $f$ in $\cI$ and $g$ in $\cC$.  If $\cI$ is a tensor ideal of $\cC$, then the \emph{quotient category} $\cC/\cI$ is the category with
\[
    \Ob (\cC/\cI) = \Ob(\cC),\quad
    (\cC/\cI)(X,Y) = \cC(X,Y) / \cI(X,Y).
\]
The composition and tensor product in $\cC/\cI$ are induced by those in $\cC$.

Following \cite{GSZ}, consider, for $n\geq 2$, the linear operators
\begin{align*}
	\Rot \colon \Tcat_\delta(\go^{\otimes n}, \go^{\otimes n}) &\to \Tcat_\delta(\go^{\otimes n}, \go^{\otimes n}),
	&
	\Rot
	\left(
	\begin{tikzpicture}[anchorbase]
		\draw (-0.1,-0.4) -- (-0.1,0.4);
		\draw[line width=2] (-0.1,-0.4) -- (-0.1,0);
		\draw (0.1,-0.4) -- (0.1,0);
		\draw[line width=2] (0.1,0.4) -- (0.1,0);
		\filldraw[fill=white,draw=black] (-0.25,0.2) rectangle (0.25,-0.2);
		\node at (0,0) {$\scriptstyle{f}$};
	\end{tikzpicture}
	\right)
	&=
	\begin{tikzpicture}[anchorbase]
		\draw (-0.25,0.2) rectangle (0.25,-0.2);
		\node at (0,0) {$\scriptstyle{f}$};
		\draw (-0.4,-0.4) -- (-0.4,0.2) arc (180:0:0.15);
		\draw (0.4,0.4) -- (0.4,-0.2) arc(360:180:0.15);
		\draw (-0.1,-0.4) -- (-0.1,-0.2);
		\draw (0.1,0.4) -- (0.1,0.2);
		\draw[line width=2] (-0.1,-0.4) -- (-0.1,-.2);
		\draw[line width=2] (0.1,0.4) -- (0.1,0.2);
	\end{tikzpicture}
	\ ,
	\\
	\Switch \colon \Tcat_\delta(\go^{\otimes 2}, \go^{\otimes n}) &\to \Tcat_\delta(\go^{\otimes 2}, \go^{\otimes n}),
	&
	\Switch
	\left(
	\begin{tikzpicture}[anchorbase]
		\draw (-0.1,-0.4) -- (-0.1,0);
		\draw (0.1,-0.4) -- (0.1,0);
		\draw[line width=2] (0,0) -- (0,0.4);
		\filldraw[fill=white,draw=black] (-0.25,0.2) rectangle (0.25,-0.2);
		\node at (0,0) {$\scriptstyle{f}$};
	\end{tikzpicture}
	\right)
	&=
	\begin{tikzpicture}[anchorbase]
		\draw (-0.1,-0.5) -- (0.1,-0.2);
		\draw (0.1,-0.5) -- (-0.1,-0.2);
		\draw[line width=2] (0,0) -- (0,0.4);
		\filldraw[fill=white,draw=black] (-0.25,0.2) rectangle (0.25,-0.2);
		\node at (0,0) {$\scriptstyle{f}$};
	\end{tikzpicture}
	\ .
\end{align*}
That is,
\[
\Rot(f) = \left( \capmor \otimes 1_\go^{\otimes n} \right) \circ (1_\go \otimes f \otimes 1_\go) \circ (1_\go^{\otimes n} \otimes \cupmor)\quad\text{and}\quad
\Switch(f) = f \circ \crossmor.
\]
We have
\begin{equation} \label{rotary}
	\begin{aligned}
		\Rot \left(\, \jail\, \right) &= \hourglass\, ,&
		\Rot \left(\, \hourglass\, \right) &= \jail\, ,&
		\Rot \left(\, \crossmor\, \right) &= \crossmor\, ,
		\\
		\Rot \left(\, \Hmor\, \right) &= \Imor\, ,&
		\Rot \left(\, \Imor\, \right) &= \Hmor\, ,&
		\Rot \left(\, \dotcross\, \right) &= -\dotcross\, .
	\end{aligned}
\end{equation}
and
\begin{equation} \label{flick}
	\begin{aligned}
		\Switch \left(\, \jail\, \right) &= \crossmor\, ,&
		\Switch \left(\, \crossmor\, \right) &= \jail\, ,&
		\Switch \left(\, \hourglass\, \right) &= \hourglass\, ,
		\\
		\Switch \left(\, \Hmor\, \right) &= -\dotcross\, ,&
		\Switch \left(\, \dotcross\, \right) &=- \Hmor\, ,&
		\Switch \left(\, \Imor\, \right) &= -\Imor\, .
	\end{aligned}
\end{equation}
Note that any tensor ideal of $\Tcat_{\delta,\alpha}$ is invariant under $\Rot$ and $\Switch$.

Vector product algebras will show up in some of the quotients of $\Tcat_{\delta,\alpha}$, so we give their definition in the general context of symmetric monoidal $\kk$-linear categories. Recall that we are assuming that $\kk$ is not of characteristic $2$. 

\begin{defin}\label{vcpdef}
A \emph{vector product algebra object} in a symmetric monoidal $\kk$-linear category is an object $\go$ together with morphisms $\capmor \colon \go \otimes \go \to \kk$ and $\cupmor\colon \kk \to \go \otimes \go$, and a morphism $\mergemor\colon  \go \otimes \go \to \go$ subject to relations \cref{vortex}, the first two  relations in~\cref{chess}  and 
\begin{equation}\label{bridge}
	\begin{tikzpicture}[anchorbase,scale=.25]
	\draw (-0.5,0) -- (1,1.5);
	\draw (2.5,0) -- (1,1.5);
	\draw (1,0) -- (0.25,0.75);
	\draw (1,1.5) -- (1,2.5);
\end{tikzpicture} + \begin{tikzpicture}[anchorbase,scale=.25]
	\draw (-0.5,0) -- (1,1.5);
	\draw (2.5,0) -- (1,1.5);
	\draw (1,0) -- (1.75,0.75);
	\draw (1,1.5) -- (1,2.5);
\end{tikzpicture}
=
2\; \begin{tikzpicture}[anchorbase,scale=.25]
	\draw[out=90,in=90,looseness=3] (-1,0) to (1,0);
	\draw (0,0) -- (1,2.5);
\end{tikzpicture} 
-
\begin{tikzpicture}[anchorbase,scale=.25]
	\draw (-1,0) -- (.5,2.5);
	\draw[out=90,in=90,looseness=2] (0,0) to (1.5,0);
\end{tikzpicture} 
-
\begin{tikzpicture}[anchorbase,scale=.25]
	\draw (2.5,0) -- (1,2.5);
	\draw[out=90,in=90,looseness=2] (0,0) to (1.5,0);
\end{tikzpicture}
.
\end{equation}
\end{defin}

In a symmetric pivotal category, relation \cref{bridge} is equivalent to  
\begin{equation}\label{viktor}
	\Imor +\; \Hmor = 2\; \crossmor -\; \jail -\; \hourglass,
\end{equation}
where $\splitmor $ is defined as in \cref{vortex} and $\Hmor$ is defined by \cref{Hmor}. In the category of vector spaces, \cref{vcpdef} stipulates the existence of an inner product $\langle -,-\rangle$ and an alternating linear map $\times \colon\go \times \go \to \go$ such that
\begin{gather}
\langle x,y\times z \rangle = \langle x\times y , z\rangle,\label{fries}\\
\left(x \times y\right) \times z + x\times \left(y \times z\right)  = 2 \langle x, z\rangle y -  \langle x, y\rangle  z - \langle z, y\rangle  x, \label{fish}
\end{gather}
for all elements $x,y,z\in \go$. Here \cref{fish} is relation \cref{bridge} written algebraically.  Assuming that $\times$ is alternating, it can be shown (\cite[Proposition~1]{Street19}) that equation \cref{fish} is equivalent to the equation
\begin{equation}\label{crpd}
\left(x\times y \right)\times x = \langle x, x\rangle y- \langle x , y \rangle x.
\end{equation}
\begin{eg}
	The Euclidean space $\R^3$, equipped with the usual inner product and the cross product of vectors, is a vector product algebra. Notice, however, that in $\R^3$, we actually have $ (x\times y )\times z = \langle x, z\rangle y - \langle x, y\rangle z$, for all $x,y,z\in \R^3$. This is more restrictive than the  relation \cref{crpd}, and is actually the motivation for the defining relation \cref{bridge}.
\end{eg}

\begin{rem}
		It turns out that, over the real numbers, there are not many vector product algebras. In fact, by Hurwitz' theorem \cite{Hur} the only possible dimensions of a vector product algebra are $0,1,3$ and $7$. A diagrammatic proof of this result can be found \cite[\S 3]{West}.  Moreover, vector product algebras, over the real numbers, are exactly the imaginary parts of the four normed division algebras over the reals: the real numbers, the complex numbers, the quaternions and the octonions. For a nice exposition on this subject and related topics we refer to Baez' paper \cite{Baez}. 
\end{rem}
Before we dive into computations, we need one more definition, that of a Lie algebra object in a   symmetric monoidal category.
\begin{defin}
	A \emph{Lie algebra object} in a $\kk$-linear symmetric monoidal category $\cC$, with braiding $s$, consists of an object $\fg \in \cC$ and a morphism, called the \emph{Lie bracket},  $[\cdot ,\cdot]\colon \fg\otimes \fg\to \fg$ subject to the following two conditions:
	\begin{itemize}
		\item[1)] the bracket $[\cdot,\cdot]$ is \emph{skew-symmetric}, that is $[\cdot,\cdot]\circ s =-[\cdot,\cdot]$, and 
		\item[2)] satisfies the \emph{Jacobi identity}
		\[
		[\cdot,\cdot]\circ \left([\cdot,\cdot]\otimes \id_{\fg}\right)\circ\left(\id_{\fg}\otimes s\right) = [\cdot,\cdot]\circ \left([\cdot,\cdot]\otimes\id_{\fg}\right) -[\cdot,\cdot]\circ \left(\id_{\fg}\otimes[\cdot,\cdot]\right).
		\]
	\end{itemize} 
\end{defin}
In the language of string diagrams, we represent the Lie bracket by a trivalent vertex ${\mergemor\colon \fg \otimes \fg\to \fg}$. Then the skew-symmetry can be translated to the diagrammatic   identity 
\[ 
\begin{tikzpicture}[anchorbase]
	\draw (-0.2,-0.5) to[out=45,in=down] (0.15,-0.2) to[out=up,in=-45] (0,0) -- (0,0.2);
	\draw (0.2,-0.5) to [out=135,in=down] (-0.15,-0.2) to[out=up,in=-135] (0,0);
\end{tikzpicture} = -\mergemor
,
\]
while the Jacobi identity is depicted as 
\[
\begin{tikzpicture}[anchorbase,scale=.25]
	\draw (-0.5,0) -- (1,1.5);
	\draw (2.5,0) -- (0.5,1);
	\draw (1,0) -- (1,1.5);
	\draw (1,1.5) -- (1,2.5);
\end{tikzpicture} = \begin{tikzpicture}[anchorbase,scale=.25]
\draw (-0.5,0) -- (1,1.5);
\draw (2.5,0) -- (1,1.5);
\draw (1,0) -- (0.25,0.75);
\draw (1,1.5) -- (1,2.5);
\end{tikzpicture} - \begin{tikzpicture}[anchorbase,scale=.25]
\draw (-0.5,0) -- (1,1.5);
\draw (2.5,0) -- (1,1.5);
\draw (1,0) -- (1.75,0.75);
\draw (1,1.5) -- (1,2.5);
\end{tikzpicture}
.\]
In a pivotal category, the Jacobi identity is equivalent to
\begin{equation}\label{Jacob}
	\dotcross = \Imor - \Hmor.
\end{equation}
\begin{rem}\label{bird}
Let $\cI$ be a tensor ideal in $\Tcat_{\delta,\alpha}$. If either $\capmor=0$ or $\cupmor=0$ in $\Tcat_{\delta,\alpha}/\cI$, then, from the first two relations in \cref{vortex}, it follows that $1_\go=0$, and so $\go$ is the zero object. Hence the category  $\Tcat_{\delta,\alpha}/\cI$ is trivial. An analogous argument using \cref{chess} show that, if $\mergemor=0$ in $\Tcat_{\delta,\alpha}/\cI$, then this quotient is trivial. This follows from the third relation in \cref{chess}.
\end{rem}
\begin{lem}\label{jhcmor}
	Let $\cI$ be a tensor ideal in $\Tcat_{\delta,\alpha}$. If the morphisms 
	\begin{equation}\label{three}
		\jail \, , \quad \hourglass\, , \quad \crossmor
	\end{equation}
	are linearly dependent in the quotient $ \Tcat_{\delta,\alpha}/\cI $, then this quotient is trivial.
\end{lem}
\begin{proof}
	Suppose there is a nontrivial relation
	\begin{equation}\label{door}
		a_1\, \jail +a_2 \, \hourglass +\, a_3 \, \crossmor =0, \quad a_1,a_2,a_3\in \kk.
	\end{equation}
	Then 
	\[
	(a_1-a_2)\left(\jail -\hourglass\right)=(1-\Rot)\left(a_1\, \jail +a_2 \, \hourglass +\, a_3 \, \crossmor\right)=0.
	\]
	If $a_1\neq a_2$, then $\jail=\hourglass$. Composing with $\mergemor$, and using \cref{teardrop}, gives $\mergemor=0$, which, by \cref{bird}, implies that the category is trivial. Hence $a_1=a_2$. Applying $\Switch$ to \cref{door} gives 
	\[
	a_3\, \jail+a_2 \, \hourglass + a_1\, \crossmor =0. 
	\]
	Then an argument similar to the above yields $a_3=a_2$. Thus 
	\[
	\jail +\, \hourglass +\, \crossmor =0.
	\]
	Composing on the top with $\capmor$, and using \cref{chess}, gives $(2+\delta)\capmor=0$. Since  $2+\delta\neq 0$, it follows that $\capmor=0$, and so, by \cref{bird}, the quotient $\Tcat_{\delta,\alpha}/\cI$ is trivial.
\end{proof}
\begin{cor}\label{tamud}There  is no tensor ideal $\cI$ of $\Tcat_{\delta,\alpha}$ for which 
	\begin{equation} 
		\dim \left( \left(\Tcat_{\delta,\alpha}/\cI\right)(\go^{\otimes 2}, \go^{\otimes 2}) \right) = 1  \text{ or } 2.
	\end{equation}
\end{cor}
\begin{proof}
	This follows from \cref{jhcmor}.
\end{proof}

	Recall the definition of the \emph{Temperley--Lieb category} $\TL(-2)$. This is the free $\kk$-linear rigid monoidal category generated by a self-dual object of dimension $-2$.  It is generated by a single object $\go$, and morphisms $\cupmor$ and $\capmor$  subject to the two first relations in \cref{vortex}. It is symmetric with braiding given by 
	\begin{equation}\label{fence}
		\crossmor := \jail +\hourglass
		\ .
	\end{equation}
The cups and caps make $\TL(-2)$ a strict pivotal category, and so we can define a twist:
	\begin{equation}\label{twist}
		\theta_\go :=
		\begin{tikzpicture}[anchorbase,scale=.6]
			\draw[-] (0,0.6) to (0,0.3);
			\draw (0.3,-0.2) to [out=0,in=-90](.5,0);
			\draw (0.5,0) to [out=90,in=0](.3,0.2);
			\draw (0,-0.3) to (0,-0.6);
			\draw (0,0.3) to [out=-90,in=180] (.3,-0.2);
			\draw (0.3,.2) to [out=180,in=90](0,-0.3);
			\draw (0.3,.2) to [out=180,in=90](0,-0.3);
		\end{tikzpicture}
		=
		- \
		\begin{tikzpicture}[centerzero,scale=.6]
			\draw (0,-0.6) -- (0,0.6);
		\end{tikzpicture}
		\ .
	\end{equation}
	This makes $\TL(-2)$ into a \emph{balanced} category. We have the following result.
\begin{rem}\label{delta}
	As we see from the proof of \cref{jhcmor}, the condition $\delta\neq -2$, which we assumed at the beginning of the section, guarantees that the morphisms $\jail\, ,\, \hourglass$ and $\crossmor$ are linearly independent. If $\delta=-2$, then there  is a functor $\TL(-2)\to \Tcat_{-2,\alpha}$, given on objects by sending the generating object $\go $ to $\go$ and on morphisms by 
 \[
 \capmor \mapsto \capmor\; , \; \cupmor \mapsto \cupmor \quad \text{and}\quad  \crossmor \mapsto -\crossmor.
 \]
 We believe that this functor is faithful, but we are not able to give a proof of that.
\end{rem}
%\cref{delta} together with \cref{magneto} give a functor $\TL(-2)\to \fg\md$. This can be seen as a categorical analog of the fact that any finite-dimensional Lie algebra contains $\fsl$ as a subalgebra. 
%

Recall that the \emph{dimension of  a self-dual object} $X$ in a strict pivotal category $\cC$ is $\tr(1_X)\in\End_{\cC}(\one)$. The \emph{trace} $\tr(1_X)$ is obtained by closing the strand of $1_X$:
\[
\tr(1_X):=\begin{tikzpicture}[baseline=-6]
\draw (0,0.2) arc(90:-270:0.3);
\node at (-.42,0) {$\scriptstyle X$};
\end{tikzpicture}.
\]
For instance, the dimension of $\go$ in $\Tcat_{\delta,\alpha}$ is $\delta$, as is given in \cref{chess}.
\begin{prop}\label{sumer}
If $\cI$ is a tensor ideal of $\Tcat_{\delta,\alpha}$ such that 
\begin{equation*}
		 \dim \left( \left(\Tcat_{\delta,\alpha}/\cI\right)(\go^{\otimes 2}, \go^{\otimes 2}) \right) = 3, 
\end{equation*}
then the relations  
\begin{equation}\label{SL2}
	\dotcross = \Imor-\, \Hmor,\,\quad \Imor = \frac{\alpha}{2}\left(\jail -\, \crossmor\right)
\end{equation}  
hold in the quotient category $ \Tcat_{\delta,\alpha}/\cI $.  In particular $\go$ is a Lie algebra object of dimension $ 3 $ in $\Tcat_{\delta,\alpha}/\cI$.
\end{prop}

\begin{proof}
By \cref{jhcmor}, the morphisms \cref{three} form a basis of the space $\Tcat_{\delta,\alpha}/\cI \left(\go^{\otimes 2}, \go^{\otimes 2}\right)$. Then we can write
\[
\Imor = a\; \crossmor + b\; \jail + c\;\hourglass,  
\]
for  $ a,b,c \in \kk$. Applying $\Switch$ to this we get 
\[
-\Imor = a\; \jail + b\; \crossmor + c\;\hourglass.
\]
If we sum the two equations and use the fact that \cref{three} are linearly independent, we get $a=-b$ and $c=0$. Hence 
\begin{equation}\label{Imomo}
\Imor = b\left(\jail -\crossmor\right).
\end{equation}
Composing with $\mergemor$, and using \cref{chess}, we get $b=\frac{\alpha}{2}$. This gives the second relation in \cref{SL2}. Applying the operator $\Rot$ to this gives 
\begin{equation}\label{Hmomo}
	\Hmor = \frac{\alpha}{2}\left(\hourglass -\crossmor\right).
\end{equation}
Composing with $\capmor$ on the top  gives $\delta=3$. Finally, applying $\Switch$ to \cref{Hmomo} and using \cref{turvy} yields
\[
\dotcross = \frac{\alpha}{2}\left(\jail -\hourglass\right).
\]
Combining this with \cref{Imomo,Hmomo} gives the first relation in \cref{SL2}.
\end{proof}
\begin{eg}
Let $ \mathpzc{SL} $ be the quotient category of $\Tcat_{3,1}$  by the relations \cref{SL2}. Let $\mathfrak{sl}_2(\kk)\md$ be the category of finite-dimensional representations of $\mathfrak{sl}_2(\kk)$.  Since 
\[
  \dim_{\kk}\End_{\fsl}(\fsl^{\otimes 2})=3, 
\]
it follows from \cref{magneto} and  \cref{sumer} that there is a functor 
\[
\Psi\colon\mathpzc{SL}\to \mathfrak{sl}_2(\kk)\md,
\]
which we will now describe explicitly. The Lie algebra $\mathfrak{sl}_2(\kk)$ has basis 
\[ 
e=\begin{pmatrix}
	0& 1\\ 0&0
\end{pmatrix},\quad f=\begin{pmatrix}
	0 & 0 \\ 1 &0
\end{pmatrix},\quad h=\begin{pmatrix}
	1&0 \\ 0& -1
\end{pmatrix}
,
\]
with Lie bracket $[e,f]=h,[h,e]=2e$ and $[h,f]=-2f$. Its Killing form is given by $\kappa(x,y)=4\tr(xy)$ for $x,y\in\mathfrak{sl}_2(\kk)$. So the dual basis of the above basis with respect to the Killing form is 
\[ 
e^{\vee}=\frac{1}{4}f,\quad f^{\vee} =\frac{1}{4}e,\quad h^{\vee}=\frac{1}{8}h
,
\]
and the Casimir element $\xi \in \mathfrak{sl}_2(\kk)^{\otimes 2}$ is 
\[
\xi = e\otimes e^\vee + f \otimes f^\vee + h\otimes h^\vee =\frac{1}{4}\left(e\otimes f + f \otimes e\right) + \frac{1}{8}h\otimes h.
\]
Then $\Psi$ can be defined  by $\Psi(\go)=\fsl$, and     \begin{gather*}
	\Psi(\mergemor)(x\otimes y)=[x,y],\quad  \Psi(\crossmor)(x\otimes y)=y\otimes x,\quad  \Psi(\capmor)(x\otimes y)=4\tr(xy),\\ \Psi(\cupmor)(1)=\xi,\quad  \Psi(\splitmor)([x,y])=\frac{1}{2}\left( x\otimes y - y \otimes x\right),
\end{gather*}   
for all $x,y\in \fsl$.
\end{eg}

%Define two linear maps $\varphi_{1}\colon \kk\to   \mathfrak{sl}_2(\kk)^{\otimes 2}$ with
%\[ 
%\varphi_{1}(c)=c\xi, \quad \text{ for }c\in\kk
%,
%\]
%and $\varphi_{2}\colon\mathfrak{sl}_2(\kk)\to\mathfrak{sl}_2(\kk)^{\otimes 2}$ with
%\begin{align*}
%	\varphi_{2}(x)  & = [x,e]\otimes e^{\vee}+[x,f]\otimes f^{\vee} + [x,h]\otimes h^{\vee} \\
%	& =\frac{1}{4} \left([x,e]\otimes f+ [x,f]\otimes e\right) + \frac{1}{8}[x,h]\otimes h.
%\end{align*} 
%Then a straightforward computation yields the identity
%\begin{equation}\label{slmid}
%	\varphi_{2}\left([x, y]\right)= \frac{1}{2}\left(x\otimes y - y\otimes x\right),
%\end{equation}
%for all $x,y\in \fsl$. Then we have a monoidal functor $\Phi\colon  \mathpzc{SL}\to \mathfrak{sl}_2(\kk)\md$, given on objects by sending $\go$ to the adjoint representation of $\mathfrak{sl}_2(\kk)$, and on morphisms by
%	\begin{align*}
%		\Phi(\mergemor) &\colon \fsl \otimes \fsl \to \fsl,&
%		x \otimes y &\mapsto [x,y],
%		\\
%		\Phi(\splitmor) &\colon \fsl\to \fsl \otimes \fsl ,&
%		x   &\mapsto  \varphi_{2}(x),\\
%		\Phi(\crossmor) &\colon\fsl \otimes \fsl \to \fsl \otimes  \fsl,&
%		x \otimes y &\mapsto y \otimes x,
%		\\
%		\Phi(\cupmor) &\colon \kk \to \fsl \otimes \fsl, &
%		1& \mapsto \varphi_{1}(1)=\xi,
%		\\
%		\Phi(\capmor) &\colon \fsl\otimes \fsl \to \kk,&
%		x \otimes y &\mapsto   4\tr(xy).
%	\end{align*}	 

Define an idempotent $e\in \TL(-2)(\go^{\otimes 2}, \go^{\otimes 2})$  by 
\begin{equation}\label{idemp}
e=\begin{tikzpicture}[centerzero]
	\draw (-.15,-.4) to (-.15,.4);
	\draw (.15,-.4) to (.15, .4);
	\symbox{-0.3,-0.15}{0.3,0.15};
\end{tikzpicture}:=\frac{1}{2} \left(\jail+\crossmor\right)=\jail + \frac{1}{2}\, \hourglass.
\end{equation}

We denote the additive Karoubi envelope of a category $\cC$ by   $\Kar(\cC)$.  Objects of $\Kar(\cC)$ are pairs $(X,e)$, where $X$ is an object in the additive envelope $\Add(\cC)$ of $\cC$, and $e \in \cC(X,X)$ is an idempotent endomorphism.  Morphisms in $\Kar(\cC)$ are given by
\[
\Kar(\cC) \big( (X,e),(X',e') \big) := e' \Add(\cC)(X,X') e.
\]
Composition is as in $\cC$. If $\cC$ is monoidal, then so is $\Kar\cC$ with a tensor product given by $(X,e)\otimes (X',e'):= (X\otimes X',e\otimes e')$. Diagrammatically, a morphism $f\colon (X,e)\to (X',e')$ is depicted by
\[
\begin{tikzpicture}[centerzero] 
	\draw[thick] (.15,-.9) to (.15, .9);
	\symbox{-.05,-0.15}{0.35,0.15};
	\symbox{-.05,0.3}{0.35,0.6};
	\symbox{-.05,-0.6}{0.35,-0.3};
	\node at (.15,-.45) {$\scriptstyle e$};
	\node at (.15,.45) {$\scriptstyle e'$};
	\node at (.15,0) {$\scriptstyle f$};
	\node at (-.15,.75) {$\scriptstyle X'$};
	\node at (-.15,-.75) {$\scriptstyle X$};
\end{tikzpicture}.
\]
Next, we relate the quotient category $\mathpzc{SL}$ to the Temperley--Lieb category.
\begin{prop}\label{Kauff}Assume $\kk$ contains a square root of $-1$. There is a unique  $\kk$-linear monoidal functor $\Phi\colon \mathpzc{SL}\to \Kar{\TL}(-2)$ given on objects by sending $\go$ to the object $(\go^{\otimes 2}, e)$, and on morphisms by
\[
\capmor \mapsto \begin{tikzpicture}[baseline=-14]
	\draw[out=90, in=90, looseness=1.5] (-.45, -.4) to (.45,-.4);
	\draw[out=90,in=90,looseness=2] (-.15,-.4) to (.15, -.4);
	\draw (-.45, -.4) to (-.45,-.7);
	\draw (.45, -.4) to (.45,-.7);
	\draw (-.15, -.4) to (-.15,-.7);
	\draw (.15, -.4) to (.15,-.7);
	\symbox{.05,-0.6}{0.55,-.4};
	\symbox{-0.55,-0.6}{-.05,-.4};
\end{tikzpicture}\; \; , \qquad	\mergemor \mapsto \sqrt{-1} \begin{tikzpicture}[baseline=-6]
		\draw[out=90, in=-90, looseness=1] (-.45, -.4) to (-.15,.2);
		\draw[out=90,in=-90,looseness=1] (.45,-.4) to (.15, .2);
		\draw[out=90,in=90,looseness=2] (-.15,-.4) to (.15, -.4);
		\draw (-.45, -.4) to (-.45,-.7);
		\draw (.45, -.4) to (.45,-.7);
		\draw (-.15, -.4) to (-.15,-.7);
		\draw (.15, -.4) to (.15,-.7);
		\draw (-.15, .2) to (-.15,.5);
		\draw (.15, .2) to (.15,.5);
		\symbox{-0.25,0.2}{0.25,0.4};
		\symbox{.05,-0.6}{0.55,-.4};
		\symbox{-0.55,-0.6}{-.05,-.4};
	\end{tikzpicture}\;, \qquad \crossmor \mapsto \begin{tikzpicture}[baseline=-6,scale=.8]
	\draw[out=90,in=-90]  (-.6,-.6) to (.3,.6);
	\draw[out=90,in=-90]  (-.3,-.6) to (.6,.6);
	\draw[out=90,in=-90]  (.6,-.6) to (-.3,.6);
	\draw[out=90,in=-90]  (.3,-.6) to (-.6,.6);
	\draw (-.6,-.9) to (-.6,-.8);
	\draw (-.3,-.9) to (-.3,-.8);
	\draw (.6,-.9) to (.6,-.8);
	\draw (.3,-.9) to (.3,-.8);
	\symbox{-.7,-.6}{-.2,-.8};
	\symbox{.7,-.6}{.2,-.8};
\end{tikzpicture}
\ .
\]	
\end{prop}
\begin{proof}
	We need to check that $\Phi$ preserves the defining relations in $\mathpzc{SL}$. The verification for the set of relations in \cref{vortex,venom} is straightforward and follows from the fact that $\TL(-2)$ is pivotal. For the first relation in \cref{chess}, note that 
\begin{equation}\label{mushroom}
		\Phi(\capmor) = \begin{tikzpicture}[centerzero]
		\draw[out=90,in=90,looseness=2] (-.15,-.15) to (.15,-.15);
		\draw[out=90,in=90,looseness=2] (-.3,-.15) to (.3,-.15);
	\end{tikzpicture}\, +\frac{1}{2}\; \capmor\,\capmor.
\end{equation}
It follows (in all the following arguments we will freely use the fact that $\TL(-2)$ is symmetric without mentioning it)
\begin{equation}\label{scarf}
		\Phi\left(\begin{tikzpicture}[anchorbase]
			\draw (-0.15,-0.4) to[out=45,in=down] (0.15,0) arc(0:180:0.15) to[out=down,in=135] (0.15,-0.4);
		\end{tikzpicture}\right)  = \begin{tikzpicture}[centerzero]
		\draw[out=90,in=90,looseness=2] (-.15,0) to (.15,0);
		\draw[out=90,in=90,looseness=2] (-.3,0) to (.3,0);
		\draw[out=-90,in=45] (.15, 0) to (-.3,-.4);
		\draw[out=-90,in=135] (-.15, 0) to (.3,-.4);
		\draw[out=-90,in=45] (.3, 0) to (-.15,-.4);
		\draw[out=-90,in=135] (-.3, 0) to (.15,-.4);
	\end{tikzpicture} 
+\frac{1}{2}\; 
\begin{tikzpicture}[baseline=-8]
	\draw[out=90,in=90, looseness=2] (-.3,0) to (-.05,0);
	\draw[out=90,in=90, looseness=2] (.05,0) to (.3,0);
	\draw[out=-90,in=45]  (.3,0) to (-.3,-.6);
	\draw[out=-90,in=45] (.05,0) to (-.4,-.4);
	\draw[out=-90,in=135] (-.3,0) to (0.1,-.6);
	\draw[out=-90,in=135] (-.05,0) to (.2,-.4);
	\draw[out=90,in=-135] (-.5,-.6) to (-.4,-.4);
	\draw[out=90,in=-45]  (.4,-.6) to (.2,-.4);
\end{tikzpicture}
\overset{\cref{twist}}{=}
- \begin{tikzpicture}[centerzero]
	\draw[out=90,in=90,looseness=2] (-.3,0) to (.3,0);
	\draw[out=90,in=90,looseness=2] (-.3, -.4) to (.3,-.4);
	\draw[out=-90,in=45] (.3, 0) to (-.15,-.4);
	\draw[out=-90,in=135] (-.3, 0) to (.15,-.4);
\end{tikzpicture} + \frac{1}{2}\, \capmor \, \capmor 
\overset{\cref{twist}}{=} \begin{tikzpicture}[centerzero]
	\draw[out=90,in=90,looseness=2] (-.15,-.15) to (.15,-.15);
	\draw[out=90,in=90,looseness=2] (-.3,-.15) to (.3,-.15);
\end{tikzpicture}\, +\frac{1}{2}\; \capmor\,\capmor  
= \Phi(\capmor).
	\end{equation}
For the second relation in \cref{chess}, notice that we can expand $\Phi(\mergemor)$ in the form
\begin{multline*}
\frac{1}{\sqrt{-1}}\Phi\left(\mergemor\right) 
 = 
\begin{tikzpicture}[centerzero]
\draw (0,-.3) to (0,.3);
\draw (.6,-.3) to (.6,.3);
\draw[out=90,in=90,looseness=2] (.15,-.3) to (.45,-.3);
\end{tikzpicture} 
+
 \frac{1}{2}\; \begin{tikzpicture}[centerzero]
\draw (0,-.3) to (0,.3);
\draw (-.2,-.3) to (-.2,.3);
\draw[out=90,in=90,looseness=2] (.15,-.3) to (.45,-.3);
\end{tikzpicture}
+
 \frac{1}{2}\; \begin{tikzpicture}[centerzero]
\draw (.8,-.3) to (.8,.3);
\draw (.6,-.3) to (.6,.3);
\draw[out=90,in=90,looseness=2] (.15,-.3) to (.45,-.3);
\end{tikzpicture}
+
\frac{1}{2}\; \begin{tikzpicture}[centerzero]
	\draw[out=90,in=90,looseness=2] (-.15,-.3) to (.15,-.3);
	\draw[out=90,in=90,looseness=2] (-.3,-.3) to (.3,-.3);
	\draw[out=-90,in=-90,looseness=2] (-.15,.3) to (.15,0.3);
\end{tikzpicture}
+
\frac{1}{2}\; \begin{tikzpicture}[centerzero]
	\draw[out=90,in=90,looseness=2] (.15,-.3) to (.45,-.3);
	\draw[out=90,in=90,looseness=2] (-.3,-.3) to (.0,-.3);
	\draw[out=-90,in=-90,looseness=2] (-.15,.3) to (.15,0.3);
\end{tikzpicture}\\
 \overset{\cref{mushroom}}{=} \begin{tikzpicture}[centerzero]
	\draw (0,-.3) to (0,.3);
	\draw (.6,-.3) to (.6,.3);
	\draw[out=90,in=90,looseness=2] (.15,-.3) to (.45,-.3);
\end{tikzpicture} 
+
\frac{1}{2}\; \begin{tikzpicture}[centerzero]
	\draw (0,-.3) to (0,.3);
	\draw (-.2,-.3) to (-.2,.3);
	\draw[out=90,in=90,looseness=2] (.15,-.3) to (.45,-.3);
\end{tikzpicture}
+
\frac{1}{2}\; \begin{tikzpicture}[centerzero]
	\draw (.8,-.3) to (.8,.3);
	\draw (.6,-.3) to (.6,.3);
	\draw[out=90,in=90,looseness=2] (.15,-.3) to (.45,-.3);
\end{tikzpicture}
+ \frac{1}{4} \; \; \begin{tikzpicture}[centerzero]
	\draw[out=90,in=90,looseness=2] (.15,-.3) to (.45,-.3);
	\draw[out=90,in=90,looseness=2] (-.3,-.3) to (.0,-.3);
	\draw[out=-90,in=-90,looseness=2] (-.15,.3) to (.15,0.3);
\end{tikzpicture}
+ \frac{1}{2}\cupmor\circ\Phi\left(\capmor\right).
\end{multline*}
Hence 
\begin{multline*}
	\frac{1}{\sqrt{-1}}\Phi\left(\begin{tikzpicture}[anchorbase]
		\draw (-0.2,-0.5) to[out=45,in=down] (0.15,-0.2) to[out=up,in=-45] (0,0) -- (0,0.2);
		\draw (0.2,-0.5) to [out=135,in=down] (-0.15,-0.2) to[out=up,in=-135] (0,0);
	\end{tikzpicture}\right)  \overset{\cref{scarf}}{=}\begin{tikzpicture}[centerzero]
	\draw[out=90,in=90,looseness=2] (-.15,0) to (.15,0);
	\draw  (-.3,0) to (-.3,.3);
	\draw  (.3,0) to (.3,.3);
	\draw[out=-90,in=45] (.15, 0) to (-.3,-.4);
	\draw[out=-90,in=135] (-.15, 0) to (.3,-.4);
	\draw[out=-90,in=45] (.3, 0) to (-.15,-.4);
	\draw[out=-90,in=135] (-.3, 0) to (.15,-.4);
\end{tikzpicture} +
\frac{1}{2}\; \begin{tikzpicture}[centerzero]
\draw (0,-.3) to (0,.3);
\draw (-.2,-.3) to (-.2,.3);
\draw[out=90,in=90,looseness=2] (.15,-.3) to (.45,-.3);
\end{tikzpicture}
+
\frac{1}{2}\; \begin{tikzpicture}[centerzero]
\draw (.8,-.3) to (.8,.3);
\draw (.6,-.3) to (.6,.3);
\draw[out=90,in=90,looseness=2] (.15,-.3) to (.45,-.3);
\end{tikzpicture}
+ \frac{1}{4} \; \; \begin{tikzpicture}[centerzero]
\draw[out=90,in=90,looseness=2] (.15,-.3) to (.45,-.3);
\draw[out=90,in=90,looseness=2] (-.3,-.3) to (.0,-.3);
\draw[out=-90,in=-90,looseness=2] (-.15,.3) to (.15,0.3);
\end{tikzpicture}
+ 
\frac{1}{2}\cupmor\circ\Phi\left(\capmor\right) \\
\overset{\cref{twist}}{\underset{\cref{mushroom}}{=}} 
- \begin{tikzpicture}[centerzero]
	\draw (-.3,0) to (-.3,.2);
	\draw (.3,0) to (.3,.2);
	\draw[out=90,in=90,looseness=2] (-.3, -.4) to (.3,-.4);
	\draw[out=-90,in=45] (.3, 0) to (-.15,-.4);
	\draw[out=-90,in=135] (-.3, 0) to (.15,-.4);
\end{tikzpicture}
+
\frac{1}{2}\; \begin{tikzpicture}[centerzero]
	\draw (0,-.3) to (0,.3);
	\draw (-.2,-.3) to (-.2,.3);
	\draw[out=90,in=90,looseness=2] (.15,-.3) to (.45,-.3);
\end{tikzpicture}
+
\frac{1}{2}\; \begin{tikzpicture}[centerzero]
	\draw (.8,-.3) to (.8,.3);
	\draw (.6,-.3) to (.6,.3);
	\draw[out=90,in=90,looseness=2] (.15,-.3) to (.45,-.3);
\end{tikzpicture}
+ 
\frac{1}{2}\; \begin{tikzpicture}[centerzero]
	\draw[out=90,in=90,looseness=2] (-.15,-.3) to (.15,-.3);
	\draw[out=90,in=90,looseness=2] (-.3,-.3) to (.3,-.3);
	\draw[out=-90,in=-90,looseness=2] (-.15,.3) to (.15,0.3);
\end{tikzpicture}
+
\frac{1}{2}\; \begin{tikzpicture}[centerzero]
	\draw[out=90,in=90,looseness=2] (.15,-.3) to (.45,-.3);
	\draw[out=90,in=90,looseness=2] (-.3,-.3) to (.0,-.3);
	\draw[out=-90,in=-90,looseness=2] (-.15,.3) to (.15,0.3);
\end{tikzpicture} 
 \overset{\cref{fence}}{=}- \frac{1}{\sqrt{-1}}\Phi\left(\mergemor\right).
\end{multline*}
In the last step, we only resolve the crossing using its definition in \cref{fence} and then compare the resulting sum to $\Phi(\mergemor)$. To show  the third relation in \cref{chess}, we first note that
\begin{equation}\label{proj}
	\begin{tikzpicture}[baseline=-17]
	\draw[out=90,in=90,looseness=2] (-.15,-.4) to (.15, -.4);
	\symbox{-0.25,-.4}{0.25,-.6};
	\draw (-0.15,-.6) to (-0.15,-.7);
	\draw (0.15,-.6) to (0.15,-.7);
\end{tikzpicture}=0 \quad \text{and} \quad \scalebox{1}[-1]{\begin{tikzpicture}[baseline=-12]
\draw[out=90,in=90,looseness=2] (-.15,-.4) to (.15, -.4);
\symbox{-0.25,-.4}{0.25,-.6};
\draw (-0.15,-.6) to (-0.15,-.7);
\draw (0.15,-.6) to (0.15,-.7);
\end{tikzpicture}}=0.
\end{equation}
Then we have 
\begin{equation*}
- \Phi\left(\begin{tikzpicture}[centerzero]
	\draw  (0,-0.4) -- (0,-0.2) to[out=45,in=down] (0.15,0) to[out=up,in=-45] (0,0.2) -- (0,0.4);
	\draw (0,-0.2) to[out=135,in=down] (-0.15,0) to[out=up,in=-135] (0,0.2);
\end{tikzpicture}\right)
  =  \begin{tikzpicture}[baseline=-17]
 	\draw[out=90, in=-90, looseness=1] (-.45, -.4) to (-.15,.2);
 	\draw[out=90,in=-90,looseness=1] (.45,-.4) to (.15, .2);
 	\draw[out=90,in=90,looseness=2] (-.15,-.4) to (.15, -.4);
 	\draw (-.15, .2) to (-.15,.5);
 	\draw (.15, .2) to (.15,.5);
 	\draw[out=-90, in=90, looseness=1] (-.45, -.6) to (-.15,-1.2);
 	\draw[out=-90,in=90,looseness=1] (.45,-.6) to (.15, -1.2);
 	\draw[out=-90,in=-90,looseness=2] (-.15,-.6) to (.15, -.6);
 	\symbox{-0.25,0.2}{0.25,0.4};
 	\symbox{.05,-0.6}{0.55,-.4};
 	\symbox{-0.55,-0.6}{-.05,-.4};
 	\symbox{-0.25,-1.2}{0.25,-1.4};
 	\draw (-0.15,-1.4) to (-0.15,-1.5);
 	\draw (0.15,-1.4) to (0.15,-1.5);
 \end{tikzpicture}
\overset{\cref{idemp}}{=}
  \begin{tikzpicture}[baseline=-14]
	\draw[out=90, in=-90, looseness=1] (-.45, -.4) to (-.15,.2);
	\draw[out=90,in=-90,looseness=1] (.45,-.4) to (.15, .2);
	\draw[out=90,in=90,looseness=2] (-.15,-.4) to (.15, -.4);
	\draw (-.15, .2) to (-.15,.5);
	\draw (.15, .2) to (.15,.5);
	\draw[out=-90, in=90, looseness=1] (-.45, -.4) to (-.15,-1);
	\draw[out=-90,in=90,looseness=1] (.45,-.4) to (.15, -1);
	\draw[out=-90,in=-90,looseness=2] (-.15,-.4) to (.15, -.4);
	\symbox{-0.25,0.2}{0.25,0.4};
	%\symbox{.05,-0.6}{0.55,-.4};
	%\symbox{-0.55,-0.6}{-.05,-.4};
	\symbox{-0.25,-1}{0.25,-1.2};
	\draw (-0.15,-1.2) to (-0.15,-1.3);
	\draw (0.15,-1.2) to (0.15,-1.3);
\end{tikzpicture}+\; \frac{1}{2} \begin{tikzpicture}[centerzero]
\draw (-.15,-.4) to (-.15,.4);
\draw (.15,-.4) to (.15, .4);
\symbox{-0.3,-0.15}{0.3,0.15};
\end{tikzpicture} +\frac{1}{2}\; \begin{tikzpicture}[centerzero]
\draw (-.15,-.4) to (-.15,.4);
\draw (.15,-.4) to (.15, .4);
\symbox{-0.3,-0.15}{0.3,0.15};
\end{tikzpicture} + \frac{1}{4}\;   \begin{tikzpicture}[baseline=-6]
\draw[out=-90, in=-90, looseness=2] (.15, .2) to (-.15,.2);
\draw (-.15, .2) to (-.15,.5);
\draw (.15, .2) to (.15,.5);
\draw[out=90,in=90,looseness=2] (-.15,-.4) to (.15, -.4);
\symbox{-0.25,0.2}{0.25,0.4};
%\symbox{.05,-0.6}{0.55,-.4};
%\symbox{-0.55,-0.6}{-.05,-.4};
\symbox{-0.25,-.4}{0.25,-.6};
\draw (-0.15,-.6) to (-0.15,-.7);
\draw (0.15,-.6) to (0.15,-.7);
\end{tikzpicture}= -\; \begin{tikzpicture}[centerzero]
\draw (-.15,-.4) to (-.15,.4);
\draw (.15,-.4) to (.15, .4);
\symbox{-0.3,-0.15}{0.3,0.15};
\end{tikzpicture}.\\
\end{equation*} 
For the last relation in \cref{chess}, we  have
\[
\Phi\left(\bubble\right) =  \begin{tikzpicture}[baseline=-17]
	\draw[out=90, in=90, looseness=1.5] (-.45, -.4) to (.45,-.4);
	\draw[out=90,in=90,looseness=2] (-.15,-.4) to (.15, -.4);
	\draw[out=-90,in=-90,looseness=1.5] (-.45, -.6) to (.45,-.6);
	\draw[out=-90,in=-90,looseness=2] (-.15, -.6) to (.15,-.6);
	\symbox{.05,-0.6}{0.55,-.4};
	\symbox{-0.55,-0.6}{-.05,-.4};
\end{tikzpicture} = \begin{tikzpicture}[baseline=-17]
\draw[out=90, in=90, looseness=1.5] (-.45, -.4) to (.45,-.4);
\draw[out=90,in=90,looseness=2] (-.15,-.4) to (.15, -.4);
\draw[out=-90,in=-90,looseness=1.5] (-.45, -.6) to (.45,-.6);
\draw[out=-90,in=-90,looseness=2] (-.15, -.6) to (.15,-.6);
\symbox{.05,-0.6}{0.55,-.4};
%\symbox{-0.55,-0.6}{-.05,-.4};
\draw (-.45, -.6) to (-.45,-.4);
\draw (-.15, -.6) to (-.15,-.4);
\end{tikzpicture} \overset{\cref{idemp}}{=}\begin{tikzpicture}[baseline=-17]
\draw[out=90, in=90, looseness=1.5] (-.45, -.4) to (.45,-.4);
\draw[out=90,in=90,looseness=2] (-.15,-.4) to (.15, -.4);
\draw[out=-90,in=-90,looseness=1.5] (-.45, -.6) to (.45,-.6);
\draw[out=-90,in=-90,looseness=2] (-.15, -.6) to (.15,-.6);
\draw (.45, -.6) to (.45,-.4);
\draw (.15, -.6) to (.15,-.4);
\draw (-.45, -.6) to (-.45,-.4);
\draw (-.15, -.6) to (-.15,-.4);
\end{tikzpicture}+\frac{1}{2}\;\begin{tikzpicture}[baseline=-17]
\draw[out=90, in=90, looseness=1.5] (-.45, -.35) to (.45,-.35);
\draw[out=90,in=90,looseness=2] (-.15,-.35) to (.15, -.35);
\draw[out=-90,in=-90,looseness=1.5] (-.45, -.65) to (.45,-.65);
\draw[out=-90,in=-90,looseness=2] (-.15, -.65) to (.15,-.65);
\draw[out=90,in=90] (.45, -.65) to (.15,-.65);
\draw[out=-90,in=-90] (.45, -.35) to (.15,-.35);
\draw (-.45, -.65) to (-.45,-.35);
\draw (-.15, -.65) to (-.15,-.35);
\end{tikzpicture}= \bubble\,\bubble +\frac{1}{2}\bubble=3\times 1_\one.
\]

Next, we verify the second relation in \cref{SL2}.  First we expand the double crossing  in $\TL(-2)$ using \cref{fence}
\begin{multline*}
	\begin{tikzpicture}[centerzero,scale=.6]
		\draw[out=90,in=-90]  (-.6,-.6) to (.3,.6);
		\draw[out=90,in=-90]  (-.3,-.6) to (.6,.6);
		\draw[out=90,in=-90]  (.6,-.6) to (-.3,.6);
		\draw[out=90,in=-90]  (.3,-.6) to (-.6,.6);
	\end{tikzpicture} = \begin{tikzpicture}[centerzero]
	\draw (0,-.3) to (0,.3);
	\draw (.6,-.3) to (.6,.3);
	\draw (.2,-.3) to (0.2,.3);
	\draw (.4,-.3) to (.4,.3);
\end{tikzpicture} + 2\, \begin{tikzpicture}[centerzero]
	\draw (0,-.3) to (0,.3);
	\draw (.6,-.3) to (.6,.3);
	\draw[out=90,in=90,looseness=2] (.15,-.3) to (.45,-.3);\draw[out=-90,in=-90,looseness=2] (.15,.3) to (.45,.3);
\end{tikzpicture} 
+
\begin{tikzpicture}[centerzero]
\draw (0,-.3) to (0,.3);
\draw (-.2,-.3) to (-.2,.3);
\draw[out=90,in=90,looseness=2] (.15,-.3) to (.45,-.3);
\draw[out=-90,in=-90,looseness=2] (.15,.3) to (.45,.3);
\end{tikzpicture}
+
\begin{tikzpicture}[centerzero]
\draw (.8,-.3) to (.8,.3);
\draw (.6,-.3) to (.6,.3);
\draw[out=90,in=90,looseness=2] (.15,-.3) to (.45,-.3);
\draw[out=-90,in=-90,looseness=2] (.15,.3) to (.45,.3);
\end{tikzpicture}
+
\begin{tikzpicture}[centerzero]
\draw[out=90,in=90,looseness=1.5] (-.15,-.3) to (.15,-.3);
\draw[out=90,in=90,looseness=1.5] (-.3,-.3) to (.3,-.3);
\draw[out=-90,in=-90,looseness=1.5] (-.15,.3) to (.15,0.3);
\draw[out=-90,in=-90,looseness=1.5] (-.3,.3) to (.3,.3);
\end{tikzpicture}
+
\begin{tikzpicture}[centerzero]
\draw[out=90,in=90,looseness=2] (.15,-.3) to (.45,-.3);
\draw[out=90,in=90,looseness=2] (-.3,-.3) to (.0,-.3);
\draw[out=-90,in=-90,looseness=1.5] (-.15,.3) to (.15,0.3);
\draw[out=-90,in=-90,looseness=1.5] (-.3,.3) to (.3,0.3);
\end{tikzpicture}
+
\begin{tikzpicture}[centerzero]
	\draw[out=90,in=90,looseness=2] (.15,-.3) to (.45,-.3);
	\draw[out=90,in=90,looseness=2] (-.3,-.3) to (.0,-.3);
	\draw[out=-90,in=-90,looseness=2] (.15,.3) to (.45,.3);
	\draw[out=-90,in=-90,looseness=2] (-.3,.3) to (.0,.3);
\end{tikzpicture}
+
\scalebox{1}[-1]{\begin{tikzpicture}[baseline=2]
	\draw[out=90,in=90,looseness=2] (.15,-.3) to (.45,-.3);
	\draw[out=90,in=90,looseness=2] (-.3,-.3) to (.0,-.3);
	\draw[out=-90,in=-90,looseness=1.5] (-.15,.3) to (.15,0.3);
	\draw[out=-90,in=-90,looseness=1.5] (-.3,.3) to (.3,0.3);
\end{tikzpicture}}
+
\begin{tikzpicture}[centerzero]
	\draw[out=90,in=-90] (0,-.3) to (0.45,.3);
	\draw (-.2,-.3) to (-.2,.3);
	\draw[out=90,in=90,looseness=2] (.15,-.3) to (.45,-.3);
	\draw[out=-90,in=-90,looseness=2] (0,.3) to (.3,.3);
\end{tikzpicture}
+
\scalebox{1}[-1]{
\begin{tikzpicture}[baseline=2]
	\draw[out=90,in=-90] (0,-.3) to (0.45,.3);
	\draw (-.2,-.3) to (-.2,.3);
	\draw[out=90,in=90,looseness=2] (.15,-.3) to (.45,-.3);
	\draw[out=-90,in=-90,looseness=2] (0,.3) to (.3,.3);
\end{tikzpicture}
}
+\scalebox{-1}[1]{
	\begin{tikzpicture}[centerzero]
		\draw[out=90,in=-90] (0,-.3) to (0.45,.3);
		\draw (-.2,-.3) to (-.2,.3);
		\draw[out=90,in=90,looseness=2] (.15,-.3) to (.45,-.3);
		\draw[out=-90,in=-90,looseness=2] (0,.3) to (.3,.3);
	\end{tikzpicture}
}
+\scalebox{-1}[-1]{
	\begin{tikzpicture}[baseline=2]
		\draw[out=90,in=-90] (0,-.3) to (0.45,.3);
		\draw (-.2,-.3) to (-.2,.3);
		\draw[out=90,in=90,looseness=2] (.15,-.3) to (.45,-.3);
		\draw[out=-90,in=-90,looseness=2] (0,.3) to (.3,.3);
	\end{tikzpicture}
}.
\end{multline*}
Composing with 
\[\begin{tikzpicture}[centerzero]
	\draw (-.15,-.4) to (-.15,.4);
	\draw (.15,-.4) to (.15, .4);
	\symbox{-0.3,-0.15}{0.3,0.15};
\end{tikzpicture}\, \begin{tikzpicture}[centerzero]
\draw (-.15,-.4) to (-.15,.4);
\draw (.15,-.4) to (.15, .4);
\symbox{-0.3,-0.15}{0.3,0.15};
\end{tikzpicture}\]
on the top and on the bottom we see that any term that has a cap or cup to the rightmost or to the leftmost will vanish, due to \cref{proj}. Consequently
\begin{equation}\label{pharaon}
\Phi(\crossmor) = \begin{tikzpicture}[baseline=-6,scale=.8]
	\draw[out=90,in=-90]  (-.6,-.6) to (.3,.6);
	\draw[out=90,in=-90]  (-.3,-.6) to (.6,.6);
	\draw[out=90,in=-90]  (.6,-.6) to (-.3,.6);
	\draw[out=90,in=-90]  (.3,-.6) to (-.6,.6);
	\draw (-.6,-.9) to (-.6,-.8);
	\draw (-.3,-.9) to (-.3,-.8);
	\draw (.6,-.9) to (.6,-.8);
	\draw (.3,-.9) to (.3,-.8);
	\symbox{-.7,-.6}{-.2,-.8};
	\symbox{.7,-.6}{.2,-.8};
\end{tikzpicture} = \begin{tikzpicture}[centerzero]
\draw (-.15,-.4) to (-.15,.4);
\draw (.15,-.4) to (.15, .4);
\symbox{-0.3,-0.15}{0.3,0.15};
\end{tikzpicture}\, \begin{tikzpicture}[centerzero]
\draw (-.15,-.4) to (-.15,.4);
\draw (.15,-.4) to (.15, .4);
\symbox{-0.3,-0.15}{0.3,0.15};
\end{tikzpicture} + 2\, \begin{tikzpicture}[centerzero]
\draw  (-.45, -.4) to (-.45,.4);
\draw  (.45,-.4) to (.45, .4);
\draw[out=-90,in=-90,looseness=2] (-.15,.4) to (.15,.4);
\draw[out=90,in=90,looseness=2](-.15,-.4) to (.15, -.4);
\draw(-.15,-.6) to (-.15,-.7);
\draw(0.15,-.6) to (0.15,-.7);
\draw(-.45,-.6) to (-.45,-.7);
\draw(0.45,-.6) to (0.45,-.7);
\draw(-.15,.6) to (-.15,.7);
\draw(0.15,.6) to (0.15,.7);
\draw(-.45,.6) to (-.45,.7);
\draw(0.45,.6) to (0.45,.7);
\symbox{-0.55,-0.6}{-.05,-.4};
\symbox{0.05,-0.6}{.55,-.4};
\symbox{-0.55,0.6}{-.05,.4};
\symbox{0.05,0.6}{.55,.4};
\end{tikzpicture}  
+
\begin{tikzpicture}[centerzero]
	\draw[out=90, in=90, looseness=1.2] (-.45, -.4) to (.45,-.4);
	\draw[out=90,in=90,looseness=1.2] (-.15,-.4) to (.15, -.4);
	\draw (-.45, -.4) to (-.45,-.7);
	\draw (.45, -.4) to (.45,-.7);
	\draw (-.15, -.4) to (-.15,-.7);
	\draw (.15, -.4) to (.15,-.7);
	\draw[out=-90, in=-90, looseness=1.2] (-.45, .4) to (.45,.4);
	\draw[out=-90,in=-90,looseness=1.2] (-.15,.4) to (.15, .4);
	\draw (-.45, .4) to (-.45,.7);
	\draw (.45, .4) to (.45,.7);
	\draw (-.15, .4) to (-.15,.7);
	\draw (.15, .4) to (.15,.7);
	\symbox{.05,-0.6}{0.55,-.4};
	\symbox{-0.55,-0.6}{-.05,-.4};
	\symbox{.05,0.6}{0.55,.4};
	\symbox{-0.55,0.6}{-.05,.4};
\end{tikzpicture}
\overset{\cref{idemp}}{=} \begin{tikzpicture}[centerzero]
	\draw (-.15,-.4) to (-.15,.4);
	\draw (.15,-.4) to (.15, .4);
	\symbox{-0.3,-0.15}{0.3,0.15};
\end{tikzpicture}\;\begin{tikzpicture}[centerzero]
\draw (-.15,-.4) to (-.15,.4);
\draw (.15,-.4) to (.15, .4);
\symbox{-0.3,-0.15}{0.3,0.15};
\end{tikzpicture}
+
2\, \begin{tikzpicture}[baseline=6]
	\draw[out=90, in=-90, looseness=1] (-.45, -.4) to (-.15,.2);
	\draw[out=90,in=-90,looseness=1] (.45,-.4) to (.15, .2);
	\draw[out=90,in=90,looseness=2] (-.15,-.4) to (.15, -.4);
	\draw[out=90, in=-90, looseness=1] (-.15, .4) to (-.45,1);
	\draw[out=90, in=-90, looseness=1] (.15, .4) to (.45,1);
	\draw[out=-90,in=-90,looseness=2] (-.15,1) to (.15, 1);
	\symbox{-0.25,0.2}{0.25,0.4};
	\symbox{.05,-0.6}{0.55,-.4};
	\symbox{-0.55,-0.6}{-.05,-.4};
	\symbox{.05,1}{0.55,1.2};
	\symbox{-0.55,1}{-.05,1.2};
	\draw (-.15,1.2) to (-.15,1.3);
	\draw (-.45,1.2) to (-.45,1.3);
	\draw (.15,1.2) to (.15,1.3);
	\draw (.45,1.2) to (.45,1.3);
	\draw (-.15,-.6) to (-.15,-.7);
	\draw (-.45,-.7) to (-.45,-.6);
	\draw (.15,-.6) to (.15,-.7);
	\draw (.45,-.6) to (.45,-.7);
\end{tikzpicture} = \Phi\left(\jail - 2\; \Imor\right),
\end{equation}
as desired. The third relation in \cref{SL2} can be derived in a similar manner. Finally for the first relation in \cref{SL2}, we have
\begin{multline*}
\Phi(\dotcross)=  \Phi\left(\begin{tikzpicture}[centerzero]
	\draw (-0.2,-0.4) -- (-0.2,0.4);
	\draw (-0.2,-0.2) -- (0.4,0.4);
	\draw (0.4,-0.4) -- (-0.2,0.2);
\end{tikzpicture}\right) 
= - \; \begin{tikzpicture}[centerzero]
	\draw[out=90,in=-90] (-.2,-.6) to (-.4,0);
	\draw[out=90,in=-90] (-.4,0) to (-.2,.6);
	\draw[out=-90,in=90] (0,.6) to (1,-.6);
	\draw[out=90,in=-90] (0,-.6) to (1,.6);
	\draw[out=-60,in=90] (.2,.0) to (.8,-.6);
	\draw[out=60,in=-90] (.2,0) to (.8,.6);
	\draw[out=120,in=0,looseness=1] (.2,0) to (0,.2);
	\draw[out=-120,in=0,looseness=1] (.2,0) to (0,-.2);
	\draw[out=180,in=-90,looseness=1] (0,-.2) to (-.2,0);
	\draw[out=90,in=180,looseness=1] (-.2,0) to (0,0.2);
	\draw (0,.8) to (0,.9);
	\draw (-.2,.8) to (-.2,.9);
	\draw (1,.8) to (1,.9);
	\draw (.8,.8) to (.8,.9);
	\draw (0,-.8) to (0,-.9);
	\draw (-.2,-.8) to (-.2,-.9);
	\draw (1,-.8) to (1,-.9);
	\draw (.8,-.8) to (.8,-.9);
	\symbox{-.3,-.6}{.1,-.8};
	\symbox{.7,-.6}{1.1,-.8};
	\symbox{-.3,.6}{.1,.8};
	\symbox{.7,.6}{1.1,.8};
	\symbox{-.5,-.1}{-.1,.1};
\end{tikzpicture}
\overset{\cref{idemp}}{=}
- \; \begin{tikzpicture}[centerzero]
	\draw[out=90,in=-90] (-.2,-.6) to (-.4,0);
	\draw[out=90,in=-90] (-.4,0) to (-.2,.6);
	\draw[out=-90,in=90] (0,.6) to (1,-.6);
	\draw[out=90,in=-90] (0,-.6) to (1,.6);
	\draw[out=-60,in=90] (.2,.0) to (.8,-.6);
	\draw[out=60,in=-90] (.2,0) to (.8,.6);
	\draw[out=120,in=0,looseness=1] (.2,0) to (0,.2);
	\draw[out=-120,in=0,looseness=1] (.2,0) to (0,-.2);
	\draw[out=180,in=-90,looseness=1] (0,-.2) to (-.2,0);
	\draw[out=90,in=180,looseness=1] (-.2,0) to (0,0.2);
	\draw (0,.8) to (0,.9);
	\draw (-.2,.8) to (-.2,.9);
	\draw (1,.8) to (1,.9);
	\draw (.8,.8) to (.8,.9);
	\draw (0,-.8) to (0,-.9);
	\draw (-.2,-.8) to (-.2,-.9);
	\draw (1,-.8) to (1,-.9);
	\draw (.8,-.8) to (.8,-.9);
	\symbox{-.3,-.6}{.1,-.8};
	\symbox{.7,-.6}{1.1,-.8};
	\symbox{-.3,.6}{.1,.8};
	\symbox{.7,.6}{1.1,.8};
\end{tikzpicture}
  - 
   \frac{1}{2}\; \begin{tikzpicture}[baseline=-6,scale=.8]
  	\draw[out=90,in=-90]  (-.6,-.6) to (.3,.6);
  	\draw[out=90,in=-90]  (-.3,-.6) to (.6,.6);
  	\draw[out=90,in=-90]  (.6,-.6) to (-.3,.6);
  	\draw[out=90,in=-90]  (.3,-.6) to (-.6,.6);
  	\draw (-.6,-.9) to (-.6,-.8);
  	\draw (-.3,-.9) to (-.3,-.8);
  	\draw (.6,-.9) to (.6,-.8);
  	\draw (.3,-.9) to (.3,-.8);
  	\symbox{-.7,-.6}{-.2,-.8};
  	\symbox{.7,-.6}{.2,-.8};
  \end{tikzpicture}
\overset{\cref{twist}}{=}
 \begin{tikzpicture}[centerzero]
	\draw[out=90,in=-90] (-.2,-.6) to (-.2,0.6);
	\draw[out=-90,in=90] (0,.6) to (1,-.6);
	\draw[out=90,in=-90] (0,-.6) to (1,.6);
	\draw[out=-90,in=90] (.2,.0) to (.8,-.6);
	\draw[out=90,in=-90] (.2,0) to (.8,.6);
	\draw (0,.8) to (0,.9);
	\draw (-.2,.8) to (-.2,.9);
	\draw (1,.8) to (1,.9);
	\draw (.8,.8) to (.8,.9);
	\draw (0,-.8) to (0,-.9);
	\draw (-.2,-.8) to (-.2,-.9);
	\draw (1,-.8) to (1,-.9);
	\draw (.8,-.8) to (.8,-.9);
	\symbox{-.3,-.6}{.1,-.8};
	\symbox{.7,-.6}{1.1,-.8};
	\symbox{-.3,.6}{.1,.8};
	\symbox{.7,.6}{1.1,.8};
\end{tikzpicture}
- 
\frac{1}{2}\; \begin{tikzpicture}[baseline=-6,scale=.8]
	\draw[out=90,in=-90]  (-.6,-.6) to (.3,.6);
	\draw[out=90,in=-90]  (-.3,-.6) to (.6,.6);
	\draw[out=90,in=-90]  (.6,-.6) to (-.3,.6);
	\draw[out=90,in=-90]  (.3,-.6) to (-.6,.6);
	\draw (-.6,-.9) to (-.6,-.8);
	\draw (-.3,-.9) to (-.3,-.8);
	\draw (.6,-.9) to (.6,-.8);
	\draw (.3,-.9) to (.3,-.8);
	\symbox{-.7,-.6}{-.2,-.8};
	\symbox{.7,-.6}{.2,-.8};
\end{tikzpicture}
\\
\overset{\cref{fence}}{=}
\begin{tikzpicture}[centerzero]
	\draw (-.15,-.4) to (-.15,.4);
	\draw (.15,-.4) to (.15, .4);
	\symbox{-0.3,-0.15}{0.3,0.15};
\end{tikzpicture}\, \begin{tikzpicture}[centerzero]
	\draw (-.15,-.4) to (-.15,.4);
	\draw (.15,-.4) to (.15, .4);
	\symbox{-0.3,-0.15}{0.3,0.15};
\end{tikzpicture}+
\begin{tikzpicture}[centerzero]
	\draw[out=90,in=-90] (-.2,-.4) to (-.2,0.4);
	\draw[out=90,in=90] (0,-.4) to (.6,-.4);
	\draw[out=-90,in=-90] (0,.4) to (.6,.4);
	\draw[out=-90,in=90] (.2,.0) to (.4,-.4);
	\draw[out=90,in=-90] (.2,0) to (.4,.4);
	\draw (0,.6) to (0,.7);
	\draw (-.2,.6) to (-.2,.7);
	\draw (.6,.6) to (.6,.7);
	\draw (.4,.6) to (.4,.7);
	\draw (0,-.6) to (0,-.7);
	\draw (-.2,-.6) to (-.2,-.7);
	\draw (.6,-.6) to (.6,-.7);
	\draw (.4,-.6) to (.4,-.7);
	\symbox{-.3,-.4}{.1,-.6};
	\symbox{.3,-.4}{.7,-.6};
	\symbox{-.3,.4}{.1,.6};
	\symbox{.3,.4}{.7,.6};
\end{tikzpicture}
- 
\frac{1}{2}\; \begin{tikzpicture}[baseline=-6,scale=.8]
	\draw[out=90,in=-90]  (-.6,-.6) to (.3,.6);
	\draw[out=90,in=-90]  (-.3,-.6) to (.6,.6);
	\draw[out=90,in=-90]  (.6,-.6) to (-.3,.6);
	\draw[out=90,in=-90]  (.3,-.6) to (-.6,.6);
	\draw (-.6,-.9) to (-.6,-.8);
	\draw (-.3,-.9) to (-.3,-.8);
	\draw (.6,-.9) to (.6,-.8);
	\draw (.3,-.9) to (.3,-.8);
	\symbox{-.7,-.6}{-.2,-.8};
	\symbox{.7,-.6}{.2,-.8};
\end{tikzpicture}
\overset{\cref{fence}}{\underset{\cref{proj}}{=}}
\begin{tikzpicture}[centerzero]
	\draw (-.15,-.4) to (-.15,.4);
	\draw (.15,-.4) to (.15, .4);
	\symbox{-0.3,-0.15}{0.3,0.15};
\end{tikzpicture}\, \begin{tikzpicture}[centerzero]
	\draw (-.15,-.4) to (-.15,.4);
	\draw (.15,-.4) to (.15, .4);
	\symbox{-0.3,-0.15}{0.3,0.15};
\end{tikzpicture}+
\begin{tikzpicture}[centerzero]
	\draw  (-.45, -.4) to (-.45,.4);
	\draw  (.45,-.4) to (.45, .4);
	\draw[out=-90,in=-90,looseness=2] (-.15,.4) to (.15,.4);
	\draw[out=90,in=90,looseness=2](-.15,-.4) to (.15, -.4);
	\draw(-.15,-.6) to (-.15,-.7);
	\draw(0.15,-.6) to (0.15,-.7);
	\draw(-.45,-.6) to (-.45,-.7);
	\draw(0.45,-.6) to (0.45,-.7);
	\draw(-.15,.6) to (-.15,.7);
	\draw(0.15,.6) to (0.15,.7);
	\draw(-.45,.6) to (-.45,.7);
	\draw(0.45,.6) to (0.45,.7);
	\symbox{-0.55,-0.6}{-.05,-.4};
	\symbox{0.05,-0.6}{.55,-.4};
	\symbox{-0.55,0.6}{-.05,.4};
	\symbox{0.05,0.6}{.55,.4};
\end{tikzpicture}  
- 
\frac{1}{2}\; \begin{tikzpicture}[baseline=-6,scale=.8]
	\draw[out=90,in=-90]  (-.6,-.6) to (.3,.6);
	\draw[out=90,in=-90]  (-.3,-.6) to (.6,.6);
	\draw[out=90,in=-90]  (.6,-.6) to (-.3,.6);
	\draw[out=90,in=-90]  (.3,-.6) to (-.6,.6);
	\draw (-.6,-.9) to (-.6,-.8);
	\draw (-.3,-.9) to (-.3,-.8);
	\draw (.6,-.9) to (.6,-.8);
	\draw (.3,-.9) to (.3,-.8);
	\symbox{-.7,-.6}{-.2,-.8};
	\symbox{.7,-.6}{.2,-.8};
\end{tikzpicture}
\overset{\cref{pharaon}}{=}
\frac{1}{2} \left(\begin{tikzpicture}[centerzero]
	\draw (-.15,-.4) to (-.15,.4);
	\draw (.15,-.4) to (.15, .4);
	\symbox{-0.3,-0.15}{0.3,0.15};
\end{tikzpicture}\, \begin{tikzpicture}[centerzero]
	\draw (-.15,-.4) to (-.15,.4);
	\draw (.15,-.4) to (.15, .4);
	\symbox{-0.3,-0.15}{0.3,0.15};
\end{tikzpicture}
-
\begin{tikzpicture}[centerzero]
	\draw[out=90, in=90, looseness=1.2] (-.45, -.4) to (.45,-.4);
	\draw[out=90,in=90,looseness=1.2] (-.15,-.4) to (.15, -.4);
	\draw (-.45, -.4) to (-.45,-.7);
	\draw (.45, -.4) to (.45,-.7);
	\draw (-.15, -.4) to (-.15,-.7);
	\draw (.15, -.4) to (.15,-.7);
	\draw[out=-90, in=-90, looseness=1.2] (-.45, .4) to (.45,.4);
	\draw[out=-90,in=-90,looseness=1.2] (-.15,.4) to (.15, .4);
	\draw (-.45, .4) to (-.45,.7);
	\draw (.45, .4) to (.45,.7);
	\draw (-.15, .4) to (-.15,.7);
	\draw (.15, .4) to (.15,.7);
	\symbox{.05,-0.6}{0.55,-.4};
	\symbox{-0.55,-0.6}{-.05,-.4};
	\symbox{.05,0.6}{0.55,.4};
	\symbox{-0.55,0.6}{-.05,.4};
\end{tikzpicture}\right) \\
 =\frac{1}{2} \Phi\left(\jail -\hourglass \right)\overset{\cref{SL2}}{=}\Phi\left(\Imor -\Hmor \right).
\end{multline*}
 This completes the proof of the proposition.
\end{proof}

The functor in \cref{Kauff} is not faithful because $\dim  \dim\Hom_{\mathpzc{SL}} (\go^{\otimes 2},\go^{\otimes 3})\geq 8$, while $\dim \Kar \TL(-2) \left(\Phi(\go^{\otimes 2}),\Phi(\go^{\otimes 3})\right)=6$; see \cref{fickle} below.
\begin{rem}
The definition of a trivalent vertex in terms of diagrams in the Temperley--Lieb algebra is due to Kauffman and Lins \cite[Ch. 3]{KL94}, where they use these trivalent vertices to construct invariants of $3$-manifolds.
\end{rem}
\begin{prop}\label{babel} If $\cI$ is a tensor ideal of $\Tcat_{\delta,\alpha}$ such that
\[
\dim \left( \left(\Tcat_{\delta,\alpha}/\cI\right)(\go^{\otimes 2}, \go^{\otimes 2}) \right) = 4,
\] 
then $ \delta\neq 1$ and the relation 
\begin{equation}\label{viktor1}
	\Imor + \Hmor = \frac{\alpha}{1-\delta}\left( 2\,  \crossmor -\, \jail -\, \hourglass \right).
\end{equation}
holds in the quotient $ \Tcat_{\delta,\alpha}/\cI $. In particular $\go$ is a vector product algebra in the category $\Tcat_{\delta,\alpha}$.
\end{prop}
\begin{proof}
There is a nontrivial linear combination 
\begin{equation}\label{egg}
	a_1\, \crossmor+ a_2\,\jail+a_3\,\hourglass+a_4\,  \Imor+a_5\, \Hmor =0\qquad a_1,a_2,a_3,a_4,a_5 \in \kk.
 \end{equation}
Applying  $ 1+\Rot  $ to this we get
\begin{equation}\label{lulu}
	2a_1\, \crossmor+ \left(a_2+a_3\right) \left(\jail+\,\hourglass\right)+\left(a_4+a_5  \right) \left(\Imor+\, \Hmor\right)=0. 
\end{equation}

We split the discussion into two cases according to whether $a_1=0$ or $a_1\neq 0$.  First assume $a_1=0$, and suppose also that $a_2+a_3\neq 0$. In this case, Equation \cref{lulu} takes the form
\begin{equation}\label{dina}
\jail+\hourglass =\lambda_1\left(\Imor+\Hmor\right),\quad \lambda_1 \in \kk. 
\end{equation}
It follows from \cref{jhcmor} that $\lambda_1\neq0$. On the other hand, we compute 
\begin{multline*}
-\lambda_1^{-1}\left(\jail+\hourglass\right)=	-\Imor -\Hmor = -\Imor -\dotcross -\Hmor+\dotcross = \left(\Switch +\Rot\circ\Switch\right)\left(\Imor+\Hmor\right)\\
\overset{\cref{dina}}{=}\lambda_1^{-1}\left( \crossmor + \hourglass\right)
+\lambda_1^{-1}\left(\crossmor+\jail\right)
	= \lambda_1^{-1}\left(2\;\crossmor+\jail +\hourglass\right).
\end{multline*}
Hence 
\[
\crossmor + \jail +\hourglass=0.
\] 
This is impossible by \cref{jhcmor}. Thus, if $a_1=0$, then $a_2+a_3=0$. A similar argument shows that if $a_1=0$ then $a_4+a_5=0$. Consequently, if $a_1=0$, then  \cref{egg} must be of the form
\begin{equation}\label{copper}
 	\Imor -\Hmor =\lambda_2\left(\jail -\hourglass\right),
 	\quad \lambda_2 \in\kk. 
 \end{equation}
We will show that \cref{copper} gives $\go$ the structure of a vector product algebra object in the in $\Tcat_{\delta,\alpha}/\cI$. We compute
\begin{multline*}
\Imor -\Hmor -2\,\dotcross   \overset{\cref{flick}}{\underset{\cref{rotary}}{=}} \left(\Rot -1\right)\circ\Switch\left(\Imor-\Hmor\right) 
 \overset{\cref{copper}}{=} \lambda_2\left(\Rot -1\right)\circ\Switch\left(\jail-\hourglass\right)\\
  = -\lambda_2\left(\jail -\hourglass\right) 
 \overset{\cref{copper}}{=} \Hmor -\Imor.
\end{multline*}
This gives 
\begin{equation}\label{ouafa}
	\dotcross = \Imor -\Hmor\overset{\cref{copper}}{=}\lambda_2\left(\jail -\hourglass\right).
\end{equation}
Applying $\Switch$ to \cref{ouafa} yields
\begin{equation}\label{gold}
	-\Hmor = \lambda_2\left(\crossmor - \hourglass\right) \quad \text{and} \quad -\Imor =\lambda_2 \left(\crossmor -\jail\right).
\end{equation}
Adding the two equations in \cref{gold} gives
\begin{equation}\label{silver}
	\Imor + \Hmor = -\lambda_2 \left( 2\,  \crossmor -\, \jail -\, \hourglass \right).
\end{equation}
To find $\lambda_2$, we compose \cref{silver},  on the top, with $\capmor$, to get $\alpha\,  \capmor = -\lambda_2\left( 1- \delta\right)\capmor $. It follows, as in \cref{bird}, that $\delta\neq 1$ and  $\displaystyle \lambda_2=\frac{\alpha}{\delta-1}$.  Now, define  $\mergemor\, '=\mergemor\, ,\; \capmor\, '=\frac{1-\delta}{\alpha}\,\capmor$ and $\cupmor\, '=\frac{\alpha }{1-\delta}\cupmor$. Then it can be easily verified that $\mergemor\, ',\capmor\, '$ and $\cupmor\, '$  make $\go$ into a vector product algebra object, in the sense of \cref{vcpdef}, in the category~$\Tcat_{\delta,\alpha}$. This concludes the  $a_1=0$ case in \cref{egg}.

  Next, we consider the other case, in which $a_1\neq 0$.   Using $ 1+\Rot$, we can rewrite \cref{egg} in the form
\begin{equation}\label{feather}
	\crossmor = \lambda \left(\jail +\hourglass\right)+\mu \left(\Imor+\Hmor\right), \quad \text{for }\lambda,\mu \in \kk.
\end{equation}
To find $\lambda$ and $\mu$, we compute
\begin{equation}\label{beak1}
	\jail = \Switch\left(\crossmor\right)\overset{\cref{feather}}{=}\lambda \left(\crossmor+\; \hourglass \right)+\mu\left(-\Imor -\dotcross\right),
\end{equation}
and 
\begin{equation}\label{beak2}
	\hourglass =\Rot\left(\jail\right)\overset{\cref{beak1}}{=}\lambda\Rot\left(\crossmor+\, \hourglass\right)+\mu\Rot \left(-\Imor-\dotcross\right)=\lambda \left(\crossmor+\, \jail\right)+\mu\left(-\Hmor+\dotcross\right).
\end{equation}
Adding \cref{beak1} to \cref{beak2} gives 
\[ 
\jail+\,\hourglass=\lambda\left(\jail +\hourglass+2 \, \crossmor \right)-\mu \left( \Imor+\, \Hmor\right)\overset{\cref{feather}}{=}2\lambda\left(\jail +\hourglass \right)+ (2\lambda-1)\,\crossmor.
\]
By \cref{jhcmor}, this gives $\lambda =\frac{1}{2}$. Composing \cref{feather} on the top with $\capmor$ and using \cref{chess} gives~$ \mu = \frac{1-\delta}{2\alpha}$. That is 
\begin{equation}\label{night}
	\crossmor =\frac{1}{2}\left( \jail +\; \hourglass \right) +\frac{1-\delta}{\alpha}\left( \Imor +\; \Hmor\right).
\end{equation}
 Note that $\delta\neq 1$ by \cref{jhcmor}. Notice also that \cref{night} is equivalent to \cref{silver}, and so we get a solution that is a vector product algebra object, as in the previous case.
\end{proof}
Let $\mathpzc{V}$ be additive Karoubi envelope of the monoidal category obtained from $\Tcat_{7,-6}$ by modding out by the two-sided tensor ideal generated by \cref{viktor1}. Take $\kk$ to be the field of real numbers, and let $G_2$ be the compact real exceptional Lie group of rank $2$. Let $\OO$ be the algebra of classical octonions. As a Euclidean space, $\OO$ is identified with $\R^8$ equipped with the usual inner product. The imaginary part of $\OO$,  
\[
\Im(\OO)=\{x\in \OO|\; \langle 1_{\OO},x\rangle =0\}
\]
is a vector product algebra, with cross product given by  $x\times y = xy -\langle x, y\rangle \cdot 1_{\OO}$. This is the unique real vector product algebra of dimension $7$. Moreover, $\Im (\OO)$ can be identified with the minimal fundamental representation of $G_2$. The following proposition exploits the connection between $G_2$ and $\Im(\OO)$.
\begin{prop}\label{G2}
There is a full, essentially surjective functor $\mathpzc{V}\to G_2\md$, given on objects by sending $\go$ to $\Im(\OO)$ and on objects by sending $\capmor$ to the inner product $\langle-,-\rangle$, and $\mergemor$ to the cross product.
\end{prop}
\begin{proof}
The group of real linear transformations that preserve the inner product and the cross product on $\Im(\OO)$ is exactly $G_2$. This result is stated without proof in \cite[Proposition~1]{Bryant}, but a proof can be found in \cite[Theorem~6.80]{Harvey}. This gives a functor $\Tcat_{7,-6}\to G_2\md$  that sends $\go$ to $\Im(\OO)$.  Since $\dim \End_{G_2}(\Im(\OO)^{\otimes 2})=4$, it follows, by \cref{babel}, that \cref{viktor1} holds in $G_2\md$, and so this functor factors through $\mathpzc{V}$. The proof of fullness and essential surjectivty is similar to the proof of \cref{full}. 
\end{proof}
\begin{rem}
		Note that the functor in \cref{G2} is different from the one in \cref{sec:functor} in that in \cref{G2}, the generating object is sent to the minimal fundamental representation, while in \cref{sec:functor} it is sent to the adjoint representation. 
\end{rem}
\begin{rem}
Kuperberg \cite{Kup96} gave  a full presentation of spider categories of rank $2$ quantum groups of types $A_2,B_2$ and $G_2$. His category, in type $G_2$, at $q=1$, is a quotient of $\mathpzc{V}$, and is equivalent to $G_2\md$.
\end{rem}

\begin{prop}\label{dilum}
	If $\cI$ is a tensor ideal of $\Tcat_{\delta,\alpha}$ such that
	\begin{equation} \label{cinco}
	  \dim \left( \left(\Tcat_{\delta,\alpha}/\cI\right)(\go^{\otimes 2}, \go^{\otimes 2}) \right) = 5,
	\end{equation}
	then  $\go$ is either a Lie algebra object or a  vector product algebra object in $\Tcat_{\delta,\alpha}/\cI$. 
\end{prop}

\begin{proof}
	By \cref{cinco}, there is a nontrivial linear dependence relation between the morphisms
	\[
	\jail\quad,\quad \hourglass\quad ,\quad\Imor\quad ,\quad \Hmor\quad,\quad \crossmor \quad\text{and}\quad \dotcross.
	\]
	
	\textbf{Case 1}\label{case1}: Suppose  we have a relation of the form 
	\[
	\dotcross = a\;\jail+b\; \hourglass +c\; \Imor+d\; \Hmor +e\; \crossmor.
	\]
	From this it follows that
	\begin{equation}\label{Gerra}
		\dotcross\overset{\cref{turvy}}{=} \frac{1}{2}\left(1-\Rot\right)\left(\dotcross\right) = \lambda\left(\jail- \;\hourglass\right)+ \mu \left(\Imor-\; \Hmor\right), 
	\end{equation}
	or equivalently (by \cref{turvy})
	\begin{equation}\label{pyramid}
		-\begin{tikzpicture}[centerzero]
			\draw (-0.2,0.2) -- (0.2,0.2);
			\draw (0.3,-0.3) -- (-0.3,0.3);
			\draw (-0.3,-0.3) -- (0.3,0.3);
		\end{tikzpicture}
		=
		\lambda\left(\jail- \;\hourglass\right)+ \mu \left(\Imor-\; \Hmor\right), 
	\end{equation}
	with $ \lambda = (a-b)/2$ and $\mu = (c-d)/2 $. Precomposing \cref{pyramid} with $\cupmor\,$, and using \cref{chess,teardrop,bird}, we get 
	\begin{equation}\label{lat}
		\alpha(1-\mu)= \lambda(\delta-1).
	\end{equation} 
	We apply $\Switch$ to \cref{Gerra} to get 
	\begin{equation}\label{Babel}
		-\Hmor = \lambda \left(\crossmor - \; \hourglass \right)+\mu \left(-\Imor +\; \dotcross\right). 
	\end{equation}
	Adding \cref{Gerra} to \cref{Babel} gives 
	\[
	\dotcross -\,\Hmor=\lambda\left(\jail+\, \crossmor -2\, \hourglass \right)+\mu\left(\dotcross -\,\Hmor\right).
	\]
	If $\mu=1$, then it follows from \cref{jhcmor} that $\lambda=0$, and so  \cref{Babel} becomes the Jacobi identity \cref{Jacob}, giving us the Lie algebra object solution in this case. Now, assume $\mu\neq 1$. This assumption implies, by \cref{lat}, that $\delta\neq 1$. We then multiply \cref{Gerra} by $\mu$, add it to \cref{Babel}, and use \cref{lat}, to get 
	\[
	\Imor +\; \Hmor = \frac{\alpha}{1-\delta}\left(2\; \crossmor - \; \jail -\; \hourglass\right).
	\]
	% Define a monoidal $\kk$-linear functor $ \Psi \colon\Tcat_{\delta,\alpha}\to\Tcat_{\delta,1-\delta} $ by 
	% \[ \mergemor\mapsto \mergemor\, ,\; \splitmor  \mapsto \frac{1-\delta}{\alpha}\,\splitmor\, ,\; \capmor\mapsto \frac{1-\delta}{\alpha}\,\capmor\, ,\; \cupmor\mapsto \frac{\alpha }{1-\delta}\cupmor\,,\; \crossmor \mapsto\crossmor.\]
	As in the proof of \cref{babel}, this makes $\go$ into a vector product algebra object in the category~$\Tcat_{\delta,\alpha}$.

\textbf{Case 2}: Now suppose we have a nontrivial relation of the form 
	\[
	a\;\jail+b\; \hourglass +c\; \Imor+d\; \Hmor +e\; \crossmor =0.
	\]
The space $\cI(\go^{\otimes 2},\go^{\otimes 2})$ is invariant under $\Rot$ (recall the definitions in \cref{QTC}), and decomposes into two eigenspaces, corresponding to the two eigenvalues, of $\Rot$,  $1$ and $-1$. 	 The first subcase corresponds to the eigenvalue $1$ and gives 
	\begin{equation}\label{sweet}
		u \left(\jail+\; \hourglass\right)+v\left(\Imor+\; \Hmor\right)+w \; \crossmor= 0,
	\end{equation}
	for some $u,v,w\in \kk$ not all zero, and the second subcase corresponds to the eigenvalue $-1$ and gives  
	\begin{equation}\label{bitter}
		u \left(\jail-\; \hourglass\right)+v\left(\Imor-\; \Hmor\right)= 0,
	\end{equation}
	for $u,v\in \kk$ not both zero. By virtue of \cref{jhcmor}, we must have $v\neq 0$ in both subcases \cref{sweet,bitter}. Then applying $\Switch$ to either \cref{sweet} or \cref{bitter} will bring us back to 	\hyperref[case1]{\textbf{Case~1}}.
\end{proof}
%=======================================
\section{The exceptional Lie algebra case \label{sec4}}
%=========================================
For the remainder of the paper, we focus on the case where the dimension in \cref{cinco} is $5$. We will  fix the ground ring $\kk$ to be the field of complex numbers. Let $\alpha=1$ and let be $\delta$ be any complex number.  We will write $\Tcat_{\delta}$ instead of $\Tcat_{\delta,1}$.
\begin{defin} \label{Fdef}
Let $\Fcat_{\delta}$ be the strict monoidal category obtained from $\Tcat_{\delta}$ by imposing the following three additional relations:
\begin{gather} \label{magic}
		\dotcross
		= \Imor - \Hmor \ ,
		\\ \label{sqburst}
		\sqmor
		=\frac{1}{6} \left(\, \Hmor + \Imor\, \right)
		+ 
		\frac{5}{6(\delta+2)} \left(\, \jail + \hourglass+\crossmor \, \right) \ ,
		\\ \label{pentburst}
		\begin{aligned}
			\pentmor &=
			\frac{1}{12}	\left(\;	\begin{tikzpicture}[anchorbase]
				\draw (-0.2,0) -- (0,0.25) -- (0.2,0);
				\draw (0,0.25) -- (0,0.4);
				\draw (-0.2,-0.25) -- (-0.2,0) -- (-0.3,0.4);
				\draw (0.2,-0.25) -- (0.2,0) -- (0.3,0.4);
			\end{tikzpicture}
			+
			\begin{tikzpicture}[anchorbase]
				\draw (-0.3,0.3) -- (0,0) -- (0.3,0.3);
				\draw (0,0.3) -- (-0.15,0.15);
				\draw (0,0) -- (0,-0.15) -- (-0.15,-0.3);
				\draw (0,-0.15) -- (0.15,-0.3);
			\end{tikzpicture}
			+
			\begin{tikzpicture}[anchorbase]
				\draw (0.3,0.3) -- (0,0) -- (-0.3,0.3);
				\draw (0,0.3) -- (0.15,0.15);
				\draw (0,0) -- (0,-0.15) -- (0.15,-0.3);
				\draw (0,-0.15) -- (-0.15,-0.3);
			\end{tikzpicture}+\begin{tikzpicture}[anchorbase]
				\draw (-0.2,0) -- (0,0.25) -- (0.2,0);
				\draw (0,0.25) -- (0,0.4);
				\draw  (-0.2,0) -- (-0.3,0.4);
				\draw (0.2,0) -- (0.3,0.4);
				\draw (0.2,0) -- (-0.3,-0.2);
				\draw  (-0.2,0) -- (0.3,-0.2);
			\end{tikzpicture}\; \right)
			\\
			&+ \frac{5}{12(\delta+2)}
			\left(
			\begin{tikzpicture}[centerzero]
				\draw (-0.15,-0.3) -- (-0.15,-0.23) arc(180:0:0.15) -- (0.15,-0.3);
				\draw (-0.3,0.3) -- (0,0.08) -- (0.3,0.3);
				\draw (0,0.3) -- (0,0.08);
			\end{tikzpicture}
			+
			\begin{tikzpicture}[centerzero]
				\draw (-0.2,-0.3) -- (-0.2,0.3);
				\draw (0,0.3) -- (0.15,0) -- (0.3,0.3);
				\draw (0.15,0) -- (0.15,-0.3);
			\end{tikzpicture}
			+
			\begin{tikzpicture}[centerzero]
				\draw (0.2,-0.3) -- (0.2,0.3);
				\draw (0,0.3) -- (-0.15,0) -- (-0.3,0.3);
				\draw (-0.15,0) -- (-0.15,-0.3);
			\end{tikzpicture}
			+
			\begin{tikzpicture}[centerzero]
				\draw (-0.3,0.3) -- (-0.3,0.23) arc(180:360:0.15) -- (0,0.3);
				\draw (0.3,0.3) -- (0.15,0) -- (-0.2,-0.3);
				\draw (0.2,-0.3) -- (0.15,0);
			\end{tikzpicture}
			+
			\begin{tikzpicture}[centerzero]
				\draw (0.3,0.3) -- (0.3,0.23) arc(360:180:0.15) -- (0,0.3);
				\draw (-0.3,0.3) -- (-0.15,0) -- (0.2,-0.3);
				\draw (-0.2,-0.3) -- (-0.15,0);
			\end{tikzpicture}
			+     \begin{tikzpicture}[centerzero]
				\draw (0,0.3) -- (0,-0.15) -- (-0.15,-0.3);
				\draw (0,-0.15) -- (0.15,-0.3);
				\draw (-0.2,0.3) -- (-0.2,0.25) arc(180:360:0.2) -- (0.2,0.3);
			\end{tikzpicture}
			+
			\begin{tikzpicture}[centerzero]
				\draw (0,0.3) to[out=-45,in=70] (0.15,-0.3);
				\draw (-0.3,0.3) -- (0,0) -- (0.3,0.3);
				\draw (0,0) -- (-0.15,-0.3);
			\end{tikzpicture}
			+
			\begin{tikzpicture}[centerzero]
				\draw (0,0.3) to[out=225,in=110] (-0.15,-0.3);
				\draw (0.3,0.3) -- (0,0) -- (-0.3,0.3);
				\draw (0,0) -- (0.15,-0.3);
			\end{tikzpicture}
			+
			\begin{tikzpicture}[centerzero]
				\draw (-0.3,0.3) -- (-0.15,0.15) -- (0,0.3);
				\draw (-0.15,0.15) -- (0.15,-0.3);
				\draw (0.3,0.3) -- (-0.15,-0.3);
			\end{tikzpicture}
			+
			\begin{tikzpicture}[centerzero]
				\draw (0.3,0.3) -- (0.15,0.15) -- (0,0.3);
				\draw (0.15,0.15) -- (-0.15,-0.3);
				\draw (-0.3,0.3) -- (0.15,-0.3);
			\end{tikzpicture}
			\right).
		\end{aligned}
\end{gather}
\end{defin}

\begin{lem} \label{spinner}
	If $\cI$ is a tensor ideal of $\Tcat_{\delta}$ such that relation \cref{magic} holds  in $\Tcat_{\delta}/\cI$,    then the relation
    \begin{equation} \label{triangle}
        \trimor = \frac{1}{2}\, \mergemor
    \end{equation}
     holds in $\Tcat_{\delta}/\cI$.
\end{lem}
\begin{proof}
    Relation \cref{triangle} follows by composing \cref{magic} on the top with $\mergemor$, then using \cref{chess,turvy}.
\end{proof}
\begin{lem} \label{SUP}
	Assume $\delta\notin \{-2,0,3\}$. If $\cI$ is a tensor ideal of $\Tcat_{\delta}$  such that $\Tcat_\delta/\cI$ is nontrivial 
and such that \cref{magic} holds in $\Tcat_{\delta}/\cI$,	then  the morphisms
\begin{equation} \label{bigfive}
		\jail\, ,\quad \hourglass\, ,\quad \crossmor\, ,\quad \Hmor\, ,\quad \Imor
\end{equation}
	are linearly independent in $(\Tcat_\delta/\cI)(\go^{\otimes 2}, \go^{\otimes 2})$.
\end{lem}
\begin{proof}
		Let $b\colon \Tcat_\delta/\cI(\go^{\otimes 2},\go^{\otimes 2})\times \Tcat_{\delta}/\cI(\go^{\otimes 2},\go^{\otimes 2})\to \C$ be the symmetric bilinear form such that $b(f,g)$ is the value, in $\C$, of the closed diagram
\[\begin{tikzpicture}[anchorbase]
	\draw (-0.4,0) -- (0.0,0) -- (0.0,0.4) -- (-0.4,0.4) -- (-0.4,0);
	\draw (0.6,0) -- (1,0) -- (1,0.4) -- (0.6,0.4) -- (0.6,0);
	%\draw[out=90,in =90,looseness=1.5] (-0.2,0.4) to (0.8,0.4);
	\draw[out=90,in =90,looseness=1] (-0.1,0.4) to (0.7,0.4);
	\draw[out=90,in =90,looseness=1] (-0.3,0.4) to (0.9,0.4);
	\draw[out=-90,in =-90,looseness=1] (-0.1,0) to (0.7,0);
	\draw[out=-90,in =-90,looseness=1] (-0.3,0) to (0.9,0);
	\node at (-0.2,0.2) {$\scriptstyle  f$};
	\node at (0.8,0.2) {$\scriptstyle  g$};	
\end{tikzpicture}\]
for $f,g \in \Tcat_\delta/\cI(\go^{\otimes 2},\go^{\otimes 2})$.

We compute the Gram matrix of $b$ with respect to the elements \cref{bigfive} using the relations in $\Tcat_{\delta}/\cI$ to get  \[\begin{pmatrix}
	\delta^2& \delta &\delta &0 &\delta \\
\delta	&\delta^2 &\delta &\delta & 0\\
\delta	&\delta &\delta^2 & -\delta& -\delta\\
0	&\delta & -\delta& \delta& \frac{1}{2} \delta\\
\delta	&0 & -\delta&\frac{1}{2} \delta & \delta 
\end{pmatrix}.\]
(An example of these computations is given in the proof of \cref{nugget}.) The determinant of this matrix is $\frac{3}{4} \delta^5 (\delta - 3)^2 (\delta + 2)1_\one$, and so the claim follows.
\end{proof}

\begin{lem} \label{sqexplode}
		Assume $\delta\notin \{-2,0,3\}$. If $\cI$ is a tensor ideal of $\Tcat_{\delta}$  such that
	\begin{equation} \label{funf}
		\dim \left( (\Tcat_\delta/\cI)(\go^{\otimes 2}, \go^{\otimes 2}) \right) = 5
	\end{equation}
	and such that \cref{magic} holds in $\Tcat_{\delta}/\cI$, then the relation \cref{sqburst} holds in $\Tcat_\delta/\cI$.
\end{lem}

\begin{proof}
    By the assumption \cref{funf} and \cref{SUP}, the morphisms \cref{bigfive} form a basis of $(\Tcat_\delta/\cI)(\go^{\otimes 2},\go^{\otimes 2})$. By \cref{spinner}, the relation \cref{triangle} also holds in $\Tcat_{\delta}/\cI$.  Since $\sqmor$ is invariant under $\Rot$, we must have a relation in $\Tcat_\delta/\cI$ of the form
\begin{equation} \label{sqbreak1}
        \sqmor =
        a\left(\, \jail + \hourglass \, \right)
        + b \left(\, \Hmor + \Imor\, \right)
        + c\, \crossmor\ .
\end{equation}
To this relation, we apply the operator $\Switch$ to get   \begin{multline}\label{dancer}
	\begin{tikzpicture}[centerzero]
	\draw (-0.15,-0.15) rectangle (0.15,0.15);
	\draw (0.15,-0.45) -- (-0.15,-0.15);
	\draw (-0.15,-0.45) -- (0.15,-0.15);
	\draw (-0.15,0.3) -- (-0.15,0.15);
	\draw (0.15,0.3) -- (0.15,0.15);
\end{tikzpicture} = \, a\left(\, \crossmor + \hourglass \, \right)
+ b \left(\, -\dotcross -\Imor\, \right)
+ c\, \jail\ \overset{\cref{magic}}{=} a\left(\, \crossmor + \hourglass \, \right)
-b \left(\,2\, \Imor -\Hmor\, \right)
+ c\, \jail\, .
\end{multline}
On the other hand, we have  
\[
	\begin{tikzpicture}[baseline=-4]
	\draw (-0.15,-0.15) rectangle (0.15,0.15);
	\draw (0.15,-0.45) -- (-0.15,-0.15);
	\draw (-0.15,-0.45) -- (0.15,-0.15);
	\draw (-0.15,0.3) -- (-0.15,0.15);
	\draw (0.15,0.3) -- (0.15,0.15);
\end{tikzpicture} \;\overset{\cref{turvy}}{=}\; -\,\begin{tikzpicture}[baseline=0]
	\draw (0.2,-0.2) -- (-0.2,0.2);
	\draw (-0.2,-0.2) -- (0.2,0.2);
	\draw[fill] (-0.08,0.08) -- (-0.08,-0.08)--(0,0);
	\draw (-0.2,0.2) -- (0.2,0.2);
	\draw (-0.2,0.2) -- (-0.2,0.4);
	\draw (0.2,0.2) -- (0.2,0.4);
\end{tikzpicture} \;\; \overset{\mathclap{\cref{magic}}}{\underset{\mathclap{\cref{triangle}}}{=}}\;\, \sqmor - \frac{1}{2}\Imor
\;\overset{\cref{sqbreak1}}{=} a\left(\, \jail + \hourglass \, \right)
+ b  \, \Hmor  
+ c\, \crossmor \,+ \left(b-\frac{1}{2}\right)\, \Imor\, .
 \]
Comparing this to \cref{dancer} and using the fact that the morphisms \cref{bigfive} are linearly independent, we see that $\displaystyle b= \frac{1}{6}$ and $a=c$. Now, attaching a cap on the top of \cref{sqbreak1} and using \cref{chess}, we get $\displaystyle a=c=\frac{5}{6(\delta +2)}$ as we want.
\end{proof}
\begin{lem}\label{nugget}
 If $\cI$ is a tensor ideal of $\Tcat_\delta$ such that $\Tcat_{\delta}/\cI$ is nontrivial and such that $\delta\notin \{-2,0,3,\frac{15\pm 3\sqrt{5}}{4} \}$,  then the elements 
\begin{equation}\label{particles}					\begin{tikzpicture}[anchorbase]
				\draw (-0.2,0) -- (0,0.25) -- (0.2,0);
				\draw (0,0.25) -- (0,0.4);
				\draw (-0.2,-0.25) -- (-0.2,0) -- (-0.3,0.4);
				\draw (0.2,-0.25) -- (0.2,0) -- (0.3,0.4);
			\end{tikzpicture}
		\,	, \;
			\begin{tikzpicture}[anchorbase]
				\draw (-0.3,0.3) -- (0,0) -- (0.3,0.3);
				\draw (0,0.3) -- (-0.15,0.15);
				\draw (0,0) -- (0,-0.15) -- (-0.15,-0.3);
				\draw (0,-0.15) -- (0.15,-0.3);
			\end{tikzpicture}
			\, ,\; 
			\begin{tikzpicture}[anchorbase]
				\draw (0.3,0.3) -- (0,0) -- (-0.3,0.3);
				\draw (0,0.3) -- (0.15,0.15);
				\draw (0,0) -- (0,-0.15) -- (0.15,-0.3);
				\draw (0,-0.15) -- (-0.15,-0.3);
			\end{tikzpicture}
		\,	, \; 
			\begin{tikzpicture}[anchorbase]
				\draw (-0.3,-0.3) -- (0.3,0.3);
				\draw (-0.3,0.3) -- (-0.15,-0.15);
				\draw (0,0.3) -- (0.15,0.15);
				\draw (0,0) -- (0.3,-0.3);
			\end{tikzpicture}
		\,	, \; 
			\begin{tikzpicture}[anchorbase]
				\draw (0.3,-0.3) -- (-0.3,0.3);
				\draw (0.3,0.3) -- (0.15,-0.15);
				\draw (0,0.3) -- (-0.15,0.15);
				\draw (0,0) -- (-0.3,-0.3);
			\end{tikzpicture}
		\,	, \; 
			\begin{tikzpicture}[centerzero]
				\draw (-0.15,-0.3) -- (-0.15,-0.23) arc(180:0:0.15) -- (0.15,-0.3);
				\draw (-0.3,0.3) -- (0,0.08) -- (0.3,0.3);
				\draw (0,0.3) -- (0,0.08);
			\end{tikzpicture}
		\,	, \;
			\begin{tikzpicture}[centerzero]
				\draw (-0.2,-0.3) -- (-0.2,0.3);
				\draw (0,0.3) -- (0.15,0) -- (0.3,0.3);
				\draw (0.15,0) -- (0.15,-0.3);
			\end{tikzpicture}
		\,	, \;
			\begin{tikzpicture}[centerzero]
				\draw (0.2,-0.3) -- (0.2,0.3);
				\draw (0,0.3) -- (-0.15,0) -- (-0.3,0.3);
				\draw (-0.15,0) -- (-0.15,-0.3);
			\end{tikzpicture}
		\,	, \;
			\begin{tikzpicture}[centerzero]
				\draw (-0.3,0.3) -- (-0.3,0.23) arc(180:360:0.15) -- (0,0.3);
				\draw (0.3,0.3) -- (0.15,0) -- (-0.2,-0.3);
				\draw (0.2,-0.3) -- (0.15,0);
			\end{tikzpicture}
		\,	, \;
			\begin{tikzpicture}[centerzero]
				\draw (0.3,0.3) -- (0.3,0.23) arc(360:180:0.15) -- (0,0.3);
				\draw (-0.3,0.3) -- (-0.15,0) -- (0.2,-0.3);
				\draw (-0.2,-0.3) -- (-0.15,0);
			\end{tikzpicture}
		\, ,\;
				\begin{tikzpicture}[centerzero]
				\draw (0,0.3) -- (0,-0.15) -- (-0.15,-0.3);
				\draw (0,-0.15) -- (0.15,-0.3);
				\draw (-0.2,0.3) -- (-0.2,0.25) arc(180:360:0.2) -- (0.2,0.3);
			\end{tikzpicture}
		\,	, \;
			\begin{tikzpicture}[centerzero]
				\draw (0,0.3) to[out=-45,in=70] (0.15,-0.3);
				\draw (-0.3,0.3) -- (0,0) -- (0.3,0.3);
				\draw (0,0) -- (-0.15,-0.3);
			\end{tikzpicture}
		\,	, \;
			\begin{tikzpicture}[centerzero]
				\draw (0,0.3) to[out=225,in=110] (-0.15,-0.3);
				\draw (0.3,0.3) -- (0,0) -- (-0.3,0.3);
				\draw (0,0) -- (0.15,-0.3);
			\end{tikzpicture}
	\,	, \;
			\begin{tikzpicture}[centerzero]
				\draw (-0.3,0.3) -- (-0.15,0.15) -- (0,0.3);
				\draw (-0.15,0.15) -- (0.15,-0.3);
				\draw (0.3,0.3) -- (-0.15,-0.3);
			\end{tikzpicture}
		\,	, \;
			\begin{tikzpicture}[centerzero]
				\draw (0.3,0.3) -- (0.15,0.15) -- (0,0.3);
				\draw (0.15,0.15) -- (-0.15,-0.3);
				\draw (-0.3,0.3) -- (0.15,-0.3);
			\end{tikzpicture}
	\,	, \;
		\begin{tikzpicture}[anchorbase]
				\draw (-0.2,0) -- (0,0.25) -- (0.2,0);
				\draw (0,0.25) -- (0,0.4);
				\draw  (-0.2,0) -- (-0.3,0.4);
				\draw (0.2,0) -- (0.3,0.4);
				\draw (0.2,0) -- (-0.3,-0.2);
				\draw  (-0.2,0) -- (0.3,-0.2);
			\end{tikzpicture}
\end{equation}
are linearly independent in  $\Tcat_\delta/\cI(\go^{\otimes 2},\go^{\otimes 3})$.
\end{lem}
\begin{proof}
	Let $b\colon \Tcat_\delta/\cI(\go^{\otimes 2},\go^{\otimes 3})\times \Tcat_{\delta}/\cI(\go^{\otimes 2},\go^{\otimes 3})\to \kk$ be the symmetric bilinear form such that $b(f,g)$ is the value, in $\kk$, of the closed diagram
\[
\begin{tikzpicture}[centerzero] 
\draw (-0.6,0) -- (0.0,0) -- (0.0,0.4) -- (-0.6,0.4) -- (-0.6,0);
\draw (0.6,0) -- (1.2,0) -- (1.2,0.4) -- (0.6,0.4) -- (0.6,0);
\draw[out=90,in =90,looseness=1] (-0.45,0.4) to (1.05,0.4);
\draw[out=90,in =90,looseness=1] (-0.15,0.4) to (0.75,0.4);
\draw[out=90,in =90,looseness=1] (-0.3,0.4) to (0.9,0.4);
\draw[out=-90,in =-90,looseness=1] (-0.45,0) to (1.05,0);
\draw[out=-90,in =-90,looseness=1] (-0.15,0) to (0.75,0);
\node at (-0.3,0.2) {$\scriptstyle  f$};
\node at (0.9,0.2) {$\scriptstyle  g$};	
\end{tikzpicture}
\]
for $f,g \in \Tcat_\delta/\cI(\go^{\otimes 2},\go^{\otimes 3})$. For example 
\begin{gather*}
	b\left(\begin{tikzpicture}[anchorbase]
		\draw (-0.2,0) -- (0,0.25) -- (0.2,0);
		\draw (0,0.25) -- (0,0.4);
		\draw (-0.2,-0.25) -- (-0.2,0) -- (-0.3,0.4);
		\draw (0.2,-0.25) -- (0.2,0) -- (0.3,0.4);
	\end{tikzpicture}, \begin{tikzpicture}[anchorbase]
		\draw (-0.2,0) -- (0,0.25) -- (0.2,0);
		\draw (0,0.25) -- (0,0.4);
		\draw (-0.2,-0.25) -- (-0.2,0) -- (-0.3,0.4);
		\draw (0.2,-0.25) -- (0.2,0) -- (0.3,0.4);
	\end{tikzpicture}\right)=\begin{tikzpicture}[anchorbase,scale=0.3]
		\draw[out=90,in=90, looseness=1.6] (-1,0) to (2,0);
		\draw[out=90,in=90,looseness=1.6] (0,0) to (1,0);
		\draw[out=90,in=90,looseness=1.6] (-0.5,0) to (1.5,0);
		\draw[out=-90,in=-90, looseness=1.6] (-1,0) to (2,0);
		\draw[out=-90,in=-90,looseness=1.6] (0,0) to (1,0);
		\draw (-1,0) -- (0,0);
		\draw (1,0) -- (2,0);
	\end{tikzpicture}\overset{\cref{chess}}{=} \begin{tikzpicture}[anchorbase,scale=0.3]
		\draw[out=90,in=90, looseness=1.6] (-1,0) to (2,0);
		\draw[out=90,in=90,looseness=1.6] (0,0) to (1,0);
		\draw[out=-90,in=-90, looseness=1.6] (-1,0) to (2,0);
		\draw[out=-90,in=-90,looseness=1.6] (0,0) to (1,0);
		\draw (-1,0) -- (0,0);
		\draw (1,0) -- (2,0);
	\end{tikzpicture}\overset{\cref{chess}}{=} \delta,
\\
b\left(\begin{tikzpicture}[anchorbase]
		\draw (-0.2,0) -- (0,0.25) -- (0.2,0);
		\draw (0,0.25) -- (0,0.4);
		\draw (-0.2,-0.25) -- (-0.2,0) -- (-0.3,0.4);
		\draw (0.2,-0.25) -- (0.2,0) -- (0.3,0.4);
	\end{tikzpicture},\begin{tikzpicture}[anchorbase]
		\draw (-0.3,0.3) -- (0,0) -- (0.3,0.3);
		\draw (0,0.3) -- (-0.15,0.15);
		\draw (0,0) -- (0,-0.15) -- (-0.15,-0.3);
		\draw (0,-0.15) -- (0.15,-0.3);
	\end{tikzpicture}\right)=\begin{tikzpicture}[anchorbase,scale=0.3]
		\draw[out=90,in=90, looseness=1.6] (-1,0) to (2,0);
		\draw[out=90,in=90,looseness=1.6] (0,0) to (1,0);
		%\draw[out=90,in=90,looseness=1.5] (-0.5,0) to (1.5,0);
		\draw[out=-90,in=-90, looseness=1.6] (-1,0) to (2,0);
		\draw  (0,0) to (1,0);
		\draw (-1,0) -- (0,0);
		\draw (1,0) -- (2,0);
		\draw (0.5,0) -- (0.5,-1.4);
	\end{tikzpicture}\overset{\cref{triangle}}{=}\frac{1}{2}\; \begin{tikzpicture}[anchorbase,scale=0.3]
		\draw[out=90,in=90, looseness=1.6] (-1,0) to (2,0);
		\draw[out=-90,in=-90, looseness=1.6] (-1,0) to (2,0);
		\draw  (0,0) to (1,0);
		\draw (-1,0) -- (0,0);
		\draw (1,0) -- (2,0);
		\draw (0.5,0) -- (0.5,-1.4);
	\end{tikzpicture}\overset{\cref{triangle}}{=}\frac{1}{4}\; \begin{tikzpicture}[anchorbase,scale=0.3]
		\draw[out=90,in=90, looseness=1.6] (-1,0) to (2,0);
		\draw[out=-90,in=-90, looseness=1.6] (-1,0) to (2,0);
		\draw  (0,0) to (1,0);
		\draw (-1,0) -- (0,0);
		\draw (1,0) -- (2,0);
	\end{tikzpicture}\overset{\cref{chess}}{=}\frac{1}{4}\delta,
	\\
	b\left(\begin{tikzpicture}[anchorbase]
		\draw (-0.2,0) -- (0,0.25) -- (0.2,0);
		\draw (0,0.25) -- (0,0.4);
		\draw (-0.2,-0.25) -- (-0.2,0) -- (-0.3,0.4);
		\draw (0.2,-0.25) -- (0.2,0) -- (0.3,0.4);
	\end{tikzpicture}, \begin{tikzpicture}[anchorbase]
		\draw (0.3,0.3) -- (0,0) -- (-0.3,0.3);
		\draw (0,0.3) -- (0.15,0.15);
		\draw (0,0) -- (0,-0.15) -- (0.15,-0.3);
		\draw (0,-0.15) -- (-0.15,-0.3);
	\end{tikzpicture}\right)= \begin{tikzpicture}[anchorbase,scale=.2]
		\draw  (0,0) ellipse (2 and 2);
		\draw (0,-2) -- (0,2);
		\draw (0,0) -- (-2,0);
		\draw (-1,0) -- (-1,1.733);
	\end{tikzpicture}\overset{\cref{triangle}}{=}\frac{1}{2}\; \begin{tikzpicture}[anchorbase,scale=.2]
		\draw  (0,0) ellipse (2 and 2);
		\draw (0,-2) -- (0,2);
		\draw (0,0) -- (-2,0);
	\end{tikzpicture}=\frac{1}{4}\delta, \quad \text{and} 
\\ b\left(\begin{tikzpicture}[anchorbase]
	\draw (-0.2,0) -- (0,0.25) -- (0.2,0);
	\draw (0,0.25) -- (0,0.4);
	\draw (-0.2,-0.25) -- (-0.2,0) -- (-0.3,0.4);
	\draw (0.2,-0.25) -- (0.2,0) -- (0.3,0.4);
\end{tikzpicture},	\begin{tikzpicture}[anchorbase]
	\draw (-0.3,-0.3) -- (0.3,0.3);
	\draw (-0.3,0.3) -- (-0.15,-0.15);
	\draw (0,0.3) -- (0.15,0.15);
	\draw (0,0) -- (0.3,-0.3);
\end{tikzpicture}\right)= \begin{tikzpicture}[anchorbase,scale=0.3]
	\draw[out=90,in=90, looseness=1.6] (-1,0) to (2,0);
	\draw[out=90,in=90,looseness=1.6] (0,0) to (1,0);
	\draw[out=-90,in=-90, looseness=1.6] (-1,0) to (2,0);
	\draw[out=-90,in=-90,looseness=1.6] (0,0) to (1,0);
	\draw (-1,0) -- (0,0);
	\draw (1,0) -- (2,0);
	\draw (-0.5,0) -- (-0.5,1.1);
\end{tikzpicture}=\frac{1}{2}\delta.
\end{gather*}

These are the first four values in the first row of the  Gram matrix of $b$ with respect to the vectors in \cref{particles}:
\[
\delta \begin{pmatrix}
		1& \quarter & \quarter & \half & \half & \half & 0& 0 & 1 & 1& \half & -1 & -1 & -\half & -\half & 0 \\
		\quarter & 1 & \half & \quarter & \half & 0 & \half & 1 & 0 & 1 & -1 & \half & -\half & -1 & -\half & -\quarter\\
		\quarter & \half & 1 & \half & \quarter & 0 & 1 & \half & 1 & 0 & -1 & -\half & \half & -\half & -1 & -\quarter\\
		\half & \quarter & \half & 1 & \quarter & 1 & 0& 1 &\half & 0 & -\half  & -1 & -\half & \half &  -1 & \quarter\\
		\half & \half & \quarter & \quarter & 1 & 1 & 1 & 0 & 0& \half & -\half & -\half & -1 & -1 & \half & \quarter\\
		\half & 0 & 0 & 1 & 1 & \delta & 1 & 1 & 0 & 0 & 0  & -1 & -1 & 1 & 1 & \half\\
		0 & \half & 1 & 0 & 1 & 1 & \delta & 0 & 1 & 0 & -1 & 0 & 1 & -1 & 1 & -1\\
		0 & 1 & \half & 1 & 0 & 1 & 0 & \delta & 0 & 1 & -1 & 1 & 0 & 1 & -1 & -\half\\
		1 & 0 & 1 & \half & 0 & 0& 1 & 0 & \delta & 1 & 1 & -1& 1 & 0& -1 & -1 \\
		1 & 1 & 0& 0& \half & 0& 0& 1 & 1 & \delta& 1 & 1 & -1 & -1& 0& -1\\
		\half & -1 & -1 & -\half & -\half & 0 & -1 & -1 &  1 & 1 & \delta & 0 & 0 & 1 & 1 & -\half\\
		-1 & \half & -\half & -1 & -\half & -1 & 0& 1 & -1 & 1 & 0 & \delta & 1 & 0 & 1 & -1\\
		-1 & -\half & \half & -\half & -1 & -1 & 1 & 0 & 1 & -1 & 0& 1& \delta & 1 & 0& -1\\
		-\half & -1 & -\half & \half & -1 & 1 & -1 & 1 & 0 & -1 & 1 & 0 & 1 & \delta & 0 & 0\\
		-\half& -\half & -1 & -1 & \half & 1 & 1 & -1 & -1 & 0 & 1 & 1 & 0 & 0 & \delta & 0\\
		0 & -\quarter & -\quarter & \quarter & \quarter & \half & -1 & -\half & -1 & -1 & -\half & -1 & -1 & 0 & 0& 1
\end{pmatrix}
\]
The determinant of this matrix is the polynomial $ p(\delta)=\frac{125}{4096}\delta^{16}(\delta - 3)^8 (4 \delta^2 - 30 \delta + 45)1_\one$. So the elements \cref{particles}  are linearly independent for any $\delta\neq -2$ and not root of $p$.  
\end{proof}
\begin{rem}\label{fickle}
Here we explain why the functor in \cref{Kauff} is not faithful. From the equivalence of categories between $\TL(-2)$ and $\fsl\md$, see for example \cite{ST09}, we get that $\dim \Kar \TL(-2)\left(\Phi(\go^{\otimes 2}),\Phi(\go^{\otimes 3})\right)=6$ (here $\Phi$ as in \cref{Kauff}).  The rank of the Gram matrix in the proof of \cref{nugget}  for $\delta=3$  is $8$ (see \cref{appendix}), and so $\dim \mathpzc{SL}(\go^{\otimes 2},\go^{\otimes 3})\geq 8$. Since $\dim \Kar \TL(-2)\left(\Phi(\go^{\otimes 2}),\Phi(\go^{\otimes 3})\right)=6$, it follows that the functor in \cref{Kauff} is not faithful. 
%The fact that  $6\times 6$ minor, in the same matrix, 
%\[
%\det\begin{pmatrix}
%	1& \quarter & \quarter & \half & \half & \half \\
%	\quarter & 1 & \half & \quarter & \half & 0 \\
%	\quarter & \half & 1 & \half & \quarter & 0 \\
%	\half & \quarter & \half & 1 & \quarter & 1  \\
%	\half & \half & \quarter & \quarter & 1 & 1 \\
%	\half & 0 & 0 & 1 & 1 & \delta 
%\end{pmatrix}\neq 0
%\]
%\[
%3\begin{pmatrix}
%		-1& -\quarter & -\quarter & \half & \half & -\half \\
%		-\quarter & \half & -1 & \half & -\quarter & 0 \\
%		-\quarter & -1 & \half & -\quarter & \half & 0 \\
%		\half & \half & -\quarter & -\quarter & -1 & 1  \\
%		\half & -\quarter & \half & -1 & -\quarter & 1 \\
%		-\half & 0 & 0 & 1 & 1 & -3 
%	\end{pmatrix} 
%	\]
%implies 
An argument similar to the proof in \cref{nugget} shows that the images under the functor $\Phi$, from \cref{Kauff}, of the elements 
\[
\begin{tikzpicture}[anchorbase]
	\draw (-0.2,0) -- (0,0.25) -- (0.2,0);
	\draw (0,0.25) -- (0,0.4);
	\draw (-0.2,-0.25) -- (-0.2,0) -- (-0.3,0.4);
	\draw (0.2,-0.25) -- (0.2,0) -- (0.3,0.4);
\end{tikzpicture}
\,	, \;
\begin{tikzpicture}[anchorbase]
	\draw (-0.3,0.3) -- (0,0) -- (0.3,0.3);
	\draw (0,0.3) -- (-0.15,0.15);
	\draw (0,0) -- (0,-0.15) -- (-0.15,-0.3);
	\draw (0,-0.15) -- (0.15,-0.3);
\end{tikzpicture}
\, ,\; 
\begin{tikzpicture}[anchorbase]
	\draw (0.3,0.3) -- (0,0) -- (-0.3,0.3);
	\draw (0,0.3) -- (0.15,0.15);
	\draw (0,0) -- (0,-0.15) -- (0.15,-0.3);
	\draw (0,-0.15) -- (-0.15,-0.3);
\end{tikzpicture}
\,	, \; 
\begin{tikzpicture}[anchorbase]
	\draw (-0.3,-0.3) -- (0.3,0.3);
	\draw (-0.3,0.3) -- (-0.15,-0.15);
	\draw (0,0.3) -- (0.15,0.15);
	\draw (0,0) -- (0.3,-0.3);
\end{tikzpicture}
\,	, \; 
\begin{tikzpicture}[anchorbase]
	\draw (0.3,-0.3) -- (-0.3,0.3);
	\draw (0.3,0.3) -- (0.15,-0.15);
	\draw (0,0.3) -- (-0.15,0.15);
	\draw (0,0) -- (-0.3,-0.3);
\end{tikzpicture}
\,	, \; 
\begin{tikzpicture}[centerzero]
	\draw (-0.15,-0.3) -- (-0.15,-0.23) arc(180:0:0.15) -- (0.15,-0.3);
	\draw (-0.3,0.3) -- (0,0.08) -- (0.3,0.3);
	\draw (0,0.3) -- (0,0.08);
\end{tikzpicture}
\,
\]
are linearly independent, and hence form a basis of $\Kar\TL(-2)\left(\Phi(\go^{\otimes 2}),\Phi(\go^{\otimes 3})\right)$.  
\end{rem}
\begin{lem} \label{pentexplode}
    If $\delta\notin \{-2, 0, 3, \frac{15\pm 3\sqrt{5}}{4}\}$ and  $\cI$ is a tensor ideal of $\Tcat_\delta$ such that \begin{equation}\label{dim16}
    	\dim\left(\Tcat_{\delta}/\cI\right)\left(\go^{\otimes 2},\go^{\otimes 2}\right)=5 \quad \text{and}\quad 	\dim\left(\Tcat_{\delta}/\cI\right)\left(\go^{\otimes 2},\go^{\otimes 3}\right)=16
    \end{equation} 
and such that \cref{magic} holds, then the relation \cref{pentburst} also holds in $\Tcat_\delta/\cI$.
\end{lem}
	
\begin{proof}
 By the assumption \cref{dim16} and \cref{nugget}, the elements \cref{particles} form a basis for $\Tcat_\delta/\cI(\go^{\otimes 2},\go^{\otimes 3})$. Then we can write the pentagon as a linear combination of these elements. Rotating this linear combination by $\Rot^n$, for $n=0,\ldots 4$, and then summing up over the rotations, we get a new linear combination in the form
\begin{equation} \label{pentbreak1}
		\begin{multlined}
			\pentmor =
			b
			\left(
			\begin{tikzpicture}[anchorbase]
				\draw (-0.2,0) -- (0,0.25) -- (0.2,0);
				\draw (0,0.25) -- (0,0.4);
				\draw (-0.2,-0.25) -- (-0.2,0) -- (-0.3,0.4);
				\draw (0.2,-0.25) -- (0.2,0) -- (0.3,0.4);
			\end{tikzpicture}
			+
			\begin{tikzpicture}[anchorbase]
				\draw (-0.3,0.3) -- (0,0) -- (0.3,0.3);
				\draw (0,0.3) -- (-0.15,0.15);
				\draw (0,0) -- (0,-0.15) -- (-0.15,-0.3);
				\draw (0,-0.15) -- (0.15,-0.3);
			\end{tikzpicture}
			+
			\begin{tikzpicture}[anchorbase]
				\draw (0.3,0.3) -- (0,0) -- (-0.3,0.3);
				\draw (0,0.3) -- (0.15,0.15);
				\draw (0,0) -- (0,-0.15) -- (0.15,-0.3);
				\draw (0,-0.15) -- (-0.15,-0.3);
			\end{tikzpicture}
			+
			\begin{tikzpicture}[anchorbase]
				\draw (-0.3,-0.3) -- (0.3,0.3);
				\draw (-0.3,0.3) -- (-0.15,-0.15);
				\draw (0,0.3) -- (0.15,0.15);
				\draw (0,0) -- (0.3,-0.3);
			\end{tikzpicture}
			+
			\begin{tikzpicture}[anchorbase]
				\draw (0.3,-0.3) -- (-0.3,0.3);
				\draw (0.3,0.3) -- (0.15,-0.15);
				\draw (0,0.3) -- (-0.15,0.15);
				\draw (0,0) -- (-0.3,-0.3);
			\end{tikzpicture}
			\right)
			+ c
			\left(
			\begin{tikzpicture}[centerzero]
				\draw (-0.15,-0.3) -- (-0.15,-0.23) arc(180:0:0.15) -- (0.15,-0.3);
				\draw (-0.3,0.3) -- (0,0.08) -- (0.3,0.3);
				\draw (0,0.3) -- (0,0.08);
			\end{tikzpicture}
			+
			\begin{tikzpicture}[centerzero]
				\draw (-0.2,-0.3) -- (-0.2,0.3);
				\draw (0,0.3) -- (0.15,0) -- (0.3,0.3);
				\draw (0.15,0) -- (0.15,-0.3);
			\end{tikzpicture}
			+
			\begin{tikzpicture}[centerzero]
				\draw (0.2,-0.3) -- (0.2,0.3);
				\draw (0,0.3) -- (-0.15,0) -- (-0.3,0.3);
				\draw (-0.15,0) -- (-0.15,-0.3);
			\end{tikzpicture}
			+
			\begin{tikzpicture}[centerzero]
				\draw (-0.3,0.3) -- (-0.3,0.23) arc(180:360:0.15) -- (0,0.3);
				\draw (0.3,0.3) -- (0.15,0) -- (-0.2,-0.3);
				\draw (0.2,-0.3) -- (0.15,0);
			\end{tikzpicture}
			+
			\begin{tikzpicture}[centerzero]
				\draw (0.3,0.3) -- (0.3,0.23) arc(360:180:0.15) -- (0,0.3);
				\draw (-0.3,0.3) -- (-0.15,0) -- (0.2,-0.3);
				\draw (-0.2,-0.3) -- (-0.15,0);
			\end{tikzpicture}
			\right)
			\\
			+ d
			\left(
			\begin{tikzpicture}[centerzero]
				\draw (0,0.3) -- (0,-0.15) -- (-0.15,-0.3);
				\draw (0,-0.15) -- (0.15,-0.3);
				\draw (-0.2,0.3) -- (-0.2,0.25) arc(180:360:0.2) -- (0.2,0.3);
			\end{tikzpicture}
			+
			\begin{tikzpicture}[centerzero]
				\draw (0,0.3) to[out=-45,in=70] (0.15,-0.3);
				\draw (-0.3,0.3) -- (0,0) -- (0.3,0.3);
				\draw (0,0) -- (-0.15,-0.3);
			\end{tikzpicture}
			+
			\begin{tikzpicture}[centerzero]
				\draw (0,0.3) to[out=225,in=110] (-0.15,-0.3);
				\draw (0.3,0.3) -- (0,0) -- (-0.3,0.3);
				\draw (0,0) -- (0.15,-0.3);
			\end{tikzpicture}
			+
			\begin{tikzpicture}[centerzero]
				\draw (-0.3,0.3) -- (-0.15,0.15) -- (0,0.3);
				\draw (-0.15,0.15) -- (0.15,-0.3);
				\draw (0.3,0.3) -- (-0.15,-0.3);
			\end{tikzpicture}
			+
			\begin{tikzpicture}[centerzero]
				\draw (0.3,0.3) -- (0.15,0.15) -- (0,0.3);
				\draw (0.15,0.15) -- (-0.15,-0.3);
				\draw (-0.3,0.3) -- (0.15,-0.3);
			\end{tikzpicture}
			\right) +e  \sum_{n=0}^{4}\Rot^{n}\left(\begin{tikzpicture}[anchorbase]
				\draw (-0.2,0) -- (0,0.25) -- (0.2,0);
				\draw (0,0.25) -- (0,0.4);
				\draw  (-0.2,0) -- (-0.3,0.4);
				\draw (0.2,0) -- (0.3,0.4);
				\draw (0.2,0) -- (-0.3,-0.2);
				\draw  (-0.2,0) -- (0.3,-0.2);
			\end{tikzpicture}\; \right).
		\end{multlined}
\end{equation}
We have
\begin{equation}
	\begin{multlined} 
\Rot\left( \begin{tikzpicture}[anchorbase]
	\draw (-0.2,0) -- (0,0.25) -- (0.2,0);
	\draw (0,0.25) -- (0,0.4);
	\draw  (-0.2,0) -- (-0.3,0.4);
	\draw (0.2,0) -- (0.3,0.4);
	\draw (0.2,0) -- (-0.3,-0.2);
	\draw  (-0.2,0) -- (0.3,-0.2);
\end{tikzpicture}\;\right)\, =\,	\begin{tikzpicture}[anchorbase]
		\draw (-0.2,0.4) -- (-0.2,0.2);
		\draw (-0.2,0.2) -- (0.2,-0.2);
		\draw (-0.2,0.2) -- (-0.4,0) -- (-0.2,-0.2);
		\draw (-0.2,-0.2) -- (0.2,0.2);
		\draw (-0.2,-0.2) -- (-0.2,-0.4);
		\draw (-0.4,0) -- (-0.4,0.4);
	\end{tikzpicture} \, \overset{\cref{venom}}{=} \, \begin{tikzpicture}[anchorbase]
	\draw (-0.2,0.4) -- (-0.2,0.2);
	\draw[out=-90,in=90]  (-0.2,0.2) to (-0.4,-0.2);
	\draw[out=-90,in=90]  (-0.4,-0.2) to (0,-0.4);
	\draw (-0.2,0.2) -- (-0.4,0) -- (-0.2,-0.2);
	\draw(-0.2,-0.2) -- (0.2,0.2);
	\draw (-0.2,-0.2) -- (-0.2,-0.4);
	\draw (-0.4,0) -- (-0.4,0.4);
\end{tikzpicture}\; \overset{\mathclap{\cref{turvy}}}{\underset{\mathclap{\cref{magic}}}{=}}\; \; \begin{tikzpicture}[anchorbase]
	\draw (-0.2,0) -- (0,0.25) -- (0.2,0);
	\draw (0,0.25) -- (0,0.4);
	\draw  (-0.2,0) -- (-0.3,0.4);
	\draw (0.2,0) -- (0.3,0.4);
	\draw (0.2,0) -- (-0.3,-0.2);
	\draw  (-0.2,0) -- (0.3,-0.2);
\end{tikzpicture}\, -\;  \begin{tikzpicture}[anchorbase]
\draw (-0.2,0.3) -- (-0.2,0.2);
\draw (-0.2,0.3) -- (-0.1,0.4);
\draw (-0.2,0.3) -- (-0.3,0.4);
\draw (-0.2,0.2) -- (0,0) -- (0.2,0.2) -- (0.2,0.4);
\draw (-0.2,0.2) -- (-0.4,0) -- (0,-0.4);
\draw (0,0) -- (-0.4,-0.4);
\end{tikzpicture}
\; \overset{\mathclap{\cref{turvy}}}{\underset{\mathclap{\cref{magic}}}{=}}\;\; \begin{tikzpicture}[anchorbase]
	\draw (-0.2,0) -- (0,0.25) -- (0.2,0);
	\draw (0,0.25) -- (0,0.4);
	\draw  (-0.2,0) -- (-0.3,0.4);
	\draw (0.2,0) -- (0.3,0.4);
	\draw (0.2,0) -- (-0.3,-0.2);
	\draw  (-0.2,0) -- (0.3,-0.2);
\end{tikzpicture}\; +\; 	\begin{tikzpicture}[anchorbase]
\draw (-0.3,0.3) -- (0,0) -- (0.3,0.3);
\draw (0,0.3) -- (-0.15,0.15);
\draw (0,0) -- (0,-0.15) -- (-0.15,-0.3);
\draw (0,-0.15) -- (0.15,-0.3);
\end{tikzpicture}\; - \; 	\begin{tikzpicture}[anchorbase]
\draw (0.3,-0.3) -- (-0.3,0.3);
\draw (0.3,0.3) -- (0.15,-0.15);
\draw (0,0.3) -- (-0.15,0.15);
\draw (0,0) -- (-0.3,-0.3);
\end{tikzpicture}.
	\end{multlined}
\end{equation}
Then
\begin{align*}
	 \sum_{n=0}^{4}\Rot^n \left( \begin{tikzpicture}[anchorbase]
	\draw (-0.2,0) -- (0,0.25) -- (0.2,0);
	\draw (0,0.25) -- (0,0.4);
	\draw  (-0.2,0) -- (-0.3,0.4);
	\draw (0.2,0) -- (0.3,0.4);
	\draw (0.2,0) -- (-0.3,-0.2);
	\draw  (-0.2,0) -- (0.3,-0.2);
\end{tikzpicture}\;\right)& = 5 \;  \begin{tikzpicture}[anchorbase]
\draw (-0.2,0) -- (0,0.25) -- (0.2,0);
\draw (0,0.25) -- (0,0.4);
\draw  (-0.2,0) -- (-0.3,0.4);
\draw (0.2,0) -- (0.3,0.4);
\draw (0.2,0) -- (-0.3,-0.2);
\draw  (-0.2,0) -- (0.3,-0.2);
\end{tikzpicture} \; +\sum_{n=0}^{3}(4-n)\Rot^n\left(	\begin{tikzpicture}[anchorbase]
\draw (-0.3,0.3) -- (0,0) -- (0.3,0.3);
\draw (0,0.3) -- (-0.15,0.15);
\draw (0,0) -- (0,-0.15) -- (-0.15,-0.3);
\draw (0,-0.15) -- (0.15,-0.3);
\end{tikzpicture}\; - \; 	\begin{tikzpicture}[anchorbase]
\draw (0.3,-0.3) -- (-0.3,0.3);
\draw (0.3,0.3) -- (0.15,-0.15);
\draw (0,0.3) -- (-0.15,0.15);
\draw (0,0) -- (-0.3,-0.3);
\end{tikzpicture}\right) \\
& =  5 \;  \begin{tikzpicture}[anchorbase]
	\draw (-0.2,0) -- (0,0.25) -- (0.2,0);
	\draw (0,0.25) -- (0,0.4);
	\draw  (-0.2,0) -- (-0.3,0.4);
	\draw (0.2,0) -- (0.3,0.4);
	\draw (0.2,0) -- (-0.3,-0.2);
	\draw  (-0.2,0) -- (0.3,-0.2);
\end{tikzpicture} +2\; \begin{tikzpicture}[anchorbase]
\draw (-0.2,0) -- (0,0.25) -- (0.2,0);
\draw (0,0.25) -- (0,0.4);
\draw (-0.2,-0.25) -- (-0.2,0) -- (-0.3,0.4);
\draw (0.2,-0.25) -- (0.2,0) -- (0.3,0.4);
\end{tikzpicture}
+2\;
\begin{tikzpicture}[anchorbase]
\draw (-0.3,0.3) -- (0,0) -- (0.3,0.3);
\draw (0,0.3) -- (-0.15,0.15);
\draw (0,0) -- (0,-0.15) -- (-0.15,-0.3);
\draw (0,-0.15) -- (0.15,-0.3);
\end{tikzpicture}
+2\;
\begin{tikzpicture}[anchorbase]
\draw (0.3,0.3) -- (0,0) -- (-0.3,0.3);
\draw (0,0.3) -- (0.15,0.15);
\draw (0,0) -- (0,-0.15) -- (0.15,-0.3);
\draw (0,-0.15) -- (-0.15,-0.3);
\end{tikzpicture}
-3\;
\begin{tikzpicture}[anchorbase]
\draw (-0.3,-0.3) -- (0.3,0.3);
\draw (-0.3,0.3) -- (-0.15,-0.15);
\draw (0,0.3) -- (0.15,0.15);
\draw (0,0) -- (0.3,-0.3);
\end{tikzpicture}
-3\;
\begin{tikzpicture}[anchorbase]
\draw (0.3,-0.3) -- (-0.3,0.3);
\draw (0.3,0.3) -- (0.15,-0.15);
\draw (0,0.3) -- (-0.15,0.15);
\draw (0,0) -- (-0.3,-0.3);
\end{tikzpicture}.
\end{align*}
Thus
\begin{equation}\label{ghoul}
	\begin{multlined}
		\pentmor =
		(b+2e)
		\left(
		\begin{tikzpicture}[anchorbase]
			\draw (-0.2,0) -- (0,0.25) -- (0.2,0);
			\draw (0,0.25) -- (0,0.4);
			\draw (-0.2,-0.25) -- (-0.2,0) -- (-0.3,0.4);
			\draw (0.2,-0.25) -- (0.2,0) -- (0.3,0.4);
		\end{tikzpicture}
		+
		\begin{tikzpicture}[anchorbase]
			\draw (-0.3,0.3) -- (0,0) -- (0.3,0.3);
			\draw (0,0.3) -- (-0.15,0.15);
			\draw (0,0) -- (0,-0.15) -- (-0.15,-0.3);
			\draw (0,-0.15) -- (0.15,-0.3);
		\end{tikzpicture}
		+
		\begin{tikzpicture}[anchorbase]
			\draw (0.3,0.3) -- (0,0) -- (-0.3,0.3);
			\draw (0,0.3) -- (0.15,0.15);
			\draw (0,0) -- (0,-0.15) -- (0.15,-0.3);
			\draw (0,-0.15) -- (-0.15,-0.3);
		\end{tikzpicture}	\right)
	+(b-3e)	\left(\begin{tikzpicture}[anchorbase]
			\draw (-0.3,-0.3) -- (0.3,0.3);
			\draw (-0.3,0.3) -- (-0.15,-0.15);
			\draw (0,0.3) -- (0.15,0.15);
			\draw (0,0) -- (0.3,-0.3);
		\end{tikzpicture}
		+
		\begin{tikzpicture}[anchorbase]
			\draw (0.3,-0.3) -- (-0.3,0.3);
			\draw (0.3,0.3) -- (0.15,-0.15);
			\draw (0,0.3) -- (-0.15,0.15);
			\draw (0,0) -- (-0.3,-0.3);
		\end{tikzpicture}\right)
		+ c
		\left(
		\begin{tikzpicture}[centerzero]
			\draw (-0.15,-0.3) -- (-0.15,-0.23) arc(180:0:0.15) -- (0.15,-0.3);
			\draw (-0.3,0.3) -- (0,0.08) -- (0.3,0.3);
			\draw (0,0.3) -- (0,0.08);
		\end{tikzpicture}
		+
		\begin{tikzpicture}[centerzero]
			\draw (-0.2,-0.3) -- (-0.2,0.3);
			\draw (0,0.3) -- (0.15,0) -- (0.3,0.3);
			\draw (0.15,0) -- (0.15,-0.3);
		\end{tikzpicture}
		+
		\begin{tikzpicture}[centerzero]
			\draw (0.2,-0.3) -- (0.2,0.3);
			\draw (0,0.3) -- (-0.15,0) -- (-0.3,0.3);
			\draw (-0.15,0) -- (-0.15,-0.3);
		\end{tikzpicture}
		+
		\begin{tikzpicture}[centerzero]
			\draw (-0.3,0.3) -- (-0.3,0.23) arc(180:360:0.15) -- (0,0.3);
			\draw (0.3,0.3) -- (0.15,0) -- (-0.2,-0.3);
			\draw (0.2,-0.3) -- (0.15,0);
		\end{tikzpicture}
		+
		\begin{tikzpicture}[centerzero]
			\draw (0.3,0.3) -- (0.3,0.23) arc(360:180:0.15) -- (0,0.3);
			\draw (-0.3,0.3) -- (-0.15,0) -- (0.2,-0.3);
			\draw (-0.2,-0.3) -- (-0.15,0);
		\end{tikzpicture}
		\right)
		\\
		+ d
		\left(
		\begin{tikzpicture}[centerzero]
			\draw (0,0.3) -- (0,-0.15) -- (-0.15,-0.3);
			\draw (0,-0.15) -- (0.15,-0.3);
			\draw (-0.2,0.3) -- (-0.2,0.25) arc(180:360:0.2) -- (0.2,0.3);
		\end{tikzpicture}
		+
		\begin{tikzpicture}[centerzero]
			\draw (0,0.3) to[out=-45,in=70] (0.15,-0.3);
			\draw (-0.3,0.3) -- (0,0) -- (0.3,0.3);
			\draw (0,0) -- (-0.15,-0.3);
		\end{tikzpicture}
		+
		\begin{tikzpicture}[centerzero]
			\draw (0,0.3) to[out=225,in=110] (-0.15,-0.3);
			\draw (0.3,0.3) -- (0,0) -- (-0.3,0.3);
			\draw (0,0) -- (0.15,-0.3);
		\end{tikzpicture}
		+
		\begin{tikzpicture}[centerzero]
			\draw (-0.3,0.3) -- (-0.15,0.15) -- (0,0.3);
			\draw (-0.15,0.15) -- (0.15,-0.3);
			\draw (0.3,0.3) -- (-0.15,-0.3);
		\end{tikzpicture}
		+
		\begin{tikzpicture}[centerzero]
			\draw (0.3,0.3) -- (0.15,0.15) -- (0,0.3);
			\draw (0.15,0.15) -- (-0.15,-0.3);
			\draw (-0.3,0.3) -- (0.15,-0.3);
		\end{tikzpicture}
		\right) + 5e  \;\begin{tikzpicture}[anchorbase]
			\draw (-0.2,0) -- (0,0.25) -- (0.2,0);
			\draw (0,0.25) -- (0,0.4);
			\draw  (-0.2,0) -- (-0.3,0.4);
			\draw (0.2,0) -- (0.3,0.4);
			\draw (0.2,0) -- (-0.3,-0.2);
			\draw  (-0.2,0) -- (0.3,-0.2);
		\end{tikzpicture}.
	\end{multlined}
\end{equation}
 Recall, from \cref{triangle}, that $\displaystyle \trimor = \frac{1}{2} \mergemor$.  Composing with $\mergemor$ on the rightmost two strings at the top of \cref{ghoul} and using \cref{turvy,magic} gives
\begin{multline*}
		\frac{1}{2}  \sqmor
		=
		 \left(\frac{b+2e}{2}+b-3e\right) \Hmor + \left(3\frac{b+2e}{2}\right)\Imor   + (b-3e)\sqmor  
		+ c
		\left( \, \jail + \hourglass\,   + \Hmor + \Imor
		\right) \\
		+ d \left(- \Hmor - \Imor+  \crossmor \right)+\frac{5e}{2}\left(\Hmor-\Imor\right).
\end{multline*}
	Hence 
\[
	\left(\frac{1}{2} - b+3e\right) \sqmor
	= \left( \frac{3}{2}b + c - d +\frac{e}{2}\right)
	\left(\, \Hmor + \Imor\, \right)
	+
	c  
	\left(\, \jail + \hourglass \, \right)
	+  d
	\, \crossmor\ .
\]
	Comparing to \cref{sqburst} and using the fact that the morphisms in \cref{bigfive}  form a basis of $\displaystyle(\Fcat_\delta/\cI)(\go^{\otimes 2}, \go^{\otimes 2})$ gives
\begin{equation}\label{head}
		\begin{aligned} 
			&	\frac{1}{6} \left(\frac{1}{2}  - b+3e\right)
		 = \frac{3}{2}b + c - d+\frac{e}{2}, \\
	&	\frac{5}{6(\delta+2)}\left(\frac{1}{2}  - b+3e\right)
	  =  c, \\
	&	\frac{5}{6(\delta+2)}\left(\frac{1}{2}  - b+3e\right)	 = d.
	\end{aligned}
\end{equation}
Applying the operator $\Switch$ to \cref{ghoul}, and using \cref{chess}, gives 
\begin{equation}\label{boogeyman}
	\begin{multlined}
		\begin{tikzpicture}[anchorbase]
			\draw (-0.25,0.2) -- (0,0.25) -- (0.25,0.2);
			\draw (0,0.25) -- (0,0.4);
			\draw  (-0.2,0) -- (-0.3,0.4);
			\draw (0.2,0) -- (0.3,0.4);
			\draw (0.2,0) -- (-0.3,-0.2);
			\draw  (-0.2,0) -- (0.3,-0.2);
			\draw  (-0.2,0) -- (0.2,0);
		\end{tikzpicture} =
		(b+2e)
		\left(\;\begin{tikzpicture}[anchorbase]
			\draw (-0.2,0) -- (0,0.25) -- (0.2,0);
			\draw (0,0.25) -- (0,0.4);
			\draw  (-0.2,0) -- (-0.3,0.4);
			\draw (0.2,0) -- (0.3,0.4);
			\draw (0.2,0) -- (-0.3,-0.2);
			\draw  (-0.2,0) -- (0.3,-0.2);
		\end{tikzpicture}
	-
		\begin{tikzpicture}[anchorbase]
			\draw (-0.3,0.3) -- (0,0) -- (0.3,0.3);
			\draw (0,0.3) -- (-0.15,0.15);
			\draw (0,0) -- (0,-0.15) -- (-0.15,-0.3);
			\draw (0,-0.15) -- (0.15,-0.3);
		\end{tikzpicture}
		-
		\begin{tikzpicture}[anchorbase]
			\draw (0.3,0.3) -- (0,0) -- (-0.3,0.3);
			\draw (0,0.3) -- (0.15,0.15);
			\draw (0,0) -- (0,-0.15) -- (0.15,-0.3);
			\draw (0,-0.15) -- (-0.15,-0.3);
		\end{tikzpicture}	\right)
		+(b-3e)	\left(\begin{tikzpicture}[anchorbase]
			\draw (0.3,0) -- (0.3,0.3);
			\draw (-0.3,0) -- (-0.3,0.15);
			\draw (0.3,0) -- (-0.3,-0.3);
			\draw (-0.3,0) -- (0.3,-0.3);
			\draw (-0.3,0) -- (0.3,0);
			\draw (-0.3,0.15) -- (-0.15,0.3);
			\draw (-0.3,0.15) -- (-0.45,0.3);
		\end{tikzpicture}
		+
		\begin{tikzpicture}[anchorbase]
		\draw (0.3,0) -- (0.3,0.15);
		\draw (-0.3,0) -- (-0.3,0.3);
		\draw (0.3,0) -- (-0.3,-0.3);
		\draw (-0.3,0) -- (0.3,-0.3);
		\draw (-0.3,0) -- (0.3,0);
		\draw (0.45,0.3) -- (0.3,0.15);
		\draw (0.15,0.3) -- (0.3,0.15);
		\end{tikzpicture}\right)
		+ c
		\left(
		\begin{tikzpicture}[centerzero]
			\draw (-0.15,-0.3) -- (-0.15,-0.23) arc(180:0:0.15) -- (0.15,-0.3);
			\draw (-0.3,0.3) -- (0,0.08) -- (0.3,0.3);
			\draw (0,0.3) -- (0,0.08);
		\end{tikzpicture}
		+
	\begin{tikzpicture}[centerzero]
		\draw (-0.3,0.3) -- (-0.15,0.15) -- (0,0.3);
		\draw (-0.15,0.15) -- (0.15,-0.3);
		\draw (0.3,0.3) -- (-0.15,-0.3);
	\end{tikzpicture}
	+
	\begin{tikzpicture}[centerzero]
		\draw (0.3,0.3) -- (0.15,0.15) -- (0,0.3);
		\draw (0.15,0.15) -- (-0.15,-0.3);
		\draw (-0.3,0.3) -- (0.15,-0.3);
	\end{tikzpicture}	-
		\begin{tikzpicture}[centerzero]
			\draw (-0.3,0.3) -- (-0.3,0.23) arc(180:360:0.15) -- (0,0.3);
			\draw (0.3,0.3) -- (0.15,0) -- (-0.2,-0.3);
			\draw (0.2,-0.3) -- (0.15,0);
		\end{tikzpicture}
	-
		\begin{tikzpicture}[centerzero]
			\draw (0.3,0.3) -- (0.3,0.23) arc(360:180:0.15) -- (0,0.3);
			\draw (-0.3,0.3) -- (-0.15,0) -- (0.2,-0.3);
			\draw (-0.2,-0.3) -- (-0.15,0);
		\end{tikzpicture}
		\right)
		\\
		+ d
		\left(
		-\,\begin{tikzpicture}[centerzero]
			\draw (0,0.3) -- (0,-0.15) -- (-0.15,-0.3);
			\draw (0,-0.15) -- (0.15,-0.3);
			\draw (-0.2,0.3) -- (-0.2,0.25) arc(180:360:0.2) -- (0.2,0.3);
		\end{tikzpicture}
		+
		\begin{tikzpicture}[centerzero]
			\draw (0,0.3) to[out=-45,in=70] (0.15,-0.3);
			\draw (-0.3,0.3) -- (0,0) -- (0.3,0.3);
			\draw (0,0) -- (-0.15,-0.3);
		\end{tikzpicture}
		+
		\begin{tikzpicture}[centerzero]
			\draw (0,0.3) to[out=225,in=110] (-0.15,-0.3);
			\draw (0.3,0.3) -- (0,0) -- (-0.3,0.3);
			\draw (0,0) -- (0.15,-0.3);
		\end{tikzpicture}
		+
	\begin{tikzpicture}[centerzero]
		\draw (-0.2,-0.3) -- (-0.2,0.3);
		\draw (0,0.3) -- (0.15,0) -- (0.3,0.3);
		\draw (0.15,0) -- (0.15,-0.3);
	\end{tikzpicture}
	+
	\begin{tikzpicture}[centerzero]
		\draw (0.2,-0.3) -- (0.2,0.3);
		\draw (0,0.3) -- (-0.15,0) -- (-0.3,0.3);
		\draw (-0.15,0) -- (-0.15,-0.3);
	\end{tikzpicture}\,
		\right) + 5e  	\begin{tikzpicture}[anchorbase]
			\draw (-0.2,0) -- (0,0.25) -- (0.2,0);
			\draw (0,0.25) -- (0,0.4);
			\draw (-0.2,-0.25) -- (-0.2,0) -- (-0.3,0.4);
			\draw (0.2,-0.25) -- (0.2,0) -- (0.3,0.4);
		\end{tikzpicture}.
	\end{multlined}
\end{equation}
Using relations \cref{magic,sqburst} and the last relation in \cref{turvy}, together with the fact, from \cref{head}, that $c=d$, leads to the equation 
\begin{multline}\label{jin}
		\pentmor =
		\left(\frac{1}{6}-2b+e\right)
		\left(
		\begin{tikzpicture}[anchorbase]
			\draw (-0.3,0.3) -- (0,0) -- (0.3,0.3);
			\draw (0,0.3) -- (-0.15,0.15);
			\draw (0,0) -- (0,-0.15) -- (-0.15,-0.3);
			\draw (0,-0.15) -- (0.15,-0.3);
		\end{tikzpicture}
		+
		\begin{tikzpicture}[anchorbase]
			\draw (0.3,0.3) -- (0,0) -- (-0.3,0.3);
			\draw (0,0.3) -- (0.15,0.15);
			\draw (0,0) -- (0,-0.15) -- (0.15,-0.3);
			\draw (0,-0.15) -- (-0.15,-0.3);
		\end{tikzpicture}	\right)+ 5e \, 
	\begin{tikzpicture}[anchorbase]
	\draw (-0.2,0) -- (0,0.25) -- (0.2,0);
	\draw (0,0.25) -- (0,0.4);
	\draw (-0.2,-0.25) -- (-0.2,0) -- (-0.3,0.4);
	\draw (0.2,-0.25) -- (0.2,0) -- (0.3,0.4);
\end{tikzpicture}
		+(b-3e)	\left(\begin{tikzpicture}[anchorbase]
			\draw (-0.3,-0.3) -- (0.3,0.3);
			\draw (-0.3,0.3) -- (-0.15,-0.15);
			\draw (0,0.3) -- (0.15,0.15);
			\draw (0,0) -- (0.3,-0.3);
		\end{tikzpicture}
		+
		\begin{tikzpicture}[anchorbase]
			\draw (0.3,-0.3) -- (-0.3,0.3);
			\draw (0.3,0.3) -- (0.15,-0.15);
			\draw (0,0.3) -- (-0.15,0.15);
			\draw (0,0) -- (-0.3,-0.3);
		\end{tikzpicture}\right)
	\\
	+ 
	\left(\frac{5}{6(\delta+2)}-c\right)
		\left(
		\begin{tikzpicture}[centerzero]
			\draw (0,0.3) -- (0,-0.15) -- (-0.15,-0.3);
			\draw (0,-0.15) -- (0.15,-0.3);
			\draw (-0.2,0.3) -- (-0.2,0.25) arc(180:360:0.2) -- (0.2,0.3);
		\end{tikzpicture}
		+
		\begin{tikzpicture}[centerzero]
			\draw (-0.3,0.3) -- (-0.3,0.23) arc(180:360:0.15) -- (0,0.3);
			\draw (0.3,0.3) -- (0.15,0) -- (-0.2,-0.3);
			\draw (0.2,-0.3) -- (0.15,0);
		\end{tikzpicture}
		+
		\begin{tikzpicture}[centerzero]
			\draw (0.3,0.3) -- (0.3,0.23) arc(360:180:0.15) -- (0,0.3);
			\draw (-0.3,0.3) -- (-0.15,0) -- (0.2,-0.3);
			\draw (-0.2,-0.3) -- (-0.15,0);
		\end{tikzpicture}
		\right)
		+ c
		\left(\begin{tikzpicture}[centerzero]
			\draw (-0.15,-0.3) -- (-0.15,-0.23) arc(180:0:0.15) -- (0.15,-0.3);
			\draw (-0.3,0.3) -- (0,0.08) -- (0.3,0.3);
			\draw (0,0.3) -- (0,0.08);
		\end{tikzpicture}
		+
		\begin{tikzpicture}[centerzero]
			\draw (-0.2,-0.3) -- (-0.2,0.3);
			\draw (0,0.3) -- (0.15,0) -- (0.3,0.3);
			\draw (0.15,0) -- (0.15,-0.3);
		\end{tikzpicture}
		+
		\begin{tikzpicture}[centerzero]
			\draw (0.2,-0.3) -- (0.2,0.3);
			\draw (0,0.3) -- (-0.15,0) -- (-0.3,0.3);
			\draw (-0.15,0) -- (-0.15,-0.3);
		\end{tikzpicture}
		+
		\begin{tikzpicture}[centerzero]
			\draw (0,0.3) to[out=-45,in=70] (0.15,-0.3);
			\draw (-0.3,0.3) -- (0,0) -- (0.3,0.3);
			\draw (0,0) -- (-0.15,-0.3);
		\end{tikzpicture}
		+
		\begin{tikzpicture}[centerzero]
			\draw (0,0.3) to[out=225,in=110] (-0.15,-0.3);
			\draw (0.3,0.3) -- (0,0) -- (-0.3,0.3);
			\draw (0,0) -- (0.15,-0.3);
		\end{tikzpicture}
		+
		\begin{tikzpicture}[centerzero]
			\draw (-0.3,0.3) -- (-0.15,0.15) -- (0,0.3);
			\draw (-0.15,0.15) -- (0.15,-0.3);
			\draw (0.3,0.3) -- (-0.15,-0.3);
		\end{tikzpicture}
		+
		\begin{tikzpicture}[centerzero]
			\draw (0.3,0.3) -- (0.15,0.15) -- (0,0.3);
			\draw (0.15,0.15) -- (-0.15,-0.3);
			\draw (-0.3,0.3) -- (0.15,-0.3);
		\end{tikzpicture}
		\right) + (b+2e)  \;\begin{tikzpicture}[anchorbase]
			\draw (-0.2,0) -- (0,0.25) -- (0.2,0);
			\draw (0,0.25) -- (0,0.4);
			\draw  (-0.2,0) -- (-0.3,0.4);
			\draw (0.2,0) -- (0.3,0.4);
			\draw (0.2,0) -- (-0.3,-0.2);
			\draw  (-0.2,0) -- (0.3,-0.2);
		\end{tikzpicture}.
\end{multline}

Comparing \cref{jin} to \cref{ghoul}, and using the fact morphisms in \cref{particles} are linearly independent, we obtain the following equations
\begin{equation}\label{tail}
	 \frac{1}{6}-2b+e=b+2e, \; 5e=b+2e, \; b-3e =b-3e, \; b+2e =5e, \; \text{and}\, \frac{5}{6(\delta+2)}-c=c.
\end{equation}
Solving the system of equations in \cref{head,tail}, we get \[1=60e,\quad b=3e,\quad  c=d= \frac{25e}{\delta+2}.\]
As the dependence relation \cref{pentbreak1} was assumed to be nontrivial, then $e\neq 0$ and we get \cref{pentburst}.
\end{proof}
 
\begin{rem}
Note that the first six terms of relation \cref{pentburst} as well as the two  coefficients in the relation  appear  in~\cite[p.~136]{AR99}, although the relation is not complete therein.  
\end{rem}

\section{A functor into representations of exceptional Lie algebras\label{sec:functor}}

For the remainder of this paper, we assume that $\fg$ is of type either $G_2,F_4, E_6,E_7$ or $E_8$. In this case we have the following dimensions computed using Sagemath (see \hyperref[appendix]{appendix}): 
 \begin{equation} \label{demon}
	\dim  \Hom_\fg\left(\fg^{\otimes 2}, \fg^{\otimes 2}\right) = 5
	\quad \text{and} \quad
	\dim\Hom_\fg \left( \fg^{\otimes 2}, \fg^{\otimes 3} \right) =16.
\end{equation}

Since the category $\fg$-mod is idempotent complete, the functor from \cref{magneto} induces a functor
\begin{equation} \label{lake}
	\Kar(\Phi) \colon \Kar \left( \Tcat_{\delta} \right) \to \fg\md.
\end{equation}
Let $G$ be the compact, simply-connected Lie group whose
complexified Lie algebra is $\fg$.
\begin{prop} \label{splat} For $\fg$ of type $G_2,F_4$ or $E_8$, the functor $\Kar(\Phi)$ is essentially surjective.
\end{prop}
\begin{proof}
We need to show that every finite-dimensional irreducible representation of $\fg$ is a direct summand of some power of $\fg$.	For $G$ of type $G_2,F_4$ or $E_8$, every finite-dimensional irreducible representation is faithful. 
%({\color{red} See the answer of Victor here: https://mathoverflow.net/questions/328138/non-faithful-irreducible-representations-of-simple-lie-groups?rq=1.  I haven't found a good reference for the claim yet.})
 So, in particular the adjoint representation $\mathrm{Ad}$ is faithful. It follows, by \cite[Ch.~3, Th.~4.4]{BD85}, that every finite-dimensional irreducible representation of $G$ is a contained  in some tensor power of $\mathrm{Ad}$. Passing to the Lie algebra side, the proposition follows.
\end{proof}
\begin{rem}
	We have already constructed in \cref{G2} a functor from a certain quotient of $\Tcat_{7,-6}$ to the category of finite-dimensional representations of $G_2$. These two functors are different in that the one in \cref{G2} sends the generating object $\go$ to the minimal fundamental representation of $G_2$, which we identified with $\Im(\OO)$, while the one in \cref{splat} sends $\go$ to the adjoint representation.
\end{rem}
\begin{prop} \label{full} For $\fg$ of type $G_2,F_4,E_7$ or $E_8$, the functor $\Phi$  is full.
\end{prop}
\begin{proof}
 Let $\cC$ be the Karoubi envelope  of the image $\Phi(\Tcat_{\delta})$. The category $\cC$ is a strict rigid  symmetric monoidal category with $\End_\cC(\one)=\C$. As a subcategory of the category of finite-dimensional complex vector spaces, $\cC$ is a $C^{*}$-category in the sense of \cite[\S~1]{DR89}. 
It follows, by \cite[Th.~6.1]{DR89}, that $\cC$ is the category of finite-dimensional continuous unitary
representations which separates the points (see \cite[p.~200]{DR89}) of some compact  group $H$. Since $\fg\md$ contains $\cC$, we have $G\subseteq H$. On the other hand, by \cite[4.8]{St61}, $G$ is the group of automorphisms of   $\fg$ that preserve $\Phi(\mergemor)$ and $\Phi(\capmor)$. As elements of $H$ preserve $\Phi(\mergemor)$ and $\Phi(\capmor)$, we must have $H\subseteq G$. Thus, $G=H$ and the proposition follows.
\end{proof}
\begin{rem}
	The proof of \cref{full} exploits the fact  that all automorphisms of $\fg$, for $\fg$ of types $G_2,F_4,E_7$ or $E_8$, are inner automorphisms (i.e. belong to the group generated by $\exp(\ad_x)$ for $x\in\fg$). In the case of $E_6$, not all automorphisms of $\mathfrak{e}_6$  are inner automorphisms, see \cite[4.8]{St61}, and so the compact group $E_6$ is strictly smaller than the group of automorphisms of it Lie algebra $\mathfrak{e}_6$.
\end{rem}
\begin{prop} \label{baja}
	The functor $\Phi$  factors through $\Fcat_{\delta}$.
\end{prop}

\begin{proof}
Let $\cI$ be the kernel of the functor $\Phi$. This is a tensor ideal. By \cref{full} we obtain isomorphisms $(\Tcat_\delta/\cI)\left(\go^{\otimes n},\go^{\otimes m}\right)\overset{\simeq}{\to}\Hom_\fg\left(\fg^{\otimes n},\fg^{\otimes m}\right)$. Hence, the Jacobi identity \cref{magic} holds in $\Tcat_\delta/\cI$. Moreover, the dimensions \cref{demon} force the hom-spaces of $\Tcat_\delta/\cI$ to satisfy the dimension restrictions \cref{dim16}. By \cref{sqexplode,pentexplode}, this implies that \cref{sqburst,pentburst} hold in $\Tcat_\delta/\cI$ and, consequently, in $\fg\md$ as well.
\end{proof}

\begin{rem}
	Relations \cref{sqburst,pentburst,spinner} allow one to remove cycles of length up to five from morphisms in $\Tcat_{\delta}$. We do not know if we can remove cycles of length six or more. 
\end{rem} 
\appendix
\label{appendix}

\includepdf[pages=1,pagecommand={\phantomsection\addcontentsline{toc}{section}{Appendix. SageMath notebook}}]{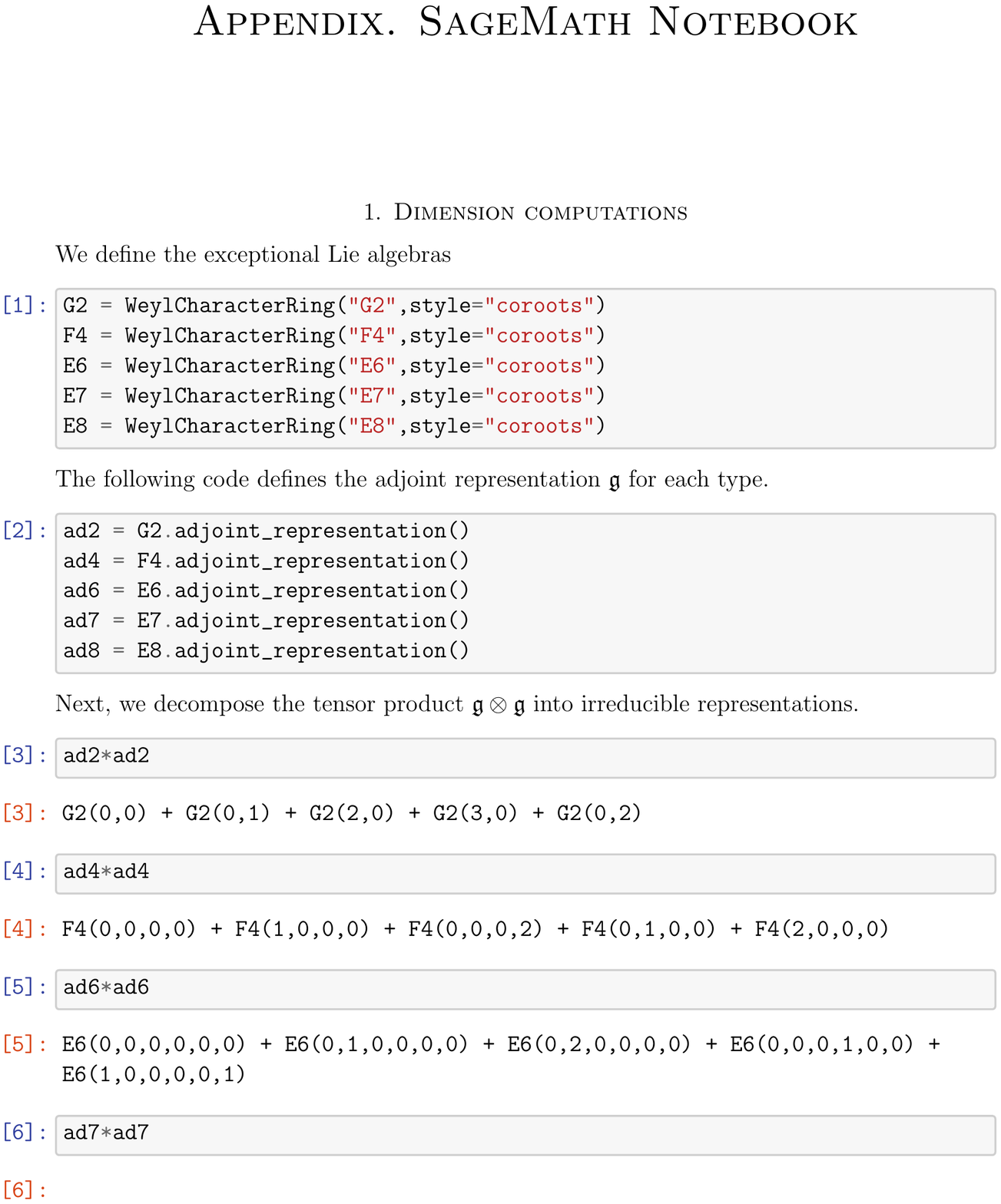}
\includepdf[pages=2-]{SageMath}

%=============
% Bibliography
%=============

\bibliographystyle{alphaurl}
\bibliography{ELA}

\end{document}